\documentclass{amsart}

\usepackage{amssymb}
\usepackage[margin=1in]{geometry}
\usepackage{hyperref}
\usepackage{amsmath,amscd}
\usepackage[enableskew]{youngtab}
\usepackage[notabloids]{ytableau} 
\usepackage[all,cmtip]{xy}
\usepackage{shuffle}
\usepackage{color}
\usepackage{graphicx}

\newtheorem{theorem}{Theorem}[subsection]
\newtheorem{lemma}[theorem]{Lemma}
\newtheorem{proposition}[theorem]{Proposition}
\newtheorem{corollary}[theorem]{Corollary}

\theoremstyle{definition}
\newtheorem{definition}[theorem]{Definition}

\theoremstyle{remark}
\newtheorem{remark}[theorem]{Remark}

\numberwithin{equation}{section}
\def\ZZ{\mathbb{Z}}  \def\CC{\mathbb{C}} \def\FF{\mathbb{F}} \def\KK{\mathbb K}
\def\SS{\mathfrak{S}} \def\P{\mathcal{P}} \def\C{\mathcal{C}} 
  
\def\B{\mathfrak{S}^B} \def\D{\mathfrak{S}^D} \def\HB{H^B} \def\HD{H^D} 
\def\chiB{\chi^B} \def\chiD{\chi^D} \def\PB{\mathcal{P}^B} \def\CB{\mathcal{C}^B}
\def\PD{\mathcal{P}^D} \def\CD{\mathcal{C}^D}
\def\bx{{\mathbf x}} \def\by{{\mathbf y}} \def\bs{{\mathbf s}}
\def\bX{\mathbf{X}} \def\bY{\mathbf{Y}}

\def\Sym{\mathrm{Sym}} \def\QSym{\mathrm{QSym}}
\def\FQSym{\mathbf{FQSym}}  \def\NSym{\mathbf{NSym}} 
\def\bF{\mathbf{F}} \def\FB{F^B} \def\MB{M^B}  \def\FD{F^D} \def\MD{M^D}
\def\bFB{\mathbf{F}^B} \def\bFD{\mathbf{F}^D}
\def\bsB{\mathbf{s}^B} \def\bhB{\mathbf{h}^B}
\def\bsD{\mathbf{s}^D} \def\bhD{\mathbf{h}^D}
\def\bh{{\mathbf h}} \def\sB{s^B} \def\hB{h^B} \def\mB{m^B}
\def\sD{s^D} \def\hD{h^D}
\def\Ch{\mathrm{Ch}} \def\bch{\mathbf{ch}}
\def\ChB{\mathrm{Ch}^B}  
\def\bchB{\mathbf{ch}^B} 

 \def\inv{\mathrm{inv}}  
\def\neg{\mathrm{neg}} \def\Neg{\mathrm{Neg}} \def\nsp{\mathrm{nsp}} \def\Nsp{\mathrm{Nsp}}
\def\cleq{{\preccurlyeq}} 
\def\qand{\quad\text{and}\quad}

\def\Hom{\operatorname{Hom}} \def\pib{\overline{\pi}}
\DeclareMathOperator{\supp}{supp}
\DeclareMathOperator{\des}{des}

\def\st{\operatorname{st}\,} \def\stB{\operatorname{{st^B}}\,} 
\def\Dst{\operatorname{{^Dst}}\,} \def\stD{\operatorname{{st^D}}\,}
\def\modelsB{\operatorname{\models^B}} \def\modelsD{\operatorname{\models^D}}

\def\unshuffle{%
               \mathrel{\raisebox{.04em}{%
               \reflectbox{\rotatebox[origin=c]{180}{${\shuffle}$}}}}}
\def\shuffleB{\,\operatorname{\shuffle^B}\,} \def\unshuffleB{\,\operatorname{\unshuffle^B}\,} 
\def\shuffleD{\,\operatorname{\shuffle^D}\,} \def\unshuffleD{\,\operatorname{\unshuffle^D}\,}
\def\CupB{\,\operatorname{\Cup^B}\,}\def\CapB{\,\operatorname{\Cap^B}\,}
\def\CupD{\,\operatorname{\Cup^D}\,}\def\CapD{\,\operatorname{\Cap^D}\,}
\def\odotB{\,\odot^B\,} \def\DeltaB{\Delta^B} \def\odotD{\,\odot^D\,} \def\DeltaD{\Delta^D}

\def\DiamondFQSym #1{ 
\xymatrix @R=16pt @C=7pt {
 & \FQSym^{#1} \ar@{->>}[rd]^{\chi^{#1}} \\
 \NSym^{#1} \ar@{^(->}[ru]^{\imath} \ar@{->>}[rd]_{\chi^{#1}} & & \QSym^{#1} \ar@{<-->}[ll]^{\txt{\small dual}} \\
 & \Sym^{#1} \ar@{^(->}[ru]_\imath } }

\def\DiamondCox #1 {
\xymatrix @R=16pt @C=15pt {
 & \ZZ{#1} \ar@{->>}[rd]^{\chi'} \\
 \Sigma({#1}) \ar@{^(->}[ru]^{\imath} \ar@{->>}[rd]_{\chi'} & & \Sigma^*({#1}) \ar@{<-->}[ll]^{\txt{\small dual}} \\
 & \Lambda({#1}) \ar@{^(->}[ru]_\imath } }
  
\begin{document}

\title{A uniform generalization of some combinatorial Hopf algebras}
\author{Jia Huang}
\address{Department of Mathematics and Statistics, University of Nebraska at Kearney, Kearney, NE 68849, USA}
\email{huangj2@unk.edu}
\thanks{The author is grateful to Victor Reiner for his kind support and valuable suggestions.}

\maketitle

\begin{abstract}
We generalize the Hopf algebras of free quasisymmetric functions, quasisymmetric functions, noncommutative symmetric functions, and symmetric functions to certain representations of the category of finite Coxeter systems and its dual category. We investigate their connections with the representation theory of 0-Hecke algebras of finite Coxeter systems. Restricted to type B and D we obtain dual graded modules and comodules over the corresponding Hopf algebras in type A. 
\end{abstract}

\section{Introduction}

The self-dual graded Hopf algebra $\Sym$ of symmetric functions plays a significant role in algebraic combinatorics. Analogues of $\Sym$ include the graded Hopf algebra $\QSym$ of quasisymmetric functions and its dual graded Hopf algebra $\NSym$ of noncommutative symmetric functions, as well as the self-dual graded Hopf algebra $\FQSym$ of free quasisymmetric functions. The relation among these Hopf algebras is illustrated by the following commutative diagram. 
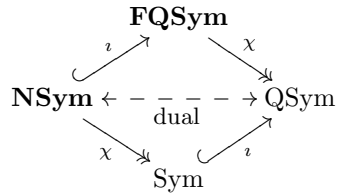
\begin{equation}\label{eq:DiamondFQSymA'} \DiamondFQSym{} \end{equation}
See Grinberg and Reiner~\cite{GrinbergReiner} for details.

On the other hand, there is a graded Hopf algebra structure on $\ZZ\SS$, the direct sum of group algebras of the symmetric groups $\SS_n$ for all $n\ge0$. This is the \emph{Malvenuto--Reutenauer Hopf algebra}~\cite{MR}. One has a Hopf algebra isomorphism between $\FQSym$ and $\ZZ\SS$, sending $\NSym$ onto the descent algebra $\Sigma(\SS)$ of type A. There is a surjection $\chi'$ from $\ZZ\SS$ to the dual $\Sigma^*(\SS)$ of $\Sigma(\SS)$, and we denote $\Lambda(\SS):=\chi'(\Sigma(\SS))$. Then one has the following commutative diagram of graded Hopf algebras isomorphic to \eqref{eq:DiamondFQSymA'}.
\begin{equation}\label{eq:DiamondCoxA'}
\DiamondCox{\SS}
\end{equation}

There are interesting connections between these Hopf algebras and representation theory. The classic Frobenius correspondence provides a Hopf algebra isomorphism from the Grothendieck group $G_0(\CC\SS_\bullet)$ of the category of finitely generated (complex) representations of the symmetric groups $\SS_n$ to the Hopf algebra $\Sym$.  Analogously to this correspondence, Krob and Thibon~\cite{KrobThibon} introduced two characteristic maps
\begin{equation}\label{eq:ChA}
\Ch: G_0(H_\bullet(0))\xrightarrow\sim \QSym \qand \bch:K_0(H_\bullet(0))\xrightarrow\sim \NSym
\end{equation}
giving Hopf algebra isomorphisms from the Grothendieck groups $G_0(H_\bullet(0))$ and $K_0(H_\bullet(0))$ of the categories of finitely generated (projective) representations of the \emph{0-Hecke algebras} $H_n(0)$ of type A to the dual Hopf algebras $\QSym$ and $\NSym$. Here the 0-Hecke algebras are certain deformations of the group algebras of finite Coxeter groups, whose representation theory were studied by Norton~\cite{Norton}.

In this paper we generalize the previously mentioned results from type A to finite Coxeter systems.
An \emph{abstract finite Coxeter system} is a pair $(W,S)$ where $W$ is an \emph{abstract group} with a Coxeter presentation $W=\langle\, S \mid (st)^{m_{st}}=1,\ \forall s,t\in S\,\rangle$.
Each isomorphism class of finite Coxeter systems contains a unique abstract finite Coxeter system.
We define a category $\mathcal Cox$ whose objects are abstract finite Coxeter systems and whose morphisms are embeddings of abstract finite Coxeter systems.
Also let $\mathcal Cox^{op}$ be the opposite/dual category of $\mathcal Cox$. 
We define a representation $\Omega$ ($\Sigma$ resp.) of both $\mathcal Cox$ and $\mathcal Cox^{op}$ by sending abstract finite Coxeter systems to their group algebras (descent algebras resp.) and sending embeddings of finite Coxeter systems to certain linear maps which generalize the product and coproduct of the Malvenuto--Reutenauer Hopf algebra (descent algebra of type A resp.). The inclusion of the descent algebras into the group algebras induces a natural transformation $\imath:\Sigma\hookrightarrow\Omega$. On the other hand, we have a representation $\Sigma^*$ dual to $\Sigma$ and a natural transformation $\chi':\Omega\twoheadrightarrow\Sigma^*$ dual to $\imath$. Furthermore, applying $\chi'\circ\imath$ to $\Sigma$ gives a representation $\Lambda$ of both $\mathcal Cox$ and $\mathcal Cox^{op}$ such that the diagram below is commutative.
\begin{equation}\label{eq:DiamondCoxCat'}
\xymatrix @R=15pt @C=25pt {
 & \Omega \ar@{->>}[rd]^{\chi'} \\
 \Sigma \ar@{^(->}[ru]^{\imath} \ar@{->>}[rd]_{\chi'} & & \Sigma^* \ar@{<-->}[ll]^{\txt{\small dual}} \\
 & \Lambda \ar@{^(->}[ru]_\imath } 
 \end{equation}
Restricting this diagram to type A recovers the commutative diagram \eqref{eq:DiamondCoxA'} of Hopf algebras.

Next, using the $P$-partition theory for finite Coxeter groups by Reiner~\cite{Reiner1}, we generalize $\FQSym$ to a  representation $\mathcal FQSym$ of both $\mathcal Cox$ and $\mathcal Cox^{op}$. There is a natural isomorphism between $\Omega$ and $\mathcal FQSym$. Applying it to \eqref{eq:DiamondCoxCat'} gives the following commutative diagram, which is in natural isomorphism with \eqref{eq:DiamondCoxCat'}.
\begin{equation}\label{eq:DiamondFQSymCat'}
\xymatrix @R=15pt @C=20pt {
 & \mathcal FQSym \ar@{->>}[rd]^{\chi} \\
 \mathcal NSym \ar@{^(->}[ru]^{\imath} \ar@{->>}[rd]_{\chi} & & \mathcal QSym \ar@{<-->}[ll]^{\txt{\small dual}} \\
 & \mathcal Sym \ar@{^(->}[ru]_\imath }
\end{equation}
Restricting this diagram to type A recovers the commutative diagram \eqref{eq:DiamondFQSymA'} of Hopf algebras.

We also obtain representations $\mathcal G_0$ and $\mathcal K_0$ of both $\mathcal Cox$ and $\mathcal Cox^{op}$ from the representation theory of 0-Hecke algebras of finite Coxeter groups, together with natural isomorphisms 
\[ \Ch: \mathcal G_0\to\mathcal QSym \qand \bch:\mathcal K_0\to\mathcal NSym.\]
Restricting $\Ch$ and $\bch$ to type A recovers the characteristic maps in \eqref{eq:ChA}.

Our generalization not only recovers known results on Hopf algebras in type A, but also leads to new results in type B and type D on certain modules and comodules over the corresponding Hopf algebras in type A. 
We summarize our results in type B below; the results in type D are similar. In fact, restricting \eqref{eq:DiamondCoxCat'} and \eqref{eq:DiamondFQSymCat'} to type B gives the following commutative diagrams.
\[ \DiamondCox{\B} \qquad\qquad \DiamondFQSym{B} \]
Here each entry is a graded right module and comodule over the corresponding type A Hopf algebra. 
We also have characteristic maps analogous to \eqref{eq:ChA}, providing isomorphisms of graded modules and comodules:
\[ \ChB: G_0(\HB_\bullet(0)) \xrightarrow\sim \QSym^B \qand \bchB: K_0(\HB_\bullet(0)) \xrightarrow\sim \NSym^B. \]
In our earlier work~\cite{H0Tab} we obtained partial results on these characteristic maps using a tableau approach to the representation theory of 0-Hecke algebras; now we are able to get more complete results on them.


For each fixed finite Coxeter system $(W,S)$, Aguiar and Mahajan \cite[Theorem~5.7.1]{AguiarMahajan} obtained a commutative diagram of vector spaces over a field of characteristic zero by a very interesting approach different from ours.
Their result agrees with what we have when applying our diagrams~\eqref{eq:DiamondCoxCat'} and \eqref{eq:DiamondFQSymCat'} to the fixed finite Coxeter system $(W,S)$.
This allows us to obtain free $\ZZ$-bases for $\Lambda(W,S)\cong \mathcal Sym(W,S)$. 
See~Section~\ref{sec:AguiarMahajan}.
We do not have a representation theoretic interpretation for $\mathcal Sym(W,S)$ except in type A.

It is mentioned in \cite[\S5.1.2]{AguiarMahajan} that the construction of Hopf algebra structures is special to type A.
In this paper we use embeddings and restrictions of Coxeter systems to realize \eqref{eq:DiamondCoxCat'} and \eqref{eq:DiamondFQSymCat'} as commutative diagrams of representations of the category $\mathcal Cox$ and its opposite category $\mathcal Cox^{op}$.
It is natural from our perspective that we do not get Hopf algebras in type B and D, as parabolic subsystems in type A are disjoint unions of Coxeter systems of the same type but this is not the case in type B and D.

\tableofcontents

\section{Uniform Results}\label{sec:General}

In this section we provide our main results for finite Coxeter systems. 

\subsection{Group algebras and descent algebras of finite Coxeter groups}\label{sec:Coxeter}  

A \emph{finite Coxeter group} is a finite group $W$ with a \emph{Coxeter presentation}
\[ W := \langle\, S \mid (st)^{m_{st}} = 1,\ \forall s,t\in S\,\rangle \]
where $S$ is a finite generating set, $m_{ss} = 1$ for all $s\in S$, and $m_{st}=m_{ts} \in\{2,3,\ldots\}$ for all $s,t\in S$.
The relations for $W$ are equivalent to the \emph{quadratic relations} $s^2=1$ for all $s\in S$ and the \emph{braid relations}
$(sts\cdots)_{m_{st}}=(tst\cdots)_{m_{st}}$ for all $s,t\in S$, where $(aba\cdots)_m$ denotes an alternating product of $m$ terms. 

The pair $(W,S)$ is called a \emph{finite Coxeter system}, which is encoded by an edge-labeled graph called the \emph{Coxeter diagram} of $(W,S)$.
The vertex set of the Coxeter diagram of $(W,S)$ is $S$ and there is an edge labeled with $m_{st}$ connecting two vertices $s$ and $t$ if $m_{st}\ge3$.
When $m_{st}\in\{3,4,5\}$ the labeled edge between $s$ and $t$ is often drawn as $m_{st}-2$ multiple edges.

If the Coxeter diagram of $(W,S)$ is connected then $(W,S)$ is \emph{irreducible}.
There is a well-known classification of finite irreducible Coxeter systems.
In general if $S_1,\ldots,S_k$ are the vertex sets of the connected components of the Coxeter diagram of $(W,S)$ then $(W_{S_i},S_i)$ is an irreducible Coxeter system for each $i\in[k]:=\{1,\ldots,k\}$ and $W\cong W_{S_1}\times\cdots\times W_{S_k}$.



Every $w\in W$ can be written as $w=s_1\cdots s_k$ where $s_1,\ldots,s_k\in S$; if $k$ is minimum then such an expression is called a \emph{reduced expression of $w$} and $\ell(w):=k$ is the \emph{length} of $w$. A generator $s\in S$ is a \emph{descent} of $w$ if $\ell(ws)<\ell(w)$, or equivalently, if some reduced expression of $w$ ends with $s$. The \emph{descent set} of $w$, denoted by $D(w)$, consists of all descents of $w$. 

Let $I\subseteq S$ and denote $I^c:=S\setminus I$. The \emph{parabolic subgroup} $W_I$ of $W$ is generated by $I$. It is also a finite Coxeter group whose elements have the same lengths and descents as in $W$. 
A left or right $W_I$-coset in $W$ has a unique representative of minimal length.
The length-minimal representatives for the left and right $W_I$-cosets in $W$ are given respectively by 
\[ W^I := \{w\in W:D(w)\subseteq I^c\} \qand ^IW := \{w\in W:D(w^{-1})\subseteq I^c\}.\]

\begin{proposition}[{\cite[Proposition~2.4.4]{BjornerBrenti},~\cite[Lemma~9.7]{Lusztig}}]\label{prop:parabolic}
Every element $w\in W$ can be written uniquely as $w=w^I\cdot{_Iw}$ such that $w^I\in W^I$ and ${_Iw}\in W_I$, and this expression implies $\ell(w)=\ell(w^I)+\ell(_Iw)$. Similarly, every element $w\in W$ can be written uniquely as $w=w_I\cdot{^Iw}$ such that $w_I\in W_I$ and $^Iw\in\,^IW$, and this expression implies $\ell(w)=\ell(w_I)+\ell(^Iw)$.
\end{proposition}

The group algebra $\ZZ W$ is self-dual under the positive definite bilinear form defined by $\langle u,v\rangle := \delta_{u,v}$ for all $u,v\in W$. For convenience of notation we identify a subset of $W$ with the sum of its elements in $\ZZ W$. 
For any group $G$ we denote by $(\ )^{-1}$ the $\ZZ$-linear map from the group algebra $\ZZ G$ to itself defined by sending every element of $G$ to its inverse.

\begin{definition}\label{def:MapsZW}
Given $I\subseteq S$, we define the following four linear maps:
\[\begin{matrix} 
\mu_I^S:& \ZZ W_I & \to & \ZZ W, \\
 & u & \mapsto & W^I\cdot u,
\end{matrix}\qquad\qquad
\begin{matrix} 
\rho_I^S: & \ZZ W & \to & \ZZ W_I, \\
 & w & \mapsto & w_I,
\end{matrix} \]
\[ \begin{matrix} 
\bar\mu_I^S: & \ZZ W_I & \to & \ZZ W, \\
 & u & \mapsto & u\cdot\, ^IW,
\end{matrix} \qquad\qquad
\begin{matrix} 
\bar\rho_I^S : & \ZZ W & \to & \ZZ W_I, \\
 & w & \mapsto & _Iw.\end{matrix}\]
\end{definition}

\begin{proposition}\label{prop:MR}
The following diagram is commutative and rotating it $180^\circ$ gives a dual diagram, i.e., $\mu_I^S$ is dual to $\bar\rho_I^S$, $\rho_I^S$ is dual to $\bar\mu_I^S$, and $(\ )^{-1}$ is self-dual.
\[ \xymatrix @R=20pt @C=25pt {
\ZZ W_I \ar[r]^{\mu_I^S} \ar@{<->}[d]^{\wr}_{(\ )^{-1}} & \ZZ W \ar@{<->}[d]^{\wr}_{(\ )^{-1}}  \ar[r]^{\rho_I^S} & \ZZ W_I \ar@{<->}[d]^{\wr}_{(\ )^{-1}}  \\
\ZZ W_I \ar[r]_{\bar\mu_I^S} & \ZZ W \ar[r]_{\bar\rho_I^S} & \ZZ W_I 
}\]
\end{proposition}
\begin{proof}
The result follows from Proposition~\ref{prop:parabolic}.
 \end{proof}

\begin{lemma}\label{lem:DesRes}
If $I\subseteq S$ and $w\in W$ then $D(w)\cap I = D(\,_Iw)$.
\end{lemma}
\begin{proof}
Assume $s\in I$. Then $\ell(w)=\ell(w^I)+\ell(\,_Iw)$ and $\ell(ws) = \ell(w^I)+\ell(\,_Iws)$ by Proposition~\ref{prop:parabolic}. Hence $\ell(ws)<\ell(w)$ if and only if $\ell(\,_Iws)<\ell(\,_Iw)$. The result follows.
 \end{proof}

\begin{proposition}\label{prop:IJS}
If $I\subseteq J\subseteq S$ then $\mu_J^S\circ \mu_I^J = \mu_I^S$, $\bar\mu_J^S\circ \mu_I^J = \bar\mu_I^S$, $\rho_I^J\circ\rho_J^S=\rho_I^S$, and $\bar\rho_I^J\circ\bar\rho_J^S=\bar\rho_I^S$.
\end{proposition}
\begin{proof}
Assume $I\subseteq J\subseteq S$. Every element $z\in W$ can be written uniquely as $z=z^J \cdot\,_{\raisebox{-1pt}{$\scriptstyle J$}}z$. Lemma~\ref{lem:DesRes} implies $D(\,_{\raisebox{-1pt}{$\scriptstyle J$}}z) = D(z)\cap J$. Hence $z\in W^I$ if and only if ${\raisebox{-3pt}{$\scriptstyle J$}}z\in W_J^I$. Then for any $u\in W_I$ one has 
\[\mu_J^S(\mu_I^J(u)) = W^J\cdot W_J^I\cdot u = W^I\cdot u = \mu_I^S(u).\]
Hence $\mu_J^S\circ\mu_I^J=\mu_I^S$. By Proposition~\ref{prop:MR}, applying the maps $(\ )^{-1}$ to this gives $\bar\mu_J^S\circ\bar\mu_I^J=\bar\mu_I^S$, and then taking the duals gives $\rho_I^J\circ\rho_J^S=\rho_I^S$ and $\bar\rho_I^J\circ\bar\rho_J^S=\bar\rho_I^S$.
\end{proof}

\begin{definition}\label{def:Cox}
An \emph{abstract finite Coxeter system} is a pair $(W,S)$ where $W$ is an \emph{abstract group} with a Coxeter presentation $W=\langle\, S \mid (st)^{m_{st}}=1,\ \forall s,t\in S\,\rangle$.
We identify a finite Coxeter system with the unique abstract finite Coxeter system isomorphic to it.
We define a category $\mathcal Cox$ whose objects are abstract finite Coxeter systems and whose morphisms are embeddings of abstract finite Coxeter systems.
Given such an embedding $(W',S')\hookrightarrow (W,S)$, there exists a unique $I\subseteq S$ such that $(W',S')\cong (W_I,I)$, and thus we can identify this embedding with the inclusion $(W_I,I)\hookrightarrow(W,S)$.
We also denote by $\mathcal Cox^{op}$ the opposite/dual category of $\mathcal Cox$ whose morphisms are restrictions $(W,S)\twoheadrightarrow (W_I,I)$ with $I\subseteq S$.
\end{definition}

A \emph{representation} of a category $\mathcal C$ is a covariant functor $R$ from $\mathcal C$ to the category of $\ZZ$-modules. 
The \emph{dual representation} $R^*$ of $R$ is a representation of the dual category $\mathcal C^{op}$ which sends each object $O$ to the dual $\ZZ$-module $R(O)^*$ and sends each morphism $f^*:O_2\to O_1$ dual to $f:O_1\to O_2$ to the morphism $R(f)^*: R(O_2)^*\to R(O_1)^*$ dual to $R(f): R(O_1)\to R(O_2)$.
See commutative diagrams below.
\[ \xymatrix @R=20pt @C=40pt {
O_1 \ar[d]^R \ar^f[r] & O_2 \ar[d]^R \\
R(O_1) \ar^{R(f)}[r] & R(O_2) } \qquad\qquad
\xymatrix @R=20pt @C=40pt {
O_1 \ar[d]^{R^*} & O_2 \ar[l]_{f^*} \ar[d]^{R^*} \\
R(O_1)^* & R(O_2)^* \ar[l]_{R(f)^*} 
} \]

\begin{definition}
By Proposition~\ref{prop:IJS}, we have a representation $\Omega$ of $\mathcal Cox$ sending an abstract finite Coxeter system $(W,S)$ to $\ZZ W$ and sending an inclusion $(W_I,I)\hookrightarrow (W,S)$ to $\mu_I^S$ whenever $I\subseteq S$. 
By abuse of notation, also let $\Omega$ be the representation of $\mathcal Cox^{op}$ sending an abstract finite Coxeter system $(W,S)$ to $\ZZ W$ and sending a restriction $(W,S)\twoheadrightarrow (W_I,I)$ to $\rho_I^S$ whenever $I\subseteq S$. By Proposition~\ref{prop:MR}, replacing $\mu_I^S$ and $\rho_I^S$ with $\bar\mu_I^S$ and $\bar\rho_I^S$ gives the dual representation $\Omega^*$ of $\Omega$, and $w\mapsto w^{-1}$ induces a natural isomorphism between $\Omega$ and $\Omega^*$ denoted by $(\ )^{-1}$.
\end{definition}

Given $I\subseteq S$, the \emph{descent class} of $I$ in $W$ is $D_I(W):=\{w\in W:D(w)=I\}$. 

\begin{proposition}\label{prop:nonempty}
For every $I\subseteq S$ the descent class $D_I(W)$ is nonempty and becomes an interval $[w_0(I),w_1(I)]$ under the left weak order of $W$, where $w_0(I)$ and $w_1(I)$ are the longest elements in $W_{I}$ and $W^{I^c}$, respectively.
\end{proposition}

\begin{proof}
By Lusztig~\cite[Lemma~9.8]{Lusztig}, the parabolic subgroup $W_I$ has a unique longest element $w_0(I)$ satisfying $w_0(I)=w_0(I)^{-1}$ and $D(w_0(I))=I$.
Hence $D_I(W)\ne\emptyset$. 
The rest of the result follows from Bj\"orner and Wachs~\cite[Theorem~6.2]{BjornerWachs}.
\end{proof}

We identity $D_I(W)$ with the sum of its elements in the group algebra $\ZZ W$. 
Let $\Sigma(W)$ be the $\ZZ$-span $\{D_I(W):I\subseteq S\}$.
Proposition~\ref{prop:nonempty} implies that this spanning set is indeed a free $\ZZ$-basis for $\Sigma(W)$.
One has an injection $\imath:\Sigma(W)\hookrightarrow \ZZ W$ by inclusion. Solomon~\cite{Solomon} showed that $\Sigma(W)$ is a subalgebra of the group algebra $\ZZ W$, called \emph{the descent algebra of $W$}. We will not use this algebra structure in this paper, but instead we will restrict the linear maps $\mu_I^S$ and $\rho_I^S$ to $\Sigma(W_I)$ and $\Sigma(W)$.

On the other hand, let $\Sigma^*(W)$ be the dual $\ZZ$-module of $\Sigma(W)$ with a free $\ZZ$-basis $\{D^*_I(W):I\subseteq S\}$ dual to $\{D_I(W): I\subseteq S\}$. 
If $w\in W$ then write $D^*_w(W):=D^*_I(W)$ where $I=D(w)$. 
Dual to the injection $\imath:\Sigma(W)\hookrightarrow \ZZ W$ is a surjection 
\[ \begin{matrix}
\chi: & \ZZ W & \twoheadrightarrow & \Sigma^*(W) \\
& w & \mapsto & D^*_w(W).
\end{matrix}\] 
 
\begin{proposition}\label{prop:IndDes}
Let $I\subseteq S$. Then $\mu_I^S: \ZZ W_I\to \ZZ W$ restricts to $\mu_I^S: \Sigma(W_I)\to\Sigma(W)$ and $\bar\rho_I^S: \ZZ W\to\ZZ W_I$ descends to 
\[ \begin{matrix}
\bar\rho_I^S: & \Sigma^*(W) & \to & \Sigma^*(W_I), \\
& D^*_w(W) & \mapsto & D^*_{_Iw}(W_I).
 \end{matrix} \]
If $J\subseteq I$ and $K\subseteq S$ then
\[ \mu_I^S(D_J(W_I)) = \sum_{J'\cap I=J} D_{J'}(W) \qand \bar\rho_I^S(D^*_K(W)) = D^*_{K\cap I}(W_I). \]
\end{proposition}

\begin{proof}
The result follows easily from Proposition~\ref{prop:parabolic} and Lemma~\ref{lem:DesRes}.
 \end{proof}

To restrict $\rho_I^S$ to descent algebras we need to first develop some lemmas.

\begin{lemma}\label{lem:Des}
Suppose that $I\subseteq S$, $z\in W^I$, $s\in D(z^{-1})^c$, and $sz\notin W^I$. Then $D(sz)\cap I=\{z^{-1}sz\}$.\end{lemma}

\begin{proof}
One has $D(sz)\cap I\ne \emptyset$ since $sz\notin W^I$. Let $r\in D(sz)\cap I$. Then $\ell(szr)<\ell(sz)$ and $\ell(zr)>\ell(z)$. One also has $\ell(sz)>\ell(z)$ since $s\in D(z^{-1})^c$. By Lusztig~\cite[Proposition 1.10]{Lusztig} (applied to $w = z$ and $t = r$), one has $sz=zr$, i.e., $r=z^{-1}sz$.
 \end{proof}

\begin{lemma}\label{lem:DesIJ}
Let $I, K\subseteq S$, $u\in W_I$, and $z\in{^IW}$. Then $D(uz)=K$ if and only if (i)--(iii) all hold.

(i) If $s\in D(z)$ then $s\in K$, i.e. $D(z)\subseteq K$.

(ii) If $s\in D(z)^c$ and $D(sz^{-1})\subseteq I^c$ then $s\in K^c$.

(iii) If $s\in D(z)^c$ and $D(sz^{-1})\not\subseteq I^c$ then $s\in K \Leftrightarrow zsz^{-1}\in D(u)$.

\noindent Moreover, condition (iii) is equivalent to:

(iii') $L(z,I,K)\subseteq D(u)\subseteq L'(z,I,K)$, where $L(z,I,K)$ and $L'(z,I,K)$ are two subsets of $I$ defined as
\[ L(z,I,K):=\bigcup_{s\in D(z)^c\cap K} \left(D(sz^{-1})\cap I\right) \qand
 L'(z,I,K):=\bigcap_{s\in D(z)^c\setminus K}\left( I\setminus D(sz^{-1})\right).\]
Consequently, $z\in{^ID_K}:=\left\{{^I}w:w\in D_K(W)\right\}$ if and only if (i), (ii), and $L(z,I,K)\subseteq L'(z,I,K)$.
 \end{lemma}

\begin{proof}
If $s\in D(z)$ then $\ell(uzs)\leq \ell(u)+\ell(zs)<\ell(u)+\ell(z)=\ell(uz)$. Hence $D(uz)=K$ implies (i).

If $s\in D(z)^c$ and $D(sz^{-1})\subseteq I^c$ then $zs\in\,^IW$ and thus $\ell(uzs)=\ell(u)+\ell(zs)>\ell(uz)$. Hence $D(uz)=K$ implies (ii).

If $s\in D(z)^c$ and $D(sz^{-1})\not\subseteq I^c$, then Lemma~\ref{lem:Des} implies that $D(sz^{-1})\cap I=\{r\}$ where $r=zsz^{-1}$. Since 
\[ \ell(uzs)=\ell(urz)=\ell(ur)+\ell(z) \qand
 \ell(uz)=\ell(u)+\ell(z)\]
one has  $s\in D(uz)$ if and only if $r\in D(u)$. Hence $D(uz)=K$ implies (iii).
    
Conversely, one sees that (i)--(iii) imply $s\in D(uz)\Leftrightarrow s\in K$ in the above three cases. Thus $D(uz)=K$ is equivalent to (i)--(iii).

Next, assume (iii') holds. Let $s\in D(z)^c$ such that $D(sz^{-1})\not\subseteq I^c$. Then $D(sz^{-1})\cap I = \{ r = zsz^{-1} \}$ by Lemma~\ref{lem:Des}. If $s\in K$ then $r\in L(z,I,K)\subseteq D(u)$. If $s\in K^c$ then $r\in I\setminus L'(z,I,K)\subseteq I\setminus D(u)$. Thus (iii) holds.

Conversely, assume (iii) holds and let $s\in D(z)^c$. If there exists $r\in D(sz^{-1})\cap I$ then $r=zsz^{-1}$ by Lemma~\ref{lem:Des}, and one has $r\in D(u) \Leftrightarrow s\in K$. This implies $L(z,I,K)\subseteq D(u)$ and $I\setminus L'(z,I,K)\subseteq I\setminus D(u)$. Thus (iii') holds.

Finally, $z\in{^ID_K}$ if and only if there exists an element $u\in W_I$ such that $D(uz)=K$. 
We know $D(uz)=K$ is equivalent to (i), (ii), and (iii').
By Proposition~\ref{prop:nonempty}, any descent class is nonempty.
Thus the existence of $u$ is equivalent to (i), (ii), and $L(z,I,K)\subseteq L'(z,I,K)$.
 \end{proof}



\begin{proposition}\label{prop:ResDes}
Let $I\subseteq S$. Then $\rho_I^S: \ZZ W\to\ZZ W_I$ restricts to $\rho_I^S: \Sigma(W)\to\Sigma(W_I)$ such that
\[ \rho_I^S(D_K(W)) = \sum_{z\in\,{^{I}D_K}} \ \sum_{ L(z,I,K)\subseteq K' \subseteq L'(z,I,K) } D_{K'}(W_I), \quad \forall K\subseteq S \]
and $\bar\mu_I^S: \ZZ W_I\to\ZZ W$ descends to the following map dual to $\rho_I^S: \Sigma(W)\to\Sigma(W_I)$:
\[ \begin{matrix}
\bar\mu_I^S: & \Sigma^*(W_I) & \to & \Sigma^*(W), \\
 & D^*_u(W_I) & \mapsto & \sum_{z\in\,^IW} D^*_{uz}(W).
 \end{matrix}\]
 
\end{proposition}

\begin{proof}
Let $w\in W$ with $w_I=u$ and $^Iw=z$. By Lemma~\ref{lem:DesIJ}, one has $D(w)=K$ if and only if $z\in{^I}D_K$ and $L(z,I,K) \subseteq D(u)\subseteq L'(z,I,K)$. Thus $\rho_I^S: \ZZ W\to\ZZ W_I$ restricts to $\rho_I^S: \Sigma(W)\to\Sigma(W_I)$ and the desired formula for $\rho_I^S(D_K(W))$ holds.

Next, fix $z\in\,^IW$. Lemma~\ref{lem:DesIJ} implies that $D(uz)$ depends only on $z$ and $D(u)$ for any $u\in W_I$. Hence $\bar\mu_I^S: \ZZ W_I\to\ZZ W$ descends to $\bar\mu_I^S:\Sigma^*(W_I)\to\Sigma^*(W)$.

Finally, if $K\subseteq S$ and $u\in W_I$ then  
\[ \langle D_K(W),\, \bar\mu_I^S(D^*_u(W_I)) \rangle = \# \left\{ z\in\,^IW: D(uz)=K \right\} \qand \]
\[ \langle \rho_I^S(D_K(W)), D^*_u(W_I) \rangle = \# \left\{ z\in\,^I D_K:L(z,I,K)\subseteq D(u) \subseteq L'(z,I,K)\right\}.\]
The two sets on the right hand side are the same by Lemma~\ref{lem:DesIJ}. The duality follows.
 \end{proof}

\begin{theorem}\label{thm:FullDiagram}
The diagram below is commutative and rotating it $180^\circ$ gives a dual diagram.
\[ \xymatrix @R=20pt @C=25pt{
\Sigma(W_I) \ar@{^(->}[r]^{\imath} \ar[d]^{\mu_I^S} & \ZZ W_I \ar@{<->}[r]^{(\ )^{-1}} \ar[d]^{\mu_I^S} & \ZZ W_I \ar@{->>}[r]^-{\chi} \ar[d]^{\bar\mu_I^S} & \Sigma^*(W_I) \ar[d]^{\bar\mu_I^S} \\
\Sigma(W) \ar@{^(->}[r]^{\imath} \ar[d]^{\rho_I^S} & \ZZ W \ar@{<->}[r]^{(\ )^{-1}} \ar[d]^{\rho_I^S} & \ZZ W \ar@{->>}[r]^-{\chi} \ar[d]^{\bar\rho_I^S} & \Sigma^*(W) \ar[d]^{\bar\rho_I^S} \\
\Sigma(W_I) \ar@{^(->}[r]^{\imath} & \ZZ W_I \ar@{<->}[r]^{(\ )^{-1}}  & \ZZ W_I \ar@{->>}[r]^-{\chi} & \Sigma^*(W_I) 
}\]
\end{theorem}
\begin{proof}
Combine the previous results obtained in this subsection.
 \end{proof}

\begin{definition}
By Theorem~\ref{thm:FullDiagram}, one has a representation $\Sigma$ of both the category $\mathcal Cox$ and its dual category $\mathcal Cox^{op}$ by sending each abstract finite Coxeter system $(W,S)$ to $\Sigma(W)$ and sending an inclusion $(W_I,I)\hookrightarrow(W,S)$ to $\mu_I^S:\Sigma(W_I)\to\Sigma(W)$ and a restriction $(W,S)\twoheadrightarrow(W_I,I)$ to $\rho_I^S:\Sigma(W)\to\Sigma(W_I)$ whenever $I\subseteq S$. 
By Theorem~\ref{thm:FullDiagram}, replacing $\Sigma(W)$, $\mu_I^S$, and $\rho_I^S$ with $\Sigma^*(W)$, $\bar\mu_I^S$, and $\bar\rho_I^S$ gives the dual representation $\Sigma^*$ of $\Sigma$. 
Theorem~\ref{thm:FullDiagram} also gives dual natural transformations $\imath:\Sigma\hookrightarrow\Omega$ and $\chi:\Omega^*\twoheadrightarrow\Sigma^*$, as well as a natural transformation $\chi':\Omega\twoheadrightarrow \Sigma^*$ induced by 
\[ \xymatrix@C=25pt{\chi': \ZZ W \ar[r]^-{(\ )^{-1}} & \ZZ W \ar@{->>}[r]^-{\chi} & \Sigma^*(W)}.\]
\end{definition}

Lastly, let $\Lambda(W):=\chi'(\Sigma(W))$. Then $\Lambda(W)$ is spanned by 
\[ \Lambda_I(W):= \chi'(D_I(W)) = \sum_{w\in D_I(W)} D^*_{w^{-1}}(W),\qquad\forall I\subseteq S.\]
Bases for $\Lambda(W)$ will be provided in Section~\ref{sec:AguiarMahajan}.
We have the following commutative diagram of $\ZZ$-modules.
\begin{equation}\label{eq:DiamondZW}
\xymatrix @R=16pt @C=7pt {
 & \ZZ W \ar@{->>}[rd]^{\chi'} \\
 \Sigma(W) \ar@{^(->}[ru]^{\imath} \ar@{->>}[rd]_{\chi'} & & \Sigma^*(W) \ar@{<-->}[ll]^{\txt{\small dual}} \\
 & \Lambda(W) \ar@{^(->}[ru]_\imath } 
\end{equation}
We define a bilinear form on $\Lambda(W)$ by $\langle \Lambda_I(W),\Lambda_J(W) \rangle:=c_{IJ}$ for all $I,J\subseteq S$, where 
\[ c_{IJ}:=\#\{w\in W:D(w^{-1})=I,\ D(w)=J\}.\]
The proof the following proposition is a straightforward exercise, left to the reader.

\begin{proposition}\label{prop:SymForm}
The above bilinear form on $\Lambda(W)$ is well defined, symmetric, and nondegenerate. With this bilinear form the injection $\imath:\Lambda(W)\hookrightarrow \Sigma^*(W)$ and the surjection $\chi':\Sigma(W)\twoheadrightarrow\Lambda(W)$ are dual to each other.
\end{proposition}


By Theorem~\ref{thm:FullDiagram}, the linear maps $\bar\mu_I^S:\Sigma^*(W_I)\to\Sigma^*(W)$ and $\bar\rho_I^S:\Sigma^*(W)\to\Sigma^*(W_I)$ restrict to linear maps $\bar\mu_I^S:\Lambda(W_I)\to\Lambda(W)$ and $\bar\rho_I^S:\Lambda(W)\to\Lambda(W_I)$. More precisely, if $J\subseteq I\subseteq S$ and $K\subseteq S$ then
\[ \bar\mu_I^S(\Lambda_J(W_I)) := \bar\mu_I^S(\chi'(D_J(W_I))) = \chi'(\mu_I^S(D_J(W_I))), \]
\[ \bar\rho_I^S(\Lambda_K(W)) := \bar\rho_I^S(\chi'(D_K(W))) = \chi'(\rho_I^S(D_K(W))).\]
It follows from Proposition~\ref{prop:IndDes} and Proposition~\ref{prop:ResDes} that
\[ \bar\mu_I^S(\Lambda_J(W_I)) = \sum_{J'\cap I=J} \Lambda_{J'}(W) \qand
\bar\rho_I^S(\Lambda_K(W)) = \sum_{z\in\,{^{I}D_K}} \ \sum_{ L(z,I,K)\subseteq K' \subseteq L'(z,I,K) } \Lambda_{K'}(W_I). \]
By diagram chasing in Theorem~\ref{thm:FullDiagram} one checks that $\mu_I^S:\Sigma(W_I)\to\Sigma(W)$ and $\rho_I^S:\Sigma(W)\to\Sigma(W_I)$ also descend to $\bar\mu_I^S:\Lambda(W_I)\to\Lambda(W)$ and $\bar\rho_I^S:\Lambda(W)\to\Lambda(W_I)$, and thus the following result holds.

\begin{corollary}\label{cor:Sym}
The diagram below is commutative and rotating it $180^\circ$ gives a dual diagram.
\[ \xymatrix @R=20pt @C=25pt{
\Sigma(W_I) \ar@{->>}[r]^{\chi'} \ar[d]^{\mu_I^S} & \Lambda(W_I) \ar@{^(->}[r]^{\imath} \ar[d]^{\bar\mu_I^S} & \Sigma^*(W_I) \ar[d]^{\bar\mu_I^S} \\
\Sigma(W) \ar@{->>}[r]^{\chi'} \ar[d]^{\rho_I^S} & \Lambda(W) \ar@{^(->}[r]^{\imath} \ar[d]^{\bar\rho_I^S} & \Sigma^*(W) \ar[d]^{\bar\rho_I^S}\\
\Sigma(W_I) \ar@{->>}[r]^{\chi'} & \Lambda(W_I) \ar@{^(->}[r]^{\imath} & \Sigma^*(W_I)
}\]
\end{corollary}


\begin{definition}
We define a representation $\Lambda$ of both the category $\mathcal Cox$ and its dual category $\mathcal Cox^{op}$ by sending each abstract finite Coxeter system $(W,S)$ to $\Lambda(W)$ and sending an inclusion $(W_I,I)\hookrightarrow (W,S)$ to $\bar\mu_I^S:\Lambda(W_I)\to\Lambda(W)$ and a restriction $(W,S)\twoheadrightarrow (W_I,I)$ to $\bar\rho_I^S:\Lambda(W)\to\Lambda(W_I)$ whenever $I\subseteq S$. There are also dual natural transformations $\imath: \Lambda\hookrightarrow\Sigma^*$ and $\chi':\Sigma\twoheadrightarrow\Lambda$.
\end{definition}

In summary, we have a commutative diagram of representations of categories below.
\begin{equation}\label{eq:DiamondCoxCat}
\xymatrix @R=15pt @C=25pt {
 & \Omega \ar@{->>}[rd]^{\chi'} \\
 \Sigma \ar@{^(->}[ru]^{\imath} \ar@{->>}[rd]_{\chi'} & & \Sigma^* \ar@{<-->}[ll]^{\txt{\small dual}} \\
 & \Lambda \ar@{^(->}[ru]_\imath } 
 \end{equation}

\subsection{P-partitions and free quasisymmetric functions}\label{sec:FQSym}

In this subsection we generalize free quasisymmetric functions from type A to finite Coxeter groups. 
We first review a generalization of the $P$-partition theory by Reiner~\cite{Reiner1}, with some slight but not essential modifications.
See Humphreys~\cite{Humphreys} for details on root systems.
 
Let $E=\mathbb R^n$ be a finite dimensional (real) Euclidean space with standard inner product $(-,-)$. A \emph{root system} is a finite set $\Phi\subset E\setminus\{0\}$ such that
\begin{itemize}
\item if $\alpha\in\Phi$ then $\mathbb R\alpha\cap\Phi=\{\pm\alpha\}$, and
\item if $\alpha,\beta\in\Phi$ then $s_\alpha\beta:=\beta-\frac{2(\alpha,\beta)}{(\alpha,\alpha)}\alpha\in\Phi$.
\end{itemize}
Here $s_\alpha$ is the reflection across the hyperplane $H_\alpha$ perpendicular to $\alpha$. The elements of $\Phi$ are called \emph{roots}. One can choose a linearly independent subset $\Delta\subset\Phi$, whose elements are called \emph{simple roots}, such that every root is either \emph{positive} or \emph{negative}, meaning that it is a linear combination of simple roots with coefficients either all nonnegative or all nonpositive. 
We do not require $\Delta$ to be a basis for $E$, as one can always restrict to the subspace of $E$ spanned by $\Delta$ if needed.
We write $\alpha>0$ ($\alpha<0$ resp.) if $\alpha$ is a positive (negative resp.) root. The root system $\Phi$ is the disjoint union of the set $\Phi^+$ of positive roots and the set $\Phi^-$ of negative roots. 
Let $W$ be the group generated by the set $S$ of \emph{simple reflections} $s_\alpha$ for all $\alpha\in\Delta$. 
Then $(W,S)$ is a finite Coxeter system. 
Note that $\Phi=W \Delta$~\cite[\S1.5]{Humphreys}. 

Conversely, every finite Coxeter system $(W,S)$ can be obtained in above way, i.e., one can realize $W$ as a group generated by a set $S$ of simple reflections of a Euclidean space $E=\mathbb R^n$.
This is called a \emph{geometric realization} of $(W,S)$. 
The simple root $\alpha\in\Delta$ corresponding to a simple reflection $s=s_\alpha\in S$ is denoted by $\alpha_s$.
The action of $W$ on $E$ is orthogonal, i.e. $(wf,wg)=(f,g)$ for all $w\in W$ and all $f,g\in E$. 

\begin{proposition}[{\cite[\S1.6, \S1.7]{Humphreys}}]\label{prop:ws}
Let $s\in S$ and $w\in W$. Then $\ell(ws)>\ell(w) \Leftrightarrow w(\alpha_s)>0$.
\end{proposition}

In general there are different geometric realizations of the same finite Coxeter system. 
In this paper we fix one geometric realization for each abstract finite irreducible Coxeter system.
We will explicitly give a ``standard'' geometric realization for type A, B, and D in \S\ref{sec:FQSymA}, \S\ref{sec:FQSymB}, and \S\ref{sec:FQSymD}, but our main results are still valid if a different geometric realization is fixed.
For other types the choice of a geometric realization is arbitrary for our purposes. 

Now suppose that $(W,S)$ is an arbitrary abstract finite Coxeter system. 
Let $S_1,\ldots,S_k$ be the vertex sets of the connected components of the Coxeter diagram of $(W,S)$.
For each $i\in[k]$ we have already fixed a geometric realization of the irreducible $W_{S_i}$ as a reflection group of a Euclidean space $\mathbb R^{n_i}$.
This gives a realization of $W\cong W_{S_1}\times\cdots\times W_{S_k}$ as a reflection group of the Euclidean space $E=\mathbb R^n$ where $n=n_1+\cdots+n_k$.
Let $\Phi$ be the root system associated with this geometric realization of $(W,S)$ and fix a set $\Delta$ of simple roots.

Reiner~\cite{Reiner1} defined a \emph{parset} (\emph{partial root system}) of $\Phi$ to be a subset $P\subseteq\Phi$ such that 
\begin{itemize}
\item if $\alpha\in P$ then $-\alpha\notin P$, and
\item if $\sum_{i=1}^k c_i\alpha_i \in\Phi$ where $\alpha_i\in P$ and $c_i>0$ for all $i=1,\ldots, k$, then $\sum_{i=1}^k c_i\alpha_i \in P$.
\end{itemize}
Let $P\subseteq\Phi$ be a parset. 
The \emph{Jordan-H\"older set} of $P$ is  $\mathcal L(P):=\{w\in W: P\subseteq w\Phi^+\}$.
Denote by $\mathcal A(P)$ the set of all $P$-partitions, where a \emph{$P$-partition} is a function $f:[n]\to\ZZ$, identified with a vector $(f(1),\ldots,f(n))\in \ZZ^n$, such that
\begin{itemize}
\item $(\alpha,f)\geq0$ for all $\alpha\in P$, and
\item $(\alpha,f)>0$ for all $\alpha\in P\cap\Phi^-$.
\end{itemize}

The next result is well known in type A and actually holds for all finite Coxeter systems. 
A proof can be found in~{\cite[Proposition~3.1.1]{Reiner1}, which does not depend on the extra axiom for a root system adopted in~\cite{Reiner1}: $\Delta$ is a basis for $E$.

\begin{theorem}[Fundamental Theorem of $P$-partitions~{\cite[Proposition~3.1.1]{Reiner1}}]\label{thm:FTP}
For every parset $P$ of $\Phi$, the set $\mathcal A(P)$ of $P$-partitions is the disjoint union of $\mathcal A(w\Phi^+)$ for all $w\in\mathcal L(P)$.
\end{theorem}

Let $\bX=\{\bx_i:i\in\ZZ\}$ be a set of noncommutative variables. The free associative algebra $\mathbb Z\langle \bX \rangle$ generated by $\bX$ has a free $\ZZ$-basis $\{\bx_a:a\in \mathbb Z^n,\ n\ge0\}$, where $\bx_{a_1\cdots a_n}:=\bx_{a_1}\cdots\bx_{a_n}$. Denote by $\FQSym^S$ the $\ZZ$-span of $\bF_P^S$ for all parsets $P$ of $\Phi$, where 
\[\bF_P^S:=\sum_{f\in\mathcal A(P)} \bx_{f(1)}\cdots \bx_{f(n)}\]
is the (noncommutative) generating function of $P$. 
Let $\bF_w^S:=\bF_{w\Phi^+}^S$ and $\bs_w^S:=\bF_{w^{-1}}^S$ for all $w\in W$.

\begin{lemma}\label{lem:Z^n}
The set $\ZZ^n$ is a disjoint union of nonempty subsets: $\ZZ^n = \bigsqcup_{\,w\in W}  \mathcal A(w\Phi^+)$.
\end{lemma}

\begin{proof}
Applying Theorem~\ref{thm:FTP} to $P=\emptyset$ one partitions $\ZZ^n$ into a disjoint union of $\mathcal A(w\Phi^+)$ for all $w\in W$. 
Since the intersection $\bigcap_{\alpha\in\Delta} \{ f\in\ZZ^n:(f,\alpha)>0\}$ is infinite, there exists $f\in\ZZ^n$ such that $f\in\mathcal A(\Phi^+)$.
Then $\mathcal A(w\Phi^+)$ contains $wf$ for each $w\in W$. 
\end{proof}

\begin{proposition}\label{prop:FQSymW}
One has two bases $\{\bF_w^S: w\in W\}$ and $\{\bs_w^S:w\in W\}$ for $\FQSym^S$.
\end{proposition}

\begin{proof}
By Lemma~\ref{lem:Z^n}, $\{\bF_w^S: w\in W\}$ is linearly independent. If $P$ is an arbitrary parset of $\Phi$ then Theorem~\ref{thm:FTP} implies that $\bF_P^S$ is the sum of $\bF_w^S$ for all $w \in \mathcal L(P)$. The result follows.
 \end{proof}

It follows that one has the following two isomorphisms of free $\ZZ$-modules:
\[\begin{matrix}
\bF: & \ZZ W & \xrightarrow{\sim} & \FQSym^S & \qand & \bs: & \ZZ W & \xrightarrow{\sim} & \FQSym^S \\
& w & \mapsto & \bF_w^S  & & & w & \mapsto & \bs_w^S.
\end{matrix} \]

We will also need the following result when we study the $\bs$-basis in type A, B, and D.

\begin{proposition}\label{prop:InverseP} 
Let $w\in W$ and  $f\in\ZZ^n$. Then $f\in\mathcal A(w^{-1}\Phi^+)$ is equivalent to 
\[\begin{cases}
(f,\alpha)\geq0, & {\rm if}\ \alpha>0,\ w\alpha>0,\\
(f,\alpha)<0, & {\rm if}\ \alpha>0,\ w\alpha<0.
\end{cases} \]
\end{proposition}
\begin{proof}
This follows from the definition of $P$-partitions and the observation that $\Phi^+$ is the disjoint union of $\{w\alpha:\alpha>0,w\alpha>0\}$ and $\{-w\alpha:\alpha>0,w\alpha<0\}$.
 \end{proof}  

Next assume $I\subseteq S$. By Humphreys~\cite[\S1.10]{Humphreys}, $W_I$ is isomorphic to the reflection group of $E$ with root system $\Phi_I=\Phi^+_I\sqcup \Phi^-_I$, where $\Phi^+_I$ consists of all roots in $\Phi$ that are nonnegative linear combinations of $\{\alpha\in\Phi^+:s_\alpha\in I\}$ and $\Phi^-_I=-\Phi^+_I$. For each $u\in W_I$ one can check that $u\Phi_I^+$ is a parset of $\Phi$.

\begin{proposition}\label{prop:shuffleW}
Let $I\subseteq S$ and $u\in W_I$. Then $\bF_{u\Phi^+_I}^S = \sum_{z\in\,^IW} \bF_{uz}^S$. 
\end{proposition}

\begin{proof} 
Each element $w\in W$ can be written as $w=uz$ with $z\in W$, and
\[ u \Phi^+_I \subseteq w\Phi^+\,\Leftrightarrow\,z^{-1}\Phi_I^+\subseteq \Phi^+\,\Leftrightarrow\,z^{-1}\in W^I.\]
Here the last equivalence uses Proposition~\ref{prop:ws}. Applying Theorem~\ref{thm:FTP} completes the proof.
 \end{proof}

Define a symmetric bilinear form on $\FQSym^S$ by $\langle \bF_u^S,\bs_v^S\rangle:=\delta_{u,v}$ for all $u,v\in W$. 
This is not positive definite and thus not isomorphic to the bilinear form on $\ZZ W$ defined in~\S\ref{sec:Coxeter}. 
With this bilinear form $\FQSym^S$ becomes self-dual. Define linear maps 
\[\begin{matrix}
\mu_I^S=\bar\mu_I^S:& \FQSym^I & \to & \FQSym^S & \qand & \rho_I^S=\bar\rho_I^S:& \FQSym^S & \to & \FQSym^I  \\
 & \bF_u^I & \mapsto & \bF_{u\Phi_I^+}^S & & & \bF_w^S & \mapsto & \bF_{_Iw}^I.
 \end{matrix}\]

\begin{proposition}\label{prop:CatFQSym}
The diagram below is commutative; rotating it $180^\circ$ gives a dual diagram.
\[ \xymatrix @R=20pt @C=25pt{
 \FQSym^I \ar[d]_-{\mu_I^S}  \ar@{->}[r]^-{\sim}_-{\bs^{-1}} & \ZZ W_I \ar[d]_-{\mu_I^S} \ar@{<->}^{(\ )^{-1}}[r] & \ZZ W_I \ar[d]_-{\bar\mu_I^S} \ar@{->}[r]^-{\sim}_-{\bF} & \FQSym^I \ar[d]_-{\bar\mu_I^S}\\
 \FQSym^S \ar[d]_-{\rho_I^S}  \ar@{->}[r]^-{\sim}_-{\bs^{-1}} & \ZZ W \ar[d]_-{\rho_I^S} \ar@{<->}^{(\ )^{-1}}[r] & \ZZ W \ar[d]_-{\bar\rho_I^S} \ar@{->}[r]^-{\sim}_-{\bF} & \FQSym^S \ar[d]_-{\bar\rho_I^S}\\
 \FQSym^I  \ar@{->}[r]^-{\sim}_-{\bs^{-1}} & \ZZ W_I \ar@{<->}^{(\ )^{-1}}[r] & \ZZ W_I \ar@{->}[r]^-{\sim}_-{\bF} & \FQSym^I
} \]
\end{proposition}
\begin{proof}
Apply Proposition~\ref{prop:MR} and Proposition~\ref{prop:shuffleW}.
 \end{proof}

\begin{definition}
We define a representation $\mathcal FQSym$ of both the category $\mathcal Cox$ and its dual category $\mathcal Cox^{op}$ by sending each abstract finite Coxeter system $(W,S)$ to $\FQSym^S$ and sending an inclusion $(W_I,I)\hookrightarrow(W,S)$ to $\mu_I^S=\bar\mu_I^S: \FQSym^I\to\FQSym^S$ and a restriction $(W,S)\twoheadrightarrow (W_I,I)$ to $\rho_I^S=\bar\rho_I^S: \FQSym^S\to\FQSym^I$ whenever $I\subseteq S$. 
\end{definition}

Proposition~\ref{prop:CatFQSym} implies that $\mathcal FQSym$ is self-dual. 
It also shows that the isomorphisms $\bF:\ZZ W\to\FQSym^S$ and $\bs:\ZZ W\to\FQSym^S$ induce natural isomorphisms $\bF: \Omega^*\to\mathcal FQSym$ and $\bs:\Omega\to\mathcal FQSym$ of representations of categories.

Next, let $\NSym^S$ be the $\ZZ$-submodule of $\FQSym^S$ with two free $\ZZ$-bases consisting of
\[ \bs_I^S:= \sum_{w\in D_I(W)} \bs_w^S \qand  \bh_I^S:=\sum_{J\subseteq I} \bs_J^S =  \bF_{\Phi^+_{S\setminus I}}^S, \quad\forall I\subseteq S\]
where the last equality follows from Proposition~\ref{prop:shuffleW}.
One has $\imath:\NSym^S\hookrightarrow\FQSym^S$ by inclusion.
By Theorem~\ref{thm:FullDiagram} and Proposition~\ref{prop:CatFQSym}, $\mu_I^S: \FQSym^I \to \FQSym^S$ and $\rho_I^S: \FQSym^S\to\FQSym^I$ restrict to $\mu_I^S:\NSym^I\to\NSym^S$ and $\rho_I^S:\NSym^S\to\NSym^I$.

\begin{definition}
We define a representation $\mathcal NSym$ of both the category $\mathcal Cox$ and its dual category $\mathcal Cox^{op}$ by sending each abstract finite Coxeter system $(W,S)$ to $\NSym^S$ and sending an inclusion $(W_I,I)\hookrightarrow(W,S)$ to $\mu_I^S:\NSym^I\to\NSym^S$ and a restriction $(W,S)\twoheadrightarrow (W_I,I)$ to $\rho_I^S:\NSym^S\to\NSym^I$ whenever $I\subseteq S$.
\end{definition}

The isomorphism $\bs:\ZZ W\to\FQSym^S$ restricts to an isomorphism 
\[\begin{matrix}
\bs: & \Sigma(W) & \xrightarrow{\sim} & \NSym^S \\
& D_I(W) & \mapsto & \bs_I^S.
\end{matrix} \]
This induces a natural isomorphism of representations of categories 
\[\bs:\Sigma\to\mathcal NSym.\]

On the other hand, denote by $\QSym^S$ the dual of $\NSym^S$ with two bases $\{F_I^S:I\subseteq S\}$ and $\{M_I^S:I\subseteq S\}$ dual to $\{\bs_I^S:I\subseteq S\}$ and $\{\bh_I^S:I\subseteq S\}$, respectively.
Dual to the injection $\imath:\NSym^S\hookrightarrow \FQSym^S$ is a surjection 
\[ \begin{matrix}
\chi:& \FQSym^S & \twoheadrightarrow & \QSym^S, \\
 & \bF_w^S & \mapsto & F_w^S:=F_{D(w)}^S.
 \end{matrix}\]
By Theorem~\ref{thm:FullDiagram} and Proposition~\ref{prop:CatFQSym}, $\bar\mu_I^S:\FQSym^I\to\FQSym^S$ and $\bar\rho_I^S:\FQSym^S\to\FQSym^I$ descend to $\bar\mu_I^S:\QSym^I\to\QSym^S$ and $\bar\rho_I^S:\QSym^S\to\QSym^I$.

\begin{definition}
We define a representation $\mathcal QSym$ of both the category $\mathcal Cox$ and its dual category $\mathcal Cox^{op}$ by sending each abstract finite Coxeter system $(W,S)$ to $\QSym^S$ and sending an inclusion $(W_I,I)\hookrightarrow(W,S)$ to $\bar\mu_I^S:\QSym^I\to\QSym^S$ and a restriction $(W,S)\to (W_I,I)$ to $\bar\rho_I^S:\QSym^S\to\QSym^I$ whenever $I\subseteq S$.
\end{definition}

The isomorphism $\bF:\ZZ W\xrightarrow\sim\FQSym^S$ descends to an isomorphism 
\[\begin{matrix}
F: & \Sigma^*(W) & \xrightarrow{\sim} & \QSym^S \\
& D^*_I(W) & \mapsto & F_I^S.
\end{matrix} \]
This induces a natural isomorphism of representations of categories 
\[ F: \Sigma^*\to\mathcal QSym.\]

Finally, let $\Sym^S$ be the $\ZZ$-span of 
\[ s_I^S:=\chi(\bs_I^S) = \sum_{w\in D_I(W)} F_{w^{-1}}^S, \quad\forall I\subseteq S.\]
Another spanning set for $\Sym^S$ consists of $h_I^S:=\sum_{J\subseteq I} s_J^S$ for all $I\subseteq S$. One sees that the injection $\imath:\NSym^S\hookrightarrow\FQSym^S$ descends to $\imath:\Sym^S\hookrightarrow\QSym^S$ and the surjection $\chi:\FQSym^S\twoheadrightarrow\QSym^S$ restricts to $\chi:\NSym^S\twoheadrightarrow \Sym^S$.
By Corollary~\ref{cor:Sym}, $\bar\mu_I^S:\QSym^I\to\QSym^S$ and $\bar\rho_I^S:\QSym^S\to\QSym^I$ restrict to $\bar\mu_I^S:\Sym^I\to\Sym^S$ and $\bar\rho_I^S:\Sym^S\to\Sym^I$, and $\mu_I^S:\NSym^I\to\NSym^S$ and $\rho_I^S:\NSym^S\to\NSym^I$ also descend to $\bar\mu_I^S:\Sym^I\to\Sym^S$ and $\bar\rho_I^S:\Sym^S\to\Sym^I$.

\begin{definition} 
We define a representation $\mathcal Sym$ of both the category $\mathcal Cox$ and its dual category $\mathcal Cox^{op}$ by sending each abstract finite Coxeter system $(W,S)$ to $\Sym^S$ and sending an inclusion $(W_I,I)\hookrightarrow(W,S)$ to $\bar\mu_I^S:\Sym^I\to\Sym^S$ and a restriction $(W,S)\twoheadrightarrow (W_I,I)$ to $\bar\rho_I^S:\Sym^S\to\Sym^I$ whenenver $I\subseteq S$.
\end{definition}

One has an isomorphism 
\[\begin{matrix}
s: & \Lambda(W) & \xrightarrow{\sim} & \Sym^S \\
& \Lambda_I(W) & \mapsto & s_I^S
\end{matrix} \]
which is compatible with both $\bs: \Sigma(W) \xrightarrow{\sim} \NSym^S$ and $F: \Sigma^*(W) \xrightarrow{\sim} \QSym^S$. This induces a natural isomorphism of representations of categories
\[ s:\Lambda\to\mathcal Sym.\]

In summary, one has the following commutative diagram of representations of categories, which are in natural isomorphism with  \eqref{eq:DiamondCoxCat}.
\[ \xymatrix @R=15pt @C=20pt {
 & \mathcal FQSym \ar@{->>}[rd]^{\chi} \\
 \mathcal NSym \ar@{^(->}[ru]^{\imath} \ar@{->>}[rd]_{\chi} & & \mathcal QSym \ar@{<-->}[ll]^{\txt{\small dual}} \\
 & \mathcal Sym \ar@{^(->}[ru]_\imath } \] 

\subsection{Representation theory of 0-Hecke algebras}
Now we investigate connections of our previous results with the representation theory of 0-Hecke algebras. 

We first review some general results on representation theory of associative algebras~\cite[\S I]{ASS}. Let $\FF$ be an arbitrary field and let $A$ be a finite dimensional (unital associative) $\FF$-algebra. Let $M$ be a (left) $A$-module. If $M$ is nonzero and has no submodules except $0$ and itself, then $M$ is \emph{simple}. If $M$ is a direct sum of simple $A$-modules then $M$ is \emph{semisimple}. The algebra $A$ is \emph{semisimple} if it is semisimple as an $A$-module. Every module over a semisimple algebra is also semisimple. If an $A$-module $M$ cannot be written as a direct sum of two nonzero $A$-submodules, then $M$ is \emph{indecomposable}. If $M$ is a direct summand of a free $A$-module, then $M$ is \emph{projective}.

One can write $A$ as a direct sum of indecomposable $A$-modules $\P_1,\ldots,\P_k$. The \emph{top} of $\P_i$, denoted by $\C_i:={\rm top}(\P_i)$, is the quotient of $\P_i$ by its unique maximal submodule, and hence simple~\cite[Proposition~I.4.5~(c)]{ASS}. Every projective indecomposable $A$-module is isomorphic to some $\P_i$, and every simple $A$-module is isomorphic to some $\C_i$. 

The \emph{Grothendieck group $G_0(A)$ of the category of finitely generated $A$-modules} is defined as the abelian group $F/R$, where $F$ is the free abelian group on the isomorphism classes $[M]$ of finitely generated $A$-modules $M$, and $R$ is the subgroup of $F$ generated by the elements $[M]-[L]-[N]$ corresponding to all exact sequences $0\to L\to M\to N\to0$ of finitely generated $A$-modules. The \emph{Grothendieck group $K_0(A)$ of the category of finitely generated projective $A$-modules} is defined similarly.   We often identify a finitely generated (projective) $A$-module with the corresponding element in the Grothendieck group $G_0(A)$ ($K_0(A)$). If $L$, $M$, and $N$ are all projective $A$-modules, then $0\to L\to M\to N\to0$ is equivalent to $M\cong L\oplus N$. If $A$ is semisimple then $G_0(A)=K_0(A)$.

Suppose that $\{\C_1,\ldots,\C_\ell\}$ is a complete list of non-isomorphic simple $A$-modules, and $\{\P_1,\ldots,\P_\ell\}$ is a complete list of pairwise non-isomorphic projective indecomposable $A$-modules, labeled in such a way that $\C_i=\mathrm{top}(\P_i)$ for each $i$. Then $G_0(A)$ and $K_0(A)$ are free abelian groups with bases $\{\C_1,\ldots,\C_\ell\}$ and $\{\P_1,\ldots,\P_\ell\}$, respectively. One has
\[ \dim_\FF\Hom_A (\P_i,\C_j)=\dim_\FF \Hom_A (\C_i,\C_j)=\delta_{i,j}. \]
This defines a pairing between $G_0(A)$ and $K_0(A)$, denoted by $\langle -,-\rangle$.
One has $\langle P,C\rangle = \dim_\FF\Hom_A(P,C)$ for any finite dimensional projective $A$-module $P$ and any finite dimensional $A$-module $C$, since the hom functor $\Hom_A (P, -)$ is exact.

Let $B$ be a subalgebra of $A$. For any $A$-module $M$ and $B$-module $N$, the induction $N\uparrow\,_B^A$ of $N$ from $B$ to $A$ is defined as the $A$-module $A\otimes_B N$, and the restriction $M\downarrow\,_B^A$ of $M$ from $A$ to $B$ is defined as $M$ itself viewed as a $B$-module. The induction and restriction are both well defined for isomorphic classes of modules.
The following result is well known.

\vskip5pt\noindent\textbf{Frobenius Reciprocity.}\quad ${\rm Hom}_A(N\uparrow\,_B^A,M) \cong {\rm Hom}_B(N,M\downarrow\,_B^A)$.
\vskip5pt


Now recall that a finite Coxeter group $W$ is generated by a finite set $S$ with 
\begin{itemize}
\item
quadratic relations $s^2=1$ for all $s\in S$ and 
\item
braid relations $(sts\cdots)_{m_{st}}=(tst\cdots)_{m_{st}}$ for all $s,t\in S$.
\end{itemize}
We focus on the \emph{$0$-Hecke algebra} $H_W(0)$ of the Coxeter system $(W,S)$, which is a deformation of the group algebra of $W$. It is the $\FF$-algebra generated by $\{\pib_s:s\in S\}$ with 
\begin{itemize}
\item
quadratic relations $\pib_s^2=-\pib_s$ for all $s\in S$ and 
\item
braid relations $(\pib_s\pib_t\pib_s\cdots)_{m_{st}}=(\pib_t\pib_s\pib_t\cdots)_{m_{st}}$ for all $s,t\in S$.
\end{itemize}
If $w\in W$ has a reduced expression $w=s_1\cdots s_\ell$, where $s_1,\ldots,s_\ell\in S$, then $\pib_w:=\pib_{s_1}\cdots\pib_{s_\ell}$ is well defined. In fact, by the Word Property of $W$~\cite[Theorem~3.3.1]{BjornerBrenti} or \cite[Theorem~1.9]{Lusztig}, $\pib_w$ depends only on $w$, not on the choice of the above reduced expression of $w$. 
By this definition, for any $s\in S$ and $w\in S$,
\[ \pib_s\pib_w = \begin{cases}
\pib_{sw}, & \text{if } \ell(sw)>\ell(w), \\
-\pib_w, & \text{if } \ell(sw)<\ell(w).
\end{cases} \]
One can show that $\{\pib_w:w\in W\}$ is an $\FF$-basis for $H_W(0)$ using~\cite[proof of Proposition 3.3]{Lusztig} with some straightforward modifications.
See also Stembridge~\cite[Proposition~2.1]{Stembridge}.

Let $\pi_s:=\pib_s+1$ for each $s\in S$. Then $\pib_s\pi_s=\pi_s\pib_s=0$.
For any $u$ and $w$ in $W$, write $u\le w$ if some reduced expression $w$ contains a subword equal to $u$.
This is the well-known \emph{Bruhat order} of $W$.
If $w\in W$ has a reduced expression $w=s_1\cdots s_\ell$ then $\pi_w:=\pi_{s_1}\cdots\pi_{s_\ell}$ is well defined as Stembridge~\cite[Lemma~3.2]{Stembridge} showed that
\[ \pi_w = \sum_{u\le w} \pib_u \]
which does not depend on the choice of the reduced expression $w=s_1\cdots s_\ell$.
This implies that $\{\pi_s:s\in S\}$ is another generating set for $H_W(0)$ satisfying the quadratic relations $\pi_s^2=\pi_s$ for all $s\in S$ and the same braid relations as $\{\pib_s:s\in S\}$~\cite[Lemma~3.3]{Stembridge}.

Norton~\cite[4.12, 4.13]{Norton} obtained an $H_W(0)$-module decomposition
\[ H_W(0)=\bigoplus_{I\subseteq S} \P_I^S \]
where 
$ \P_I^S := H_W(0)\pib_{w_0(I)}\pi_{w_0(I^c)} $
is an indecomposable $H_W(0)$-module with an $\FF$-basis 
\[ \left\{\pib_w\pi_{w_0(I^c)}: w\in W,\ D(w)=I\right\}. \]
The top of $\P_I^S$, denoted by $\C_I^S$, is a one-dimensional simple $H_W(0)$-module on which $\pib_i$ acts by $-1$ if $i\in I$ or by $0$ if $i\notin I$. Thus $\{\P_I^S:I\subseteq S\}$ and $\{\C_I^S:I\subseteq S\}$ are complete lists of pairwise non-isomorphic projective indecomposable and simple $H_W(0)$-modules, respectively. 

If $w\in W$ then denote by $\C_w^S$ the simple $H_W(0)$-module indexed by $D(w)\subseteq S$. Sending $\C_w^S$ to $D^*_w(W)$ for all $w\in W$ gives an isomorphism $G_0(H_W(0))\cong \Sigma^*(W)$ of free $\ZZ$-modules. Similarly, sending $\P_I^S$ to $D_I(W)$ for all $I\subseteq S$ gives an isomorphism $K_0(H_W(0))\cong \Sigma(W)$ of free $\ZZ$-modules.

Let $I\subseteq S$. The \emph{parabolic subalgebra} $H_{W_I}(0)$ of the 0-Hecke algebra $H_W(0)$ is generated by $\left\{\, \pib_s:s\in I \,\right\}$ and it is also isomorphic to the $0$-Hecke algebra of the parabolic subgroup $W_{I}$ of $W$. 
By Proposition~\ref{prop:parabolic}, 
\begin{equation}\label{eq:S/I}
H_W(0)=\bigoplus_{w\in W^{I}} \pib_w H_{W_I}(0)
=\bigoplus_{w\in\,^{I}W} H_{W_I}(0)\pib_w.
\end{equation}
We define two linear maps
\begin{equation}\label{def:IndResG0}
\begin{matrix}
\bar\mu_I^S: & G_0(H_{W_I}(0)) & \to & G_0(H_W(0)) & \text{ and } & \bar\rho_I^S: & G_0(H_W(0)) & \to & G_0(H_{W_I}(0)) \\
& M & \mapsto &  M \uparrow\,_{H_{W_I}(0)}^{H_W(0)} & & & N & \mapsto & N\downarrow\,_{H_{W_I}(0)}^{H_W(0)}.
\end{matrix}
\end{equation}

\begin{proposition}\label{prop:IndResG0K0}
The linear maps $\bar\mu_{I}^S$ and $\bar\rho_{I}^S$ are well-defined by \eqref{def:IndResG0}. Moreover, they restrict to linear maps $\mu_I^S: K_0(H_{W_I}(0)) \to K_0(H_W(0))$ and $\rho_I^S: K_0(H_W(0)) \to K_0(H_{W_I}(0))$.
\end{proposition}

\begin{proof}
This can be proved similarly as Bergeron and Li~\cite[\S3]{BergeronLi}.
 \end{proof}

\begin{proposition}\label{prop:IndC}
If $I\subseteq S$ and $w\in W_I$ then $\bar\mu_I^S(\C_w^I) = \sum_{z\in\,{^I W}} \C_{wz}^S.$
\end{proposition}

\begin{proof}
Fix $w\in W_I$ and write $\C_w^I = \FF a$. The $\pib_s(a)$ equals $-a$ if $s\in D(w)$ or $0$ otherwise.  Let $M=\bar\mu(\C_w^I)$.  Denote by $\otimes_I$ the tensor product over $H_{W_I}(0)$. By \eqref{eq:S/I}, $M$ has an $\FF$-basis $ \left\{\,\pib_z\otimes_I a: z\in W^{I}\,\right\}. $
Define $M_k$ to be the $\FF$-span of the elements $\pib_z\otimes_I a$ for all $z\in W^{I}$ with $\ell(z)\geq k$. Then one has a filtration $M=M_0\supseteq M_1\supseteq \cdots\supseteq M_\ell=0$ for some positive integer $\ell$. It suffices to show that, if $z\in W^{I}$ with $\ell(z)=k$, and if $s\in S$, then
\[
\begin{cases}
\pib_s(\pib_z \otimes_I a)=-\pib_z \otimes_I a, & {\rm if}\ s\in D(wz^{-1}),\\
\pib_s(\pib_z\otimes_I a)\in M_{k+1}, & {\rm if}\ s\in S\setminus D(wz^{-1}).
\end{cases}
\]
We distinguish the following cases.

If $s\in D(z^{-1})$ then one has $\pib_s(\pib_z\otimes_I a)=-\pib_z\otimes_I a$ and also $s\in D(wz^{-1})$ since
\[
\ell(wz^{-1}s)\leq\ell(w)+\ell(z^{-1}s) < \ell(w)+\ell(z^{-1})=\ell(wz^{-1}).
\]

If $s\notin D(z^{-1})$ and $sz\in W^{I}$, then one has $\pib_s(\pib_z\otimes_I a)=\pib_{sz}\otimes_I a\in M_{k+1}$ and also $s\notin D(wz^{-1})$ since
\[
\ell(wz^{-1}s)=\ell(w)+\ell(z^{-1}s)>\ell(w)+\ell(z^{-1})=\ell(wz^{-1}).
\]

If $s\notin D(z^{-1})$ and $sz\notin W^{I}$, then by Lemma~\ref{lem:Des} there exists a unique $r\in I$ such that $sz=zr$. 
Thus $\ell(zr)=\ell(sz)=\ell(z)+1$ and so $\pib_s\pib_z=\pib_{sz}=\pib_{zr}=\pib_z\pib_r$. If $r\in D(w)$ then 
\[
\pib_s(\pib_z\otimes_I a)=\pib_z\pib_r\otimes_I a = \pib_z\otimes_I \pib_r(a) = -\pib_z\otimes_I a \qand
\]
\[
\ell(wz^{-1}s)=\ell(wrz^{-1})=\ell(wr)+\ell(z^{-1})<\ell(w)+\ell(z^{-1})=\ell(wz^{-1}).
\]
Similarly, if $r\notin D(w)$ then $\pib_s(\pib_z\otimes_I a)=0$ and $\ell(wz^{-1}s)>\ell(wz^{-1})$. 
 \end{proof}

\begin{proposition}\label{prop:ResC}
Let $I, K\subseteq S$ and $w\in W$. Then $\bar\rho_I^S(\C_K^S) = \C_{I\cap K}^I$ and $\bar\rho_I^S(\C_w^S) = \C_{_Iw}^I$.
\end{proposition}

\begin{proof}
The first equality follows immediately from the definition of $\C_K$. 
The second equality is equivalent to the first one since $D(\,_Iw)=D(w)\cap I$ by Lemma~\ref{lem:DesRes}.
 \end{proof}

Next, we consider the induction of projective modules. Let $I,J\subseteq S$. We define a cyclic $H_W(0)$-module $\P_{I,J}^S:= H_W(0) \pib_{w_0(I)} \pi_{w_0(J\setminus I)}$. We may assume $I\subseteq J\subseteq S$ without loss of generality, since $\P_{I,J}^S = \P_{I,I\cup J}^S$. One sees that $\P_{I,S}^S=\P_I$. In general, we showed the following result in~\cite[Theorem~3.2]{H0Tab}.

\begin{proposition}\label{prop:IndP}\cite[Theorem~3.2]{H0Tab}
If $I, J\subseteq S$ then $\P_{I,J}^S$ has a basis 
\[ \left\{ \pib_w\pi_{w_0(J\setminus I)}: w\in W,\ I\subseteq D(w)\subseteq (S\setminus J)\cup I \right\}.\]
If $I\subseteq J\subseteq S$ then
\[ \mu_J^S(\P_I^J) \cong \P_{I,J}^S\cong  \bigoplus_{K\subseteq S\setminus J} \P_{I\cup K}. \]
\end{proposition}

Finally, we investigate the restriction of a projective indecomposable module.

\begin{proposition}\label{prop:ResP}
Let $I,K\subseteq S$. Write $L(z)=L(z,I,K)$ and $L'(z)=L'(z,I,K)$. Then 
\[ \rho_I^S(\P_K^S)\cong \bigoplus_{z\in\,{^{I}D_K}}\ \P_{L(z),\, I\setminus L'(z)}^I
\cong \bigoplus_{z\in\,{^{I}D_K}}\ \ \bigoplus_{ L(z) \subseteq K'\subseteq L'(z) } \P_{K'}^I.
\]
\end{proposition}

\begin{proof}
Since $\P_{K}^S$ has an $\FF$-basis $\{\pib_w\pi_{w_0(K^c)}: w\in D_K(W)\}$, one has a decomposition 
\[ \rho_I^S(\P_K^S)\ = \bigoplus_{z\in\,{^I}D_K} \P_{K,z} \]
of vector spaces, where each $\P_{K,z}$ has an $\FF$-basis 
\[ \left\{\pib_w\pi_{w_0(K^c)}: w\in D_K(W),\ ^Iw=z\, \right\}. \]
By Proposition~\ref{prop:ResDes} and Proposition~\ref{prop:IndP}, one has an isomorphism 
$ \phi: \P_{K,z} \xrightarrow\sim \P_{L(z),\, I\setminus L'(z)}^I$ of vector spaces by 
\begin{equation}\label{eq:iso}
\pib_w\pi_{w_0(K^c)}\mapsto \pib_{w_{I}}\pi_{w_0(I\setminus L'(z))}, \quad \forall w\in D_K(W)\ {\rm with}\ ^Iw=z.
\end{equation}
It remains to show that $\phi$ preserves $H_{W_I}(0)$-actions. Let $w\in D_K(W)$, $w_{I}=u$,  $^Iw=z$, and $s\in I$.

If $s\in D(u^{-1})$ then $s\in D(w^{-1})$ and thus $\pib_i$ acts by $-1$ on both sides of (\ref{eq:iso}). 

Next we assume $s\notin D(u^{-1})$ and $D(su)\subseteq L'(z)$. Then $\ell(sw)=\ell(su)+\ell(z)>\ell(w)$. Thus
\[ \pib_s\pib_w\pi_{w_0(K^c)}=\pib_{sw}\pi_{w_0(K^c)} \qand \pib_s\pib_u\pi_{w_0(I\setminus L'(z))}=\pib_{su}\pi_{w_0(I\setminus L'(z))}. \] 
One also has $(sw)_{I}=su$, $^I(sw)=z$, and $L(z)\subseteq D(u)\subseteq D(su)\subseteq L'(z)$. Hence $D(sw)=K$ by Lemma~\ref{lem:DesIJ}.

Finally assume $s\notin D(u^{-1})$ and $D(su)\not\subseteq L'(z)$. Lemma~\ref{lem:Des} implies that $D(su)\setminus L'(z) = \{r\}$ where $r=u^{-1}su$, i.e. $ur=su$. The definition of $L'(z)$ implies that $r\in D(tz^{-1})\cap I$ for some $t\in D(z)^c\setminus K$. Then applying Lemma~\ref{lem:Des} again gives $tz^{-1}=z^{-1}r$. Hence $sw=suz=urz=uzt$. It follows that 
\[ \pib_s\pib_w\pi_{w_0(K^c)}=\pib_{uz}\pib_t\pi_{w_0(K^c)}=0 \qand \pib_s\pib_u\pi_{w_0(I\setminus L'(z))}=\pib_u\pib_r\pi_{w_0(I\setminus L'(z))} = 0.\]

Therefore $\phi$ is indeed an isomorphism of $H_{W_I}(0)$-modules. This shows the first desired isomorphism. The second one follows immediately from Proposition~\ref{prop:IndP}.
 \end{proof}

\begin{theorem}\label{thm:GrH0}
The following two diagrams are commutative and dual to each other:
\[ \xymatrix @R=25pt @C=25pt {
G_0(H_{W_I}(0)) \ar[d]_{\bar\mu_I^S}  \ar@{<->}[r]^-{\sim} & \Sigma^*(W_I) \ar[d]^{\bar\mu_I^S} \\
G_0(H_W(0)) \ar[d]_{\bar\rho_I^S} \ar@{<->}[r]^-{\sim} & \Sigma^*(W) \ar[d]^{\bar\rho_I^S} \\
G_0(H_{W_I}(0)) \ar@{<->}[r]^-{\sim} & \Sigma^*(W_I)
}\qquad\qquad
\xymatrix @R=25pt @C=25pt {
K_0(H_{W_I}(0))  \ar@{<->}[r]^-{\sim} & \Sigma(W_I) \\
K_0(H_W(0)) \ar[u]^{\rho_I^S} \ar@{<->}[r]^-{\sim} & \Sigma(W) \ar[u]^{\rho_I^S} \\
K_0(H_{W_I}(0)) \ar[u]^{\mu_I^S} \ar@{<->}[r]^-{\sim} & \Sigma(W_I) \ar[u]^{\mu_I^S} 
} \]
\end{theorem}

\begin{proof}
Compare Proposition~\ref{prop:IndC}, \ref{prop:ResC},  \ref{prop:IndP}, and \ref{prop:ResP} with Proposition~\ref{prop:IndDes} and \ref{prop:ResDes}.
 \end{proof}

\begin{remark}
By Theorem~\ref{thm:GrH0}, if $J\subseteq I\subseteq S$ and $K\subseteq S$ then
\[ \left\langle \mu_I^S(\P_J^I),\C_K^S \right\rangle
=\left\langle \P_J^I, \bar\rho_I^S(\C_K^S) \right\rangle \qand 
\left\langle \P_K^S, \bar\mu_I^S(\C_J^I) \right\rangle
= \left\langle \rho_I^S(\P_K^S), \C_J^I \right\rangle. \]
Note that the first equality is also a consequence of the Frobenius reciprocity, but the second one is not.
\end{remark}

\begin{definition}
We define a representation $\mathcal G_0$ ($\mathcal K_0$ resp.) of both the category $\mathcal Cox$ and its dual category $\mathcal Cox^{op}$ by sending each abstract finite Coxeter system $(W,S)$ to the Grothendieck group $G_0(H_W(0))$ ($K_0(H_W(0))$ resp.) and sending an inclusion $(W_I,I)\hookrightarrow(W,S)$ to $\bar\mu_I^S$ ($\mu_I^S$ resp.) and sending a restriction $(W,S)\twoheadrightarrow(W_I,I)$ to $\bar\rho_I^S$ ($\rho_I^S$ resp.) whenever $I\subseteq S$.
\end{definition}

Theorem~\ref{thm:GrH0} gives a natural isomorphism between $\mathcal G_0$ ($\mathcal K_0$ resp.) and $\Sigma^*$ ($\Sigma$ resp.). Combining this with the natural isomorphism $F: \Sigma^* \to \mathcal QSym$ and $\bs: \Sigma\to\mathcal NSym$ one has two natural isomorphisms
\[ \Ch:  \mathcal G_0 \to  \mathcal QSym \qand
\bch:  \mathcal K_0 \to  \mathcal NSym \]
which are induced by the following characteristic maps 
\[ \begin{matrix}
\Ch: & G_0(H_W(0)) & \to & \QSym^S & \qand & \bch: & K_0(H_W(0)) & \to & \NSym^S \\
 & \C_w^S & \mapsto & F_w^S & &  & \P_I^S & \mapsto & \bs_I^S.
 \end{matrix} \]
 
Lastly, we provide a result on $\Ch(\P_I^S)$ for later use.

\begin{proposition}\label{prop:PtoC}
If $I\subseteq S$ then $\Ch(\P_I^S) = \chi(\bch(\P_I^S)) = s_I^S$.
\end{proposition}

\begin{proof}
Recall that $\P_I^S$ has a basis $\{\pib_w\pi_{w_0(I^c)}:w\in D_I(W)\}$. One has a filtration of $\P_I^S$ by the length of $w$ for all $w\in D_I(W)$. If $s\in S$ and $w\in D_I(W)$ then
\[ \pib_s\pib_w\pi_{w_0(I^c)} = 
\begin{cases}
-\pib_w\pi_{w_0(I^c)}, & \text{if } s\in D(w^{-1}),\\
0, & \text{if } s\notin D(w^{-1}),\ D(sw)\not\subseteq I,\\
\pib_{sw} \pi_{w_0(I^c)} , & \text{if } s\notin D(w^{-1}),\ D(sw)\subseteq I.
\end{cases}\]
It follows that $\Ch(\P_I^S)=\sum_{w\in D_I(W)} F_{w^{-1}}^S$. Hence the result holds.
 \end{proof}

\subsection{Connections to work of Aguiar and Mahajan}\label{sec:AguiarMahajan}

Let $(W,S)$ be a finite Coxeter system.
In this subsection we compare the following diagrams.
\begin{equation}\label{eq:DiamondKW}
\xymatrix @R=10pt @C=5pt {
 & \FQSym^S \ar@{=}[r] & \FQSym^S \ar@{->>}[rd]^{\chi} \\
 \NSym^S \ar@{^(->}[ru]^{\imath} \ar@{->>}[rd]_{\chi} & & & \QSym^S \ar@{<-->}[lll]^{\txt{\small dual}} \\
 & \Sym^S \ar@{=}[r] & \Sym^S \ar@{^(->}[ru]_{\imath} } 
\quad
\xymatrix @R=12pt @C=12pt {
 & \KK W \ar@{<->}^{s}[r] & \KK W^* \ar@{->>}[rd]^{\des} \\
 \KK\overline Q \ar@{^(->}[ru]^{\theta} \ar@{->>}[rd]_{\supp} & & & \KK\overline Q^* \ar@{<-->}[lll]^{\txt{\small dual}} \\
 & \KK \overline L \ar@{<->}^{\phi}[r] & \KK\overline L^* \ar@{^(->}[ru]_{\supp^*} } 
\end{equation}
The first diagram is a commutative diagram of free $\ZZ$-modules obtained from our results in Section~\ref{sec:FQSym} and is equivalent to the diagram~\eqref{eq:DiamondZW}.
The second diagram is a commutative diagram of vector spaces over a field $\KK$ of characteristic zero obtained by Aguiar and Mahajan~\cite[Theorem~5.7.1]{AguiarMahajan}, which can be generalized to \emph{left regular bands}.
The remaining of this subsection is devoted to the proof of the following result. 

\begin{proposition}\label{prop:DiagramIso}
In \eqref{eq:DiamondKW}, with an extension of scalars from $\ZZ$ to $\KK$, the first diagram becomes isomorphic to the second one. 
\end{proposition}

Recall that $\FQSym^S$ has dual bases $\{\bF^S_w:w\in W \}$ and $\{\bs^S_w=\bF^S_{w^{-1}}:w\in W\}$, and $\NSym$ and $\QSym$ have dual bases $\left\{\bs^S_I:I\subseteq S \right\}$ and $\left\{F^S_I :I\subseteq S\right\}$.
One has $\imath:\NSym^S\hookrightarrow \FQSym^S$ defined by $\bs_I^S=\sum_{w\in D_I(W)}\bs_w^S$ and $\chi:\FQSym^S\twoheadrightarrow\QSym^S$ defined by $\chi(\bF_w^S):=F_{D(w)}^S$.
We keep the superscript $S$ even though the Coxeter system $(W,S)$ is fixed in this subsection.

Write $U_\KK := U\otimes_\ZZ \KK$ for any $\ZZ$-module $U$.
By~\cite[\S5.7.1]{AguiarMahajan}, $\KK W$ and $\KK\overline Q$ have canonical bases $\{K_w:w\in W\}$ and $\{K_I:I\subseteq S\}$, and $\KK W^*$ and $\KK\overline Q^*$ have dual bases $\{F_w:w\in W\}$ and $\{F_I:I\subseteq S\}$.
There are dual bases $\{H_I:I\subseteq S\}$ and $\{M_I:I\subseteq S\}$ for $\KK\overline Q$ and $\KK\overline Q^*$ given by
\[ H_I = \sum_{J\subseteq I} K_J \qand F_I = \sum_{I\subseteq J} M_J. \]
There is no superscript $S$ in these bases. 
One has isomorphisms of vector spaces:
\[\begin{matrix}
\NSym^S_\KK & \cong & \KK\overline Q, \\
\bs^S_I & \mapsto & K_I,
\end{matrix} \qquad
\begin{matrix}
\FQSym^S_\KK & \cong & \KK W, \\ 
\bs_w^S & \mapsto & K_w,
\end{matrix} \qquad
\begin{matrix}
\FQSym^S_\KK & \cong & \KK W^*, \\
\bF_w^S & \mapsto & F_w,
\end{matrix} \qquad
\begin{matrix}
\QSym^S_\KK & \cong & \KK\overline Q^*, \\
F_I^S & \mapsto & F_I.
\end{matrix} \] 
By the definitions of $\theta$, $s$, and $\des$~\cite[\S5.7.4]{AguiarMahajan}, one has the following commutative diagram.
\begin{equation}\label{eq:top} \xymatrix @R=15pt @C=20pt {
\NSym^S_\KK \ar@{^(->}^\imath[r] \ar@{<->}^\wr[d] & \FQSym^S_\KK \ar@{=}[r] \ar@{<->}^\wr[d] & \FQSym^S_\KK \ar@{->>}[r]^{\chi} \ar@{<->}^\wr[d] & \QSym^S \ar@{<->}^\wr[d] \\
\KK\overline Q \ar@{^(->}^\theta[r]  & \KK W \ar@{<->}^{s}[r] & \KK W^* \ar@{->>}[r]^{\des} & \KK\overline Q^* }
\end{equation}
This shows that the top halves of the two diagrams in~\eqref{eq:DiamondKW} agree with each other.
  
Next we study the bottom halves of the two diagrams in~\eqref{eq:DiamondKW}.
We first define $\supp$ and $\overline L$.
Recall that $W$ can be realized as a group generated by a set $S$ of reflections of a Euclidean space $E=\mathbb R^n$.
The \emph{reflection arrangement} associated with $(W,S)$ consists of all hyperplanes such that the reflections across these hyperplanes belong to $W$.
This hyperplane arrangement gives rise a simplicial complex called the \emph{Coxeter complex} $\Sigma$ of $(W,S)$, which can be also constructed algebracially by partially ordering the parabolic cosets $wW_I$ for all $w\in W$ and all $I\subseteq S$ using reverse inclusion. 
A \emph{flat} is an intersection of a subset of reflection hyperplanes of $(W,S)$.
The \emph{intersection lattice} $L$ consists of all flats ordered by inclusion.
There is a sujection $\supp: \Sigma\twoheadrightarrow L$ sending a face of $\Sigma$ to its linear span, i.e., the intersection of the reflection hyperplanes containing this face.

The $W$-action on the Euclidean space $E$ induces $W$-actions on $\Sigma$ and $L$.
If $W$ acts on a vector space $V$ then $V^W:=\{ v\in V: w(v)=v,\ \forall w\in W\}$ is the \emph{$W$-fixed subspace of $V$}.
By Barcelo and Ihrig~\cite{BarceloIhrig}, we can identify the set $\{ wW_I w^{-1}: w\in W,\ I\subseteq S\}$ ordered by reserve inclusion with the intersection lattice $L$ via $wW_Iw^{-1}\mapsto E^{wW_Iw^{-1}}$, and the $W$-action on $L$ corresponds to the conjugate action of $W$ on parabolic subgroups.
The following algebraic interpretation of $\supp: \Sigma\twoheadrightarrow L$ has been generalized by Miller~\cite{Miller} to well-generated complex reflection groups:
\begin{equation}\label{eq:supp} 
\supp(wW_I) = E^{wW_Iw^{-1}}, \quad \forall w\in W,\ \forall I\subseteq S.
\end{equation}

Next, one sees that $(\KK\Sigma)^W$ has a basis consisting of 
\begin{equation}\label{eq:sigma} 
\sigma_I:= \sum_{w\in W^{I^c}} wW_{I^c},\quad\forall I\in S.
\end{equation}
Sending $H_I$ to $\sigma_I$ for all $I\subseteq S$ gives an isomorphism $\KK \overline Q\cong (\KK\Sigma)^W$.
On the other hand, let $\overline L$ be the set of the orbits of the $W$-action on $L$. 
Sending an orbit to the sum of its elements gives $\KK \overline L\cong (\KK L)^W$.
Then $\supp: \Sigma \twoheadrightarrow L$ induces $\supp: \KK\overline Q \twoheadrightarrow \KK \overline L$, as illustrated below.
\[ \xymatrix @R=15pt{ 
\KK \overline Q \cong (\KK \Sigma)^W \ar@{->>}[d]^\supp \ar@{^(->}[r] & \KK W \ar@{->>}[d]^\supp \\
\KK \overline L \cong (\KK L)^W \ar@{^(->}[r] & \KK L
} \]

Write $h_I:=\supp(H_I)$ for all $I\subseteq S$. Then $\KK\overline L$ is $\KK$-spanned by $\{h_I: I\subseteq S\}$.
We write $I\sim J$ if two parabolic subgroups $W_{I^c}$ and $W_{J^c}$ are conjugate in $W$.
This defines an equivalence relation on subsets of $S$ and we denote by $\Pi(W,S)$ the set of all equivalence classes.
Then $|\overline L|=|\Pi(W,S)|$.
By~\eqref{eq:supp} and \eqref{eq:sigma}, $\supp:\KK\overline Q\twoheadrightarrow \KK \overline L$ sends $H_I$ to the conjugacy class $\{wW_{I^c}w^{-1}:w\in W\}$ and thus $h_I = h_J\ \Leftrightarrow\ I\sim J$.
Hence $h_\lambda:=h_I$ is well defined if $\lambda\in\Pi(W,S)$ and $I\in\lambda$, and $\KK \overline L$ has a basis $\{h_\lambda:\lambda\in\Pi(W,S)\}$.

By~\cite[\S2.6]{AguiarMahajan}, $\KK\overline L$ has a bilinear form defined by 
\[ \langle h_I, h_J \rangle := \#\{w\in W: D(w)\subseteq I,\ \des(w^{-1})\subseteq J\}, \quad\forall I,J\subseteq S. \]
The dual space $\KK \overline L^*$ has a basis $\{m_\lambda:\lambda\in\Pi(W,S)\}$ dual to $\{h_\lambda:\lambda\in\Pi(W,S)\}$.
The map $\phi:\KK\overline L\to\KK L^*$ and the map $\supp^*:\KK \overline L^*\to\KK \overline Q^*$ are defined by 
\[ \phi(h_\lambda) := \sum_{\mu\in\Pi(W,S)} \langle h_\lambda,h_\mu\rangle m_\mu
\qand \supp^*(m_\lambda) := \sum_{I\in\lambda} M_I. \]

Next, we construct a free $\ZZ$-basis for $\Sym^S$.
Recall that $\NSym^S$ and $\QSym^S$ have dual bases $\{\bh_I^S:I\subseteq S\}$ and $\{M_I^S:I\subseteq S\}$ defined by $\bh_I^S = \sum_{J\subseteq I} \bs_J^S$ and $F^S_I = \sum_{I\subseteq J} M^S_J$, and $\Sym^S$ is the $\ZZ$-span of $\{h_I^S:I\subseteq S\}$ where $h_I^S:=\chi(\bh_I^S)$.

\begin{proposition}\label{prop:Sym}
The element $h_\lambda^S:=h_I^S$ is well defined whenever $\lambda\in\Pi(W,S)$ and $I\in\lambda$, and the set $\{h_\lambda^S:\lambda\in \Pi(W,S) \}$ is a free $\ZZ$-basis for $\Sym^S$.
\end{proposition}

\begin{proof}
We first need an easy group theory result.
If $L$ and $R$ are two subgroups of $W$ then we write $L\backslash W/R:=\{LwR:w\in W\}$ for the set of all double $(L,R)$-cosets in $W$.
If $R'=uRu^{-1}$ for some $u\in W$ then $|L\backslash W/R| = |L\backslash W/R'|$ as there is a bijection between $L\backslash W/R$ and $L\backslash W/R'$ by $LwR\mapsto LwRu^{-1} = Lwu^{-1}R'$ for all $w\in W$.
Now if $I,J\subseteq S$ then
\[ \langle \bh^S_I, h^S_J\rangle  =  \sum_{K\subseteq I} \sum_{\substack{ w\in W \\ D(w)\subseteq J}} \langle \bs^S_K, F_{w^{-1}}^S \rangle = \left | \{w\in W: D(w^{-1})\subseteq I, \ D(w)\subseteq J \} \right|= \left| W_{I^c}\backslash W/W_{J^c} \right|. \]
If $J\sim J'$ then $\langle \bh^S_I, h^S_J\rangle = \langle \bh^S_I, h^S_{J'}\rangle$, i.e., $h^S_J$ and $h^S_{J'}$ have the same expansion in the basis $\{M_I:I\subseteq S\}$ for $\QSym^S$, and thus $h^S_J = h^S_{J'}$.
Then $h_\lambda^S:=h_I^S$ is well defined if $\lambda\in\Pi(W,S)$ and $I\in\lambda$.
This gives a spanning set $\left\{h_\lambda^S:\lambda\in\Pi(W,S) \right\}$ for $\Sym^S$. Since
\[ \Sym^S_\KK\cong\chi\circ\imath(\NSym^S_\KK) \cong \des\circ s\circ \theta(\KK\overline Q) \cong \supp(\KK\overline Q) \cong \KK \overline L\]
and $|\Pi(W,S)|=|\overline L|$, this spanning set must be linearly independent.
\end{proof}

Recall that $\Lambda(W)$ has a symmetric bilinear form defined by $\langle \Lambda_I(W), \Lambda_J(W) \rangle := c_{IJ}$.
Applying the isomorphism $s: \Lambda(W) \cong\Sym^S$ defined by $\Lambda_I(W) \mapsto s_I$ one has
\[ \langle h^S_I,h^S_J\rangle = \# \{ w\in W: D(w) \subseteq I,\ D(w^{-1})\subseteq J \}, \qquad \forall I,J\in S. \] 
This is compatible with the pairing between $\NSym^S$ and $\QSym^S$ by Proposition~\ref{prop:SymForm} and the isomorphisms $\Sigma(W)\cong\NSym^S$ and $\Sigma^*(W)\cong\QSym^S$.

There is a linearly independent set $\{m_{\lambda}^S: \lambda\in \Pi(W,S)\}$ in $\QSym^S$, where
\[m_{\lambda}^S := \sum_{I\in\lambda } M_I^S, \qquad \forall \lambda\in\Pi(W,S).\]

\begin{proposition}\label{prop:SymK}
$\Sym^S_\KK$ has dual $\KK$-bases $\left\{ h_\lambda^S:\lambda\in\Pi(W,S) \right\}$ and $\left\{m_\lambda^S: \lambda\in \Pi(W,S) \right\}$.
\end{proposition}

\begin{proof}
Let $I,I',J\subseteq S$ with $I\sim I'$.
One shows $\langle \bh_I^S, h_J^S \rangle = \langle \bh_{I'}^S, h_J^S \rangle$ similarly as the proof of Proposition~\ref{prop:Sym}.
Thus $h_J^S$ lies in the $\ZZ$-span of $\{m_\lambda^S: \lambda \in\Pi(W,S)\}$.  
It follows that the free $\ZZ$-basis $\{h_\lambda: \lambda\in\Pi(W,S)\}$ for $\Sym^S$ is contained in the $\ZZ$-span of the linearly independent set $\{ m_\lambda: \lambda\in\Pi(W,S) \}$. 
This implies that $\left\{m_\lambda^S: \lambda\in \Pi(W,S) \right\}$ is a $\KK$-basis for $\Sym^S_\KK$.
If $\lambda,\mu\in\Pi(W,S)$ and $I\in\lambda$ then 
\[ \langle h_\lambda^S, m_\mu^S \rangle = \langle H_I, \sum_{J\in\mu} M_J \rangle = \delta_{\lambda,\mu}.\]
Thus $\left\{ h_\lambda^S:\lambda\in\Pi(W,S) \right\}$ and $\{ m^S_\lambda: \lambda\in\Pi(W,S) \}$ are dual $\KK$-bases for $\Sym^S_\KK$.
\end{proof}

Now we have the following isomorphisms of vector spaces:
\[\begin{matrix}
\NSym^S_\KK & \cong & \KK\overline Q, \\
\bh_I^S & \mapsto & H_I,
\end{matrix} \qquad
\begin{matrix}
\Sym^S_\KK & \cong & \KK \overline L, \\ 
h_\lambda^S & \mapsto & h_\lambda,
\end{matrix} \qquad
\begin{matrix}
\Sym^S_\KK & \cong & \KK \overline L^*, \\
m_\lambda^S & \mapsto & m_\lambda,
\end{matrix} \qquad
\begin{matrix}
\QSym^S_\KK & \cong & \KK\overline Q^*, \\
M_I^S & \mapsto & M_I.
\end{matrix} \] 
Then one can verify the following commutative diagram.
\begin{equation}\label{eq:bottom} 
\xymatrix @R=15pt @C=20pt {
\NSym^S_\KK \ar@{->>}[r]^{\chi} \ar@{<->}^\wr[d] & \Sym^S_\KK \ar@{=}[r] \ar@{<->}^\wr[d] & \Sym^S_\KK  \ar@{^(->}^\imath[r]  \ar@{<->}^\wr[d] & \QSym^S_\KK \ar@{<->}^\wr[d] \\
\KK\overline Q \ar@{->>}[r]^{\supp} & \KK \overline L \ar@{<->}^{\phi}[r] & \KK \overline L^* \ar@{^(->}^{\supp*}[r] & \KK\overline Q^*
}\end{equation}
Combining \eqref{eq:top} and \eqref{eq:bottom} one proves Proposition~\ref{prop:DiagramIso}.

\begin{remark}\label{rem:Sym}
(i) By Geck and Pfeiffer~\cite[Chapter 2 and Appendix A]{GeckPfeiffer}, the conjugacy class of a parabolic subgroup $W_I$ of $W$ is determined by the type of the Coxeter system $(W_I, I)$ except that in type $D_{2m}$ there are two conjugacy classes of parabolic subgroups of type $A_{k_1}\times\cdots\times A_{k_r}$ with all $k_i$ odd such that $\sum_{i=1}^r (k_i+1) = 2m$ and in type $E_7$ there are two conjugacy classes of parabolic subgroups of type $A_1\times A_1\times A_1$, $A_1\times A_3$, and $A_5$.

\noindent (ii)
One can obtain dual orthogonal bases $\{q_\lambda^S:\lambda\in\Pi(W,S)\}$ and $\{p_\lambda^S:\lambda\in\Pi(W,S)\}$ for $\Sym^S_\KK$ using work of Aguiar and Mahajan~\cite[\S5.7]{AguiarMahajan}.
In fact, if $\lambda\in\Pi(W,S)$ and $I\in\lambda$ then 
\[ p_\lambda^S := \sum_{J\subseteq I} |\lambda(J)| m_{\lambda(J)} \]
where $\lambda(J)$ is the equivalent class in $\Pi(W,S)$ containing $J$ and $|\lambda(J)|$ is the quotient of $|W^{J^c}|$ by the size of the conjugacy class of the parabolic subgroup $W_{J^c}$ of $W$.
In type A this recovers the well-known power sum basis $\{p_\lambda\}$ for $\Sym_\KK$.

\noindent(iii) 
Though $\{m_\lambda^S:\lambda\in\Pi(W,S)\}$ is a $\KK$-basis for $\Sym^S_\KK$, we will show in type B that it is not a free $\ZZ$-basis for $\Sym^S$ because the $\ZZ$-span of $\{m_\lambda^S:\lambda\in\Pi(W,S)\}$ strictly contains $\Sym^S$.
\end{remark}

\section{Type A}\label{sec:A}

In this section we recover some well-known results in type A from our results in Section~\ref{sec:General}.

\subsection{Malvenuto--Reutenauer algebra and descent algebra}\label{sec:MRA}

The symmetric group $\SS_n$ consists of all permutations of the set $[n]:=\{1,2,\ldots,n\}$ and is generated by $\{s_1,\ldots,s_{n-1}\}$ where $s_i$ is the adjacent transposition $(i,i+1)$. This gives the finite irreducible Coxeter system of type $A_{n-1}$, whose Coxeter diagram is drawn below.
\[ \xymatrix @C12pt{
s_1 \ar@{-}[r] & s_2 \ar@{-}[r] & s_3 \ar@{-}[r] & \cdots \ar@{-}[r] & s_{n-2} \ar@{-}[r] & s_{n-1}
}  \]
Let $w$ be a permutation in $\SS_n$. The \emph{one-line notation} for $w$ is the word $w(1)\cdots w(n)$.
The descent set of $w$ is $\{ s_i: i\in[n-1],\ w(i)>w(i+1)\}$.
The length of $w$ equals $\inv(w)$, where 
\begin{equation}\label{eq:inv}
\inv(a_1\cdots a_n):=\#\{(i,j):1\le i<j\le n,\ a_i>a_j\}.
\end{equation}

It is often convenient to use compositions of $n$  to index subsets of the generating set $\{s_1,\ldots,s_{n-1}\}$. A \emph{composition} is a sequence $\alpha=(\alpha_1,\ldots,\alpha_\ell)$ of positive integers  $\alpha_1,\ldots,\alpha_\ell$. The \emph{length} of $\alpha$ is $\ell(\alpha):=\ell$ and the \emph{size} of $\alpha$ is $|\alpha|:=\alpha_1+\cdots+\alpha_\ell$. If the size of $\alpha$ is $n$ then say $\alpha$ is a composition of $n$ and write $\alpha\models n$. 
The only composition of $n=0$ is the \emph{empty composition} $\alpha=\varnothing$.
Assume $n\ge1$. The \emph{descent set} of $\alpha$ is
\[
D(\alpha):=\{\alpha_1,\alpha_1+\alpha_2,\ldots,\alpha_1+\cdots+\alpha_{\ell-1}\}.
\] 
The map $\alpha\mapsto D(\alpha)$ is a bijection between compositions of $n$ and subsets of $[n-1]$. The \emph{complement} of $\alpha$ is the composition $\alpha^c$ of $n$ with $D(\alpha^c)=[n-1]\setminus D(\alpha)$. 

The parabolic subgroup $\SS_\alpha\cong \SS_{\alpha_1}\times\cdots\times\SS_{\alpha_\ell}$ is generated by $\{s_i:i\in D(\alpha^c)\}$. 
The set of minimal representatives of left $\SS_\alpha$-cosets in $\SS_n$ is $\SS^\alpha:=\{w\in \SS_n: D(w)\subseteq D(\alpha)\}$.

A \emph{ribbon} is a connected skew Young diagram without $2\times2$ boxes. A composition $\alpha=(\alpha_1,\ldots,\alpha_\ell)$ can be identified with a ribbon whose rows have length $\alpha_1,\ldots,\alpha_\ell$ from bottom to top. See the example below.

{\scriptsize \[ \begin{matrix} 
 \ydiagram{3+1,3+1,1+3,2} & \qquad\qquad\qquad\qquad &  \ydiagram{1+3,1+1,2,1} \\ \\
\alpha=(2,3,1,1) & & \alpha^c=(1,2,1,3)
\end{matrix} \] }

If $\alpha=(\alpha_1,\ldots,\alpha_\ell)$ and $\beta=(\beta_1,\ldots,\beta_k)$ are two nonempty compositions then define
\begin{eqnarray*}
\alpha\cdot\beta &:=& (\alpha_1,\ldots,\alpha_\ell,\beta_1,\ldots,\beta_k) \qand \\
\alpha\rhd\beta &:=& (\alpha_1,\ldots,\alpha_{\ell-1},\alpha_\ell+\beta_1,\beta_2,\ldots,\beta_k).
\end{eqnarray*}
Let $\alpha\cdot\varnothing := \alpha$ and $\varnothing\cdot\beta:=\beta$. 
Also set $\alpha\rhd\varnothing$ and $\varnothing\rhd\beta$ undefined.
We write $\alpha\cleq\beta$ if $\alpha$ and $\beta$ are compositions of the same size and $D(\alpha)\subseteq D(\beta)$, that is, $\alpha$ is refined by $\beta$. 

The \emph{Malvenuto--Reutenauer algebra}~\cite{MR} is a self-dual graded Hopf algebra whose underlying space is the free $\ZZ$-module $\ZZ\SS:=\bigoplus_{n\ge0} \ZZ\SS_n$ with a basis $\SS:=\bigsqcup_{\,n\geq0}\SS_n$. We will first review its product and coproduct, and then recover them from the linear maps defined in Section~\ref{sec:Coxeter}. 

Let $a=a_1\cdots a_n\in\ZZ^n$ be a word of $n$ integers. Let $i$ and $j$ be two nonnegative integers. We write $a[i,j]:=a_i\cdots a_j$ and denote by $a|[i,j]$ the subword of $a$ obtained by keeping only the letters whose absolute values belong to the interval $[i,j]$. In case $i>j$ we define both $a[i,j]$ and $a|[i,j]$ to be the empty word $\varnothing$.

The \emph{standardization} $\st(a)$ of the word $a$ is the unique permutation $w\in\SS_n$ such that 
\begin{equation}\label{def:st}
w(i)<w(j) \Leftrightarrow a_i \le a_j \quad\text{whenever}\quad 1\le i<j\le n.
\end{equation}
One can obtain $\st(a)$ by reading the letters of $a$ from smallest to largest, breaking up ties from left to right. For example, one has $\st (3223625) =4125736$. 

We identify a set of permutations with the sum of its elements inside $\ZZ\SS$. 
Define
\[ u\shuffle v:= \left\{w\in\SS_{m+n}: w|[1,m]=u,\ \st(w|[m+1,m+n])=v \right\} \qand\]
\[ \unshuffle u := \sum_{0\leq i\leq m} \st(u[1,i])\otimes \st(u[i+1,m])\]
for all $u\in\SS_m$ and $v\in\SS_n$.
This gives a graded Hopf algebra $(\ZZ\SS,\shuffle,\unshuffle)$.
The product $\shuffle$ is often called the \emph{(shifted) shuffle product}. For example, one has 
\[ 21\shuffle {\color{red}12} = 21{\color{red}34} + 2{\color{red}3}1{\color{red}4} + {\color{red}3}21{\color{red}4} + 2{\color{red}34}1 + {\color{red}3}2{\color{red}4}1 + {\color{red}34}21 \qand \]
\[ \unshuffle(2431)=\varnothing\otimes 2431+1\otimes 321+12\otimes 21+132\otimes 1+2431\otimes\varnothing.\]
There is another graded Hopf algebra $(\ZZ\SS,\Cup,\Cap)$ defined by  
\[ u\Cup v := \left\{\,w\in\SS_{m+n}:\st(w[1,m])=u,\ \st(w[m+1,n])=v \,\right\},\]
\[ \Cap\, u := \sum_{0\leq i\leq m} u|[1,i]\otimes\st(u|[i+1,m]),\]
for all $u\in\SS_m$ and $v\in\SS_n$. 
For example, one has 
\[ 21 \Cup {\color{red}12}=21{\color{red}34}+31{\color{red}24}+32{\color{red}14}+41{\color{red}23}+42{\color{red}13}+43{\color{red}12} \qand \]
\[ \Cap\, (2431)=\varnothing\otimes2431+1\otimes132+21\otimes21+231\otimes1+2431\otimes\varnothing.\]

Let $u\in\SS_m$ and $v\in \SS_n$. An element $(u,v)\in\SS_m\times\SS_n$ can be identified with a permutation $u\times v\in\SS_{m+n}$ whose one-line notation is $u(1)\cdots u(m)(m+v(1))\cdots(m+v(n))$. This gives an isomorphism between $\SS_m\times\SS_n$ and the parabolic subgroup $\SS_{m,n}$ of $\SS_{m+n}$, and hence an embedding $\SS_m\times\SS_n\hookrightarrow\SS_{m+n}$. Let $w\in\SS_{m+n}$. Denote by $\unshuffle_m w$ the $m$-th term in the coproduct $\unshuffle w$, and similarly for $\Cap_m w$.
Using the linear maps defined in Definition~\ref{def:MapsZW}, with $S=\{s_1,\ldots,s_{m+n-1}\}$ and $I=S\setminus\{s_m\}$, one can check that
\[ u\shuffle v = \,\bar\mu_I^S(u\times v),\qquad \unshuffle_m w = \,\bar\rho_I^S(w), \]
\[ u\Cup v = \mu_I^S(u\times v),\qquad \Cap_m \,w = \rho_{I}^S(w).\]
Thus Proposition~\ref{prop:IJS} implies the associativity and coassociativity of $\shuffle, \unshuffle, \Cup, \Cap$. Moreover, Proposition~\ref{prop:MR} implies that the two graded Hopf algebras $(\ZZ\SS,\shuffle,\unshuffle)$ and $(\ZZ\SS,\Cup,\Cap)$ are isomorphic to each other via the map $w\mapsto w^{-1}$ for all $w\in\SS$, and dual to each other via the bilinear form $\langle u,v\rangle=\delta_{u,v}$, $\forall u,v\in\SS$. 

\begin{remark}
Our results in Section~\ref{sec:General} do not imply that $(\ZZ\SS,\shuffle,\unshuffle)$ and $(\ZZ\SS,\Cup,\Cap)$ are bialgebras. One needs to verify the compatibility of the products and coproducts directly. In fact, we will obtain analogues of $\shuffle$ and $\unshuffle$ in type B and D, which are not compatible with each other, and similarly for analogues of $\Cup$ and $\Cap$ in type B and D.
\end{remark}

Let $\Sigma(\SS):=\bigoplus_{n\geq0}\Sigma(\SS_n)$ where $\Sigma(\SS_n)$ is the free $\ZZ$-module with a basis consisting of descent classes 
\[ D_\alpha(\SS_n):=\{w\in\SS_n:D(w)=D(\alpha)\}, \qquad\forall \alpha\models n.\]
One has an embedding $\imath:\Sigma(\SS)\hookrightarrow \ZZ\SS$ by inclusion. If $\alpha\models m$ and $\beta\models n$ then  
\[ D_\alpha(\SS_m) \Cup D_\beta(\SS_n) = D_{\alpha\cdot\beta}(\SS_{m+n}) + D_{\alpha\rhd\beta}(\SS_{m+n})\]
by Proposition~\ref{prop:IndDes}, where the last term is treated as zero when $\alpha\rhd\beta$ is undefined.

Let $\Sigma^*(\SS):=\bigoplus_{n\ge0} \Sigma^*(\SS_n)$ where $\Sigma^*(\SS_n)$ is the dual of $\Sigma(\SS_n)$ with a dual basis $\{D^*_\alpha(\SS_n):\alpha\models n\}$. If $w\in\SS_n$, $\alpha\models n$, and $D(w)=D(\alpha)$, then define $D^*_w(\SS_n):=D^*_\alpha(\SS_n)$. Sending $w\in\SS_n$ to $D^*_w(\SS_n)$ gives a surjection $\chi:\ZZ\SS\twoheadrightarrow\Sigma^*(\SS)$ dual to the embedding $\imath:\Sigma(\SS)\hookrightarrow \ZZ\SS$.

Recall from \S\ref{sec:Coxeter} that $\Lambda(\SS_n)$ is the $\ZZ$-span of $\Lambda_\alpha(\SS_n):=\chi'(D_\alpha(\SS_n))$ for all $\alpha\models n$, where $\chi':=\chi\circ(\ )^{-1}$.
Let $\Lambda(\SS):=\bigoplus_{n\ge0}\Lambda(\SS_n)$. 
If $\alpha\models m$ and $\beta\models n$ then define 
\[ \langle \Lambda_\alpha(\SS_m),\Lambda_\beta(\SS_n)\rangle := \begin{cases}
\#\{w\in\SS_n: D(w^{-1}) = D(\alpha),\ D(w)=D(\beta) \}, & \text{if } m=n, \\
0, & \text{if } m\ne n.
\end{cases} \]
By Proposition~\ref{prop:SymForm}, this gives a well-defined nondegenerate symmetric bilinear form on $\Lambda(\SS)$ such that $\imath:\Lambda(\SS)\hookrightarrow\Sigma^*(\SS)$ and $\chi':\Sigma(\SS)\twoheadrightarrow\Lambda(\SS)$ are dual to each other.

Finally, restricting the commutative diagram \eqref{eq:DiamondCoxCat} of representations of categories to type A gives the following commutative diagram of graded Hopf algebras. 
\begin{equation}\label{eq:DiamondCoxA}
\xymatrix @R=16pt @C=7pt {
 & \ZZ\SS \ar@{->>}[rd]^{\chi'} \\
 \Sigma(\SS) \ar@{^(->}[ru]^{\imath} \ar@{->>}[rd]_{\chi'} & & \Sigma^*(\SS) \ar@{<-->}[ll]^{\txt{\small dual}} \\
 & \Lambda(\SS) \ar@{^(->}[ru]_\imath } 
\end{equation}
Reflecting this diagram across the vertical line through $\ZZ\SS$ and $\Lambda(\SS)$ gives a dual diagram. 

\subsection{Free quasisymmetric functions and related results}\label{sec:FQSymA}

Let $(W,S)$ be the Coxeter system of type $A_{n-1}$, where $W=\SS_n$ and $S=\{s_1,\ldots,s_{n-1}\}$. 
Let $E=\mathbb R^n$ be a Euclidean space whose standard basis is $\{e_1,\ldots,e_n\}$. 
One can realize $\SS_n$ as a reflection group of $E$ whose root system $\Phi=\Phi^+\sqcup\Phi^-$ is the disjoint union of $\Phi^+=\{e_j-e_i:1\leq i<j\leq n\}$ and $\Phi^-=-\Phi^+$. The set $\Delta=\{e_{i+1}-e_i:1\leq i\leq n-1\}$ of simple roots corresponds to the generating set $S$ of simple reflections. 

A parset $P$ in the root system $\Phi$ is equivalent to a partial order $P$ on $[n]$ defined by $i<_P j$ if $e_j-e_i\in P$. 
For $w\in\SS_n$, the parset $w\Phi^+$ is equivalent to the total order (linear order) $P_w:w(1)<_w w(2)<_w\cdots<_w w(n)$. The Jordan-H\"older set $\mathcal L(P)$ consists of all linear extensions of the partial order $P$. 
A $P$-partition is a function $f:[n]\to \ZZ$ satisfying 
$i<_P j\Rightarrow f(i)\leq f(j)$ and $(i<_P j,\ i>j) \Rightarrow f(i)<f(j).$ 

Let $\bX=\{\bx_i:i\in\ZZ\}$ be a set of noncommutative variables. For each $w\in\SS_n$ the generating function of $\mathcal A(P_w)$ is 
\[\bF_w:=\sum_{\substack{
f(w(1))\leq\cdots\leq f(w(n)) \\
i\in D(w)\Rightarrow f(w(i))<f(w(i+1)) }}
\bx_{f(1)}\cdots\bx_{f(n)}.\]
Let $\FQSym=\FQSym(\bX)$ be the free $\ZZ$-module with a basis $\{\bF_w:w\in\SS\}$. If $u\in\SS_m$ and $v\in\SS_{n-m}$ then $u\times v\in W_I\cong\SS_m\times\SS_{n-m}$ where $I=S\setminus\{s_m\}$, and thus Proposition~\ref{prop:shuffleW} implies
\[ \bF_u\cdot \bF_v = \bF_{(u\times v)\Phi^+_I}=\sum_{w\in u\shuffle v} \bF_w.\]
Thus $\FQSym$ is a graded algebra isomorphic to $(\ZZ\SS,\shuffle)$. It also has a coproduct defined by
\begin{equation}\label{eq:Delta bFw}
\Delta \bF_w :=\sum_{1\le i\le n} \bF_{\st(w[1,i])}\otimes \bF_{\st(w[i+1,n])},\quad\forall w\in\SS_n.
\end{equation}
Then $\FQSym$ becomes a self-dual graded Hopf algebra isomorphic to the Malvenuto--Reutenauer algebra $(\ZZ\SS,\shuffle,\unshuffle)$ via $\bF_w\mapsto w$ for all $w\in\SS$.

\begin{remark}
Replacing $\bX$ in the above definition of $\FQSym$ with any totally ordered countably infinite set $\mathbf Z$ of noncommutative variables gives a Hopf algebra $\FQSym(\mathbf Z)$ isomorphic to $\FQSym$. In fact, Duchamp, Hivert, and Thibon~\cite{NCSF_VI} used  $\bX_{>0}:=\{\bx_i:i\in\ZZ_{>0}\}$ when they introduced $\FQSym$. We will also use $\bX_{\ge0}:=\{\bx_i:i\in\ZZ_{\ge0}\}$ later.
\end{remark}

\begin{remark}
One can define~\cite[\S8.1]{GrinbergReiner} the coproduct of $\FQSym$ by $\Delta\bF:=\bF(\bX+\bY)$ for all $\bF\in\FQSym$, where $\bX+\bY$ is the union of two totally ordered countable sets of noncommutative variables $\bX=\{\bx_i:i\in\ZZ\}$ and $\bY=\{\by_i:i\in\ZZ\}$ with $\bx_i<\by_j$ and $\bx_i\by_j=\by_j\bx_i$ for all $i,j\in\ZZ$. To avoid technicality we simply use \eqref{eq:Delta bFw} as the definition of the coproduct of $\FQSym$, and will similarly deal with this issue in type B and D.
\end{remark}

Applying Proposition~\ref{prop:InverseP} to $\alpha=e_i-e_j$ for all pairs $i,j\in[n]$ satisfying $1\le i<j\le n$ shows that $f\in\mathcal A(w^{-1}\Phi^+)$ if and only if $\st(f)=w$ for all $w\in\SS_n$. Thus one has
\begin{equation}\label{eq:bs}
 \bs_w:=\bF_{w^{-1}} =\sum_{f\in\ZZ^n:\, \st(f)=w} \bx_f, \quad \forall w\in\SS_n.
\end{equation}
This leads to a Hopf algebra isomorphism $\FQSym\cong (\ZZ\SS,\Cup,\Cap)$ by $\bs_w\mapsto w$ for all $w\in\SS$.

The graded Hopf algebra $\NSym$ of noncommutative symmetric functions is the free associative algebra $\ZZ\langle\bh_1,\bh_2,\ldots\rangle$ generated by $\bh_k=\sum_{i_1\leq\cdots\leq i_k}\bx_{i_1}\cdots \bx_{i_k}$ for all $k\ge1$. It has a free $\ZZ$-basis consisting of the \emph{complete homogeneous noncommutative symmetric functions} $\bh_\alpha:=\bh_{\alpha_1}\cdots \bh_{\alpha_\ell}$ and another basis consisting of the \emph{noncommutative ribbon schur functions}
\[ \bs_\alpha := \sum_{\beta\cleq\alpha}(-1)^{\ell(\alpha)-\ell(\beta)}\bh_\beta\]
where $\alpha=(\alpha_1,\ldots,\alpha_\ell)$ runs through all compositions.
Note that $\bs_\varnothing=\bh_\varnothing:=1$. 

The product of $\NSym$ is determined by the following two equivalent formulas
\[ \bh_\alpha\bh_\beta=\bh_{\alpha\cdot\beta} \qand \bs_\alpha\bs_\beta=\bs_{\alpha\cdot\beta}+\bs_{\alpha\rhd\beta} \]
for all compositions $\alpha$ and $\beta$, where the last term is treated as zero when $\alpha\rhd\beta$ is undefined. 
The coproduct of $\NSym$ is defined by $\Delta \bh_k=\sum_{i=0}^k \bh_i\otimes\bh_{k-i}$, where $\bh_0:=1$. A more explicit coproduct formula was provided in our earlier work~\cite{H0Tab}.

Given a composition $\alpha$, a \emph{tableau} $\tau$ of shape $\alpha$ is a filling of the ribbon diagram of $\alpha$ with integers. Reading these integers from the bottom row to the top row and proceeding from left to right within each row gives the \emph{reading word} $w(\tau)$ of $\tau$. We call $\tau$ a \emph{semistandard tableau of shape $\alpha$} if each row is weakly increasing from left to right and each column is strictly increasing from top to bottom. One sees that $\bs_\alpha$ is the sum of $\bx_{w(\tau)}$ for all semistandard tableaux $\tau$ of shape $\alpha$. Moreover, each $f\in\ZZ^n$ corresponds to a unique tableau $\tau$ of shape $\alpha$ such that $w(\tau)=f$; one has $D(\st(f))=D(\alpha)$ if and only if $\tau$ is semistandard. Thus $\bs_\alpha$ equals the sum of $\bs_w$ for all $w$ in the descent class of $\alpha$. It follows that $\NSym$ is a Hopf subalgebra of $\FQSym$ isomorphic to the descent algebra $\Sigma(\SS)$.

Let $X_{>0}=\{x_1,x_2,\ldots\}$ be a set of commutative variables. The graded Hopf algebra of quasisymmetric functions is defined as $\QSym := \chi(\FQSym)$, where $\chi$ replaces $\bX$ with $X_{>0}$. If $\alpha=(\alpha_1,\ldots,\alpha_\ell)\models n$ then one has the \emph{monomial quasisymmetric function}
\[ M_\alpha:=\sum_{0< i_1<\cdots< i_\ell} x_{i_1}^{\alpha_1}\cdots x_{i_\ell}^{\alpha_\ell} \]
and the \emph{fundamental quasisymmetric function}
\[ F_\alpha:=\sum_{\alpha\cleq\beta} M_\beta = \sum_{ \substack{0<i_1\le \cdots\le i_n \\ j\in D(\alpha)\Rightarrow i_j<i_{j+1} }} x_{i_1} \cdots x_{i_n}.\]
Ones sees that $F_w:=\chi(\bF_w)$ equals $F_\alpha$ if $w$ is in the descent class of $\alpha$. Hence $\QSym$ has two bases $\{M_\alpha\}$ and $\{F_\alpha\}$ where $\alpha$ runs through all compositions. The product and coproduct of $F_\alpha$ can be easily obtained by applying $\chi$ to $\FQSym$. Thus $\QSym$ is isomorphic to the dual $\Sigma^*(\SS)$ of the descent algebra $\Sigma(\SS)$. The self-duality of $\FQSym$ induces the duality between $\QSym$ and $\NSym$ via $\langle \bh_\alpha,M_\beta\rangle =\langle \bs_\alpha,F_\beta\rangle:=\delta_{\alpha,\beta}$ for all compositions $\alpha$ and $\beta$.

A \emph{partition} $\lambda$ of $n$, denoted by $\lambda\vdash n$, is a weakly decreasing sequence of positive integers $\lambda=(\lambda_1,\ldots,\lambda_\ell)$ such that its \emph{size} $|\lambda|:=\lambda_1+\cdots+\lambda_\ell=n$. Denote by $\lambda(\alpha)$ be the partition obtained from a composition $\alpha$ by rearranging its parts. 
One sees that two parabolic subgroups $\SS_\alpha$ and $\SS_\beta$ are conjugate if and only if $\lambda(\alpha) = \lambda(\beta)$.

The graded Hopf algebra $\Sym:=\chi(\NSym)$ of symmetric functions is the $\ZZ$-span of the \emph{ribbon Schur functions} $s_\alpha:=\chi(\bs_\alpha)$ for all compositions $\alpha$, where $\chi:\NSym\twoheadrightarrow\Sym$ is a surjection of Hopf algebras defined by replacing $\bX$ with $X_{>0}$.
One sees that $\Sym$ is isomorphic to $\Lambda(\SS)$ defined in Section~\ref{sec:Coxeter}.
If $\lambda$ is a partition then the \emph{complete homogeneous symmetric function} $h_\lambda:=\chi(\bh_\alpha)$ is well defined for any composition $\alpha$ with $\lambda(\alpha)=\lambda$. 
Hence $\Sym$ has a free $\ZZ$-basis consisting of $h_\lambda$ for all partitions $\lambda$.
Another free $\ZZ$-basis for $\Sym$ consists of the \emph{monomial symmetric functions} $m_\lambda$ for all partitions $\lambda$, where $m_\lambda$ is the sum of $M_\alpha$ for all compositions $\alpha$ with $\lambda(\alpha)=\lambda$.
Hence $\Sym$ is a Hopf subalgebra of $\QSym$.
The product and coproduct of $\Sym$ can be obtained from $\QSym$ and $\NSym$ by Corollary~\ref{cor:Sym}.
The pairing between $\NSym$ and $\QSym$ induces a bilinear form on $\Sym$ such that $\langle h_\lambda,m_\lambda\rangle = \delta_{\lambda,\mu}$ for all partitions $\lambda$ and $\mu$.
With this bilinear form $\Sym$ becomes self-dual.
The inclusion $\imath: \Sym\hookrightarrow\QSym$ and the surjection $\chi:\NSym\twoheadrightarrow\Sym$ are also dual to each other.

In summary, one has the following commutative diagram of graded Hopf algebras isomorphic to \eqref{eq:DiamondCoxA}.
\begin{equation}\label{eq:DiamondFQSymA}
\DiamondFQSym{}
\end{equation}

\subsection{Representation theory}\label{sec:GroupA}

We first briefly review the (complex) representation theory of $\SS_n$; see, e.g., Stanley~\cite[Chapter 7]{EC2} for details.
The group algebra $\CC\SS_n$ is semisimple. The simple $\CC\SS_n$-modules $S_\lambda$ are indexed by partitions $\lambda\vdash n$. The Grothendieck group $G_0(\CC\SS_\bullet)$ of the tower $\CC\SS_0\hookrightarrow \CC\SS_1\hookrightarrow \CC\SS_2\hookrightarrow\cdots$ of algebras is a graded Hopf algebra, whose product and coproduct are given by induction and restriction along natural embeddings $\SS_m\times\SS_n\hookrightarrow \SS_{m+n}$. This Hopf algebra is self-dual under the bilinear form defined by $\langle S_\lambda,S_\mu\rangle :=\delta_{\lambda,\mu}$ for all partitions $\lambda$ and $\mu$. 
The \emph{Frobenius characteristic} is a graded Hopf algebra isomorphism $G_0(\CC\SS_\bullet) \cong \Sym$ defined by sending $S_\lambda$ to the Schur function $s_\lambda$ for all partitions $\lambda$. 

Now let $H_n(0)$ be the $0$-Hecke algebra of the Coxeter system $(W,S)$ of type $A_{n-1}$, where $W=\SS_n$ and $S=\{s_1,\ldots,s_{n-1}\}$. The generators $\pi_1,\ldots,\pi_{n-1}$ for $H_n(0)$ can be interpreted as the \emph{bubble-sorting operators}\,\!: $\pi_i$ swaps adjacent positions $a_i$ and $a_{i+1}$ in a word $a_1\cdots a_n\in\ZZ^n$ if $a_i<a_{i+1}$, or fixes the word otherwise. 

The projective indecomposable $H_n(0)$-modules and simple $H_n(0)$-modules are given by $\P_\alpha:=\P_I^S$ and $\C_\alpha:=\C_I^S$, respectively,  where $I=\{s_i:i\in D(\alpha)\}$, for all $\alpha\models n$. One can realize $\P_\alpha$ as the space of \emph{standard tableaux of ribbon shape} $\alpha$ with an appropriate $H_n(0)$-action~\cite{H0Tab}. 

The parabolic subalgebra $H_\alpha(0)$ of $H_n(0)$ is generated by $\{\pi_i:i\in [n-1]\setminus D(\alpha)\}$. One has an embedding $H_m(0)\otimes H_n(0)\cong H_{m,n}(0) \subseteq H_{m+n}(0)$ if $m$ and $n$ are nonnegative integers.

Associated with the tower of algebras $H_\bullet(0): H_0(0)\hookrightarrow H_1(0)\hookrightarrow H_2(0)\hookrightarrow \cdots$ are Grothendieck groups
\[ G_0(H_\bullet(0)):=\bigoplus_{n\geq0}G_0(H_n(0)) \qand K_0(H_\bullet(0)):=\bigoplus_{n\geq0}K_0(H_n(0)).\]
The Grothendieck groups $G_0(H_\bullet(0))$ and $K_0(H_\bullet(0))$ are both graded Hopf algebras, whose product $\odot$ and coproduct $\Delta$ are defined by 
\[ M \odot N := (M\otimes N) \uparrow\,_{H_{m,n}(0)}^{H_{m+n}(0)} \qand
 \Delta (M) := \sum_{0\le i\le m} M\downarrow\,_{H_{i,m-i}(0)}^{H_{m}(0)} \]
for all finitely generated (projective) modules $M$ and $N$ over $H_m(0)$ and $H_n(0)$, respectively.
Using the linear maps $\bar\mu_I^S:G_0(H_{W_I}(0))\to G_0(H_W(0))$ and $\bar\rho_I^S:G_0(H_W(0))\to G_0(H_{W_I}(0))$ defined in \eqref{def:IndResG0}, with $S=\{s_1,\ldots,s_{m+n-1}\}$ and $I=S\setminus\{s_m\}$, one has
\[ M\odot N = \bar\mu_I^S(M\otimes N) \qand \Delta_m (Q) := Q\downarrow\,_{H_{m,n}(0)}^{H_{m+n}(0)} = \bar\rho_I^S(Q) \]
where $Q$ is a finitely generated $H_{m+n}(0)$-module.
Also recall from Proposition~\ref{prop:IndResG0K0} that $\bar\mu_I^S$ and  $\bar\rho_I^S$ restrict to $\mu_I^S:K_0(H_{W_I}(0))\to K_0(H_W(0))$ and $\rho_I^S:K_0(H_W(0))\to K_0(H_{W_I}(0))$.
If $M$, $N$, and $Q$ are all projective then 
\[ M\odot N = \mu_I^S(M\otimes N) \qand \Delta_m (Q) = \rho_I^S(Q).\]


Specializing the natural isomorphisms $\Ch:\mathcal G_0\to \mathcal QSym$ and $\bch:\mathcal K_0\to \mathcal NSym$ of representations of the category $\mathcal Cox$ and its dual $\mathcal Cox^{op}$ to type A gives 
\[\begin{matrix}
\Ch: & G_0(H_\bullet(0)) & \xrightarrow\sim & \QSym & \qand & \bch: & K_0(H_\bullet(0)) & \xrightarrow\sim & \NSym \\
& \C_\alpha & \mapsto & F_\alpha & & & \P_\alpha & \mapsto & \bs_\alpha.
\end{matrix}\]
These two characteristic maps in type A were studied by Krob and Thibon~\cite{KrobThibon} and others.  They are isomorphisms of graded Hopf algebras by Theorem~\ref{thm:GrH0}.

Lastly, Proposition~\ref{prop:PtoC} implies a result of Krob and Thibon~\cite{KrobThibon}: $\Ch(\P_\alpha)=s_\alpha$, $\forall \alpha\models n$.

\section{Type B}\label{sec:B}

In this section we apply our results in Section~\ref{sec:General} to type B and get some new results.

\subsection{Malvenuto--Reutenauer algebra and descent algebra of type B}\label{sec:MRBA}

A \emph{signed permutation} $w$ of $[n]$ is a bijection of the set $[\pm n]:=\{\pm1,\ldots,\pm n\}$ onto itself such that $w(-i)=-w(i)$ for all $i\in[\pm n]$. 
We set $w(0)=0$. Since $w$ is determined by where it sends $1,\ldots,n$, one can identify $w$ with $w(1)\cdots w(n)$ or $[w(1),\ldots,w(n)]$, where a negative integer $-k<0$ is often written as $\bar k$.
This is the \emph{window notation} of $w$.
 
The \emph{hyperoctahedral group} $\B_n$ consists of all signed permutations of $[n]$.
It is generated by $s_0=s_0^B:=\bar12\cdots n$ and $s_i:=[1,\ldots,i-1,i+1,i,i+2,\ldots,n]$ for all $i\in[n-1]$.
The parabolic subgroup of $\B_n$ generated by $s_1,\ldots,s_{n-1}$ is isomorphic to $\SS_n$.
The pair $(\B_n,\{s_0,\ldots,s_{n-1}\})$ is the finite irreducible Coxeter system of type $B_n$ whose Coxeter diagram is below.
\[ \xymatrix @C=12pt{
s_0 \ar@{=}[r] & s_1 \ar@{-}[r] & s_2 \ar@{-}[r] & \cdots \ar@{-}[r] & s_{n-2} \ar@{-}[r] & s_{n-1}
} \]

Let $w\in \B_n$. The descent set of $w$ is $\{s_i: i\in\{0,1,\ldots,n-1\},\ w(i)>w(i+1)\}$. Given a word $a=(a_1,\ldots,a_n)\in\ZZ^n$, we define $\neg(a) := \#\Neg(a)$ and $\nsp(w):=\#\Nsp(w)$ where
\begin{eqnarray*}
\Neg(a) &:=& \{i\in[n]:a_i<0\} \qand \\
\Nsp(a) &:=& \{(i,j):1\leq i<j\leq n,\ a_i+a_j<0\}.
\end{eqnarray*}
Then the length of $w$ equals $\inv(w)+\neg(w)+\nsp(w)$. Note that $\inv(w)$ is defined by \eqref{eq:inv}.

A \emph{pseudo-composition of $n$} is a sequence $\alpha=(\alpha_1,\ldots,\alpha_\ell)$ of integers such that $\alpha_1\geq0$, $\alpha_2,\ldots,\alpha_\ell>0$, and the \emph{size} $|\alpha|:=\alpha_1+\cdots+\alpha_\ell$ equals $n$. This is denoted by $\alpha\modelsB n$. The \emph{length} of $\alpha$ is $\ell(\alpha):=\ell$. The \emph{descent set} of $\alpha$ is $D(\alpha):=\{\alpha_1,\alpha_1+\alpha_2,\ldots,\alpha_1+\cdots+\alpha_{\ell-1}\}.$ The map $\alpha\mapsto D(\alpha)$ is a bijection between pseudo-compositions of $n$ and the subsets of $\{0,1,\ldots,n-1\}$. Analogously to type A, it is convenient to index subsets of the generating set $\{s_0,s_1,\ldots,s_{n-1}\}$ of $\B_n$ by pseudo-compositions of $n$. 

If $\alpha=(\alpha_1,\ldots,\alpha_\ell)\modelsB n$ then the \emph{parabolic subgroup} $\B_\alpha\cong\B_{\alpha_1}\times\SS_{\alpha_2}\times\cdots\times\SS_{\alpha_\ell}$ of $\B_n$ is generated by $\{s_i:0\le i\le n-1,\ i\notin D(\alpha)\}$.
The minimal representatives for left $\B_\alpha$-cosets in $\B_n$ form the set $(\B)^\alpha:=\{w\in\B_n:D(w)\subseteq D(\alpha)\}$.

A pseudo-composition $\alpha=(\alpha_1,\ldots,\alpha_\ell)$ can be identified with a \emph{pseudo-ribbon} in the following way. If $\alpha_1>0$ then $\alpha$ can be viewed as a composition which corresponds to a ribbon, and we draw an extra 0-box to the left of the bottom row of this ribbon. If $\alpha_1=0$ then $(\alpha_2,\ldots,\alpha_\ell)$ corresponds to a ribbon, and we draw an extra 0-box below the leftmost column of this ribbon. Some examples are below.

{\scriptsize \[ {(2,3,1,1) \quad \leftrightarrow \quad} \raisebox{-12pt}{\young(::::\hfill,::::\hfill,::\hfill\hfill\hfill,0\hfill\hfill)}
\qquad\qquad\qquad 
(0,2,3,1,1) \quad \leftrightarrow \quad \raisebox{-15pt}{\young(:::\hfill,:::\hfill,:\hfill\hfill\hfill,\hfill\hfill,0)} \] }

If $\alpha=(\alpha_1,\ldots,\alpha_\ell)\modelsB m\ge1$ and $\beta=(\beta_1,\ldots,\beta_k)\models n\ge1$ then define
\begin{eqnarray*}
\alpha\cdot\beta &:=& (\alpha_1,\ldots,\alpha_\ell,\beta_1,\ldots,\beta_k) \qand \\
\alpha\rhd\beta &:=& (\alpha_1,\ldots,\alpha_{\ell-1},\alpha_\ell+\beta_1,\beta_2,\ldots,\beta_k).
\end{eqnarray*}
Let $\alpha\cdot\varnothing:=\alpha$ and $\varnothing\cdot\beta:=\beta$.
Also set $\alpha\rhd\varnothing$ and $\varnothing\rhd\beta$ undefined.
If $\alpha$ and $\beta$ are pseudo-compositions of the same size and $D(\alpha)\subseteq D(\beta)$ then write $\alpha\cleq\beta$. 

We define the \emph{signed standardization} of a word $a\in \ZZ^n$, denoted by $\stB(a)$, to be the unique signed permutation $w\in \B_n$ such that $\Neg(a)=\Neg(w)$ and the following holds whenever $1\leq i<j\leq n$:
\begin{equation}\label{def:bst}
|w(i)|<|w(j)| \Leftrightarrow |a_i|<|a_j| \textrm{ or } a_i=a_j \geq 0 \textrm{ or } a_i=-a_j<0.
\end{equation}
Proposition~\ref{prop:bsB} provides an interpretation of the signed standardization by P-partition theory.

One can obtain $\stB(a)$ by reading the letters in $a$ in the increasing order of their absolute values, breaking up ties first from right to left for the negative letters and then from left to right for the positive ones, and finally inserting negative signs at the same positions as in $a$. For example, one has $\stB (2\bar43\bar2020\bar2) = 5\bar87\bar4162\bar3$.	

Let $a=a_1\cdots a_n\in\ZZ^n$. Suppose that $\Neg(a)=\{i_1,\ldots,i_k\}$, where $i_1<\cdots<i_k$, and $[n]\setminus\Neg(a)=\{j_1,\ldots,j_{n-k}\}$, where $j_1<\cdots<j_{n-k}$. We define 
\[ \hat a:=\overline{a_{i_k}}\cdots\overline{a_{i_1}}\ a_{j_1}\cdots a_{j_{n-k}}.\]
One sees that $\stB(a)=w$ if and only if $\Neg(a)=\Neg(w)$ and $\st(\hat a)=\hat w$.

An element $(u,v)\in \B_m\times\SS_n$ can be identified with a signed permutation $u\times v\in\B_{m+n}$ defined by the window notation $[u(1),\ldots,u(m),m+v(1),\ldots,m+v(n)]$. 
This gives an isomorphism between $\B_m\times\SS_n$ and the parabolic subgroup $\B_{m,n}$ of $\B_{m+n}$ generated by the set $\{s_i:0\le i\le m+n-1, i\ne m\}$.
Thus one has an embedding $\B_m\times\SS_n\hookrightarrow \B_{m+n}$. The set of minimal representatives for left $\B_{m,n}$-cosets in $\B_{m+n}$ is 
\[ (\B)^{m,n}:=\{ z\in\B_{m+n}: 0<z(1)<\cdots<z(m),\ z(m+1)<\cdots<z(m+n)\}.\]

\begin{proposition}\label{prop:BA}
Every element $w\in\B_{m+n}$ can be written uniquely as $w=(u\times v)z$ where $u\in \B_m$, $v\in \SS_n$, and $z^{-1}\in(\B)^{m,n}$. Moreover, one has 
\[\begin{matrix} 
u & = & w|[1,m],\quad & v & = & \st (\hat{w}|[m+1,m+n]),\\
u^{-1} & = & \stB(w^{-1}[1,m]),\quad & v^{-1} & = & \st(w^{-1}[m+1,m+n]).
\end{matrix}\]
\end{proposition}

\begin{proof}
Applying Proposition~\ref{prop:parabolic} to the parabolic subgroup $\B_{m,n}$ shows that every element $w\in \B_{m+n}$ can be written uniquely as $w=(u\times v)z$ where $u\in \B_m$, $v\in \SS_n$, and $z^{-1}\in(\B)^{m,n}$.

Since $w(z^{-1}(i))=u(i)$ for all $i\in[m]$ and $0<z^{-1}(1)<\cdots<z^{-1}(m)$, one can obtain $u(1),\cdots,u(m)$ by reading from left to right those letters in $w(1),\ldots,w(m+n)$ with absolute values in $[m]$. This implies $u=w|[1,m]$. Similarly, since $w(z^{-1}(m+j)) = m+v(j)$ for all $j\in[n]$ and $z^{-1}(m+1)<\cdots<z^{-1}(m+n)$, one can obtain $m+v(1),\ldots,m+v(n)$ by reading those letters in $w(1),\ldots, w(m+n)$ with absolute values in $[m+1,m+n]$, beginning with the negative ones from right to left, then followed by the positive ones from left to right. This implies $\st(\hat{w}|[m+1,m+n])=v$.

On the other hand, one has $w^{-1}=z^{-1}(u^{-1}\times v^{-1})$. Since $0<z^{-1}(1)<\cdots<z^{-1}(m)$, one sees that $w^{-1}(i)=z^{-1}(u^{-1}(i))$ and $u^{-1}(i)$ have the same sign for all $i\in[m]$, and if $1\le i<j\le n$ then
\[ |w^{-1}(i)|<|w^{-1}(j)| \Leftrightarrow |u^{-1}(i)|<|u^{-1}(j)|.\]
Thus $u^{-1}=\stB(w^{-1}[1,m])$. 
Similarly, one has $w^{-1}(m+i)=z^{-1}(m+v^{-1}(i))$ for all in $[n]$. It follows from $z^{-1}(m+1)<\cdots<z^{-1}(m+n)$ that
\[ w^{-1}(m+i)<w^{-1}(m+j) \Leftrightarrow v^{-1}(i)<v^{-1}(j)\]
whenever $1\le i<j\le n$. Hence $v^{-1}=\st(w^{-1}[m+1,m+n])$. 
 \end{proof}

Let $\B:=\bigsqcup_{\,n\ge0}\B_n$. We use Proposition~\ref{prop:BA} to realize $\ZZ\B=\bigoplus_{n\geq0}\ZZ\B_n$ as a dual graded right module and comodule over the Malvenuto--Reutenauer algebra $\ZZ\SS$. We define
\[ u\shuffleB v := \left\{ w\in\B_{m+n}: w|[1,m]=u,\ \st (\hat w|[m+1,m+n] )=v \right\}, \]
\[ u\CupB v := \left\{ w\in\B_{m+n}:\stB ( w[1,m] )=u,\ \st( w[m+1,m+n] )=v \right\}, \]
\[ \unshuffleB u :=\sum_{0\le i\le m} \stB (u[1,i])\otimes\st (u[i+1,m] ), \]
\[ \CapB u := \sum_{0\le i\le m} u|[1,i]\otimes \st( \hat{u}|[i+1,m] ) \]
for all $u\in \B_m$ and $v\in\SS_n$. For example, one has
\[ 
\bar1\shuffleB {\color{red}21} =
\bar1{\color{red}32}+{\color{red}3}\bar1{\color{red}2}+{\color{red}32}\bar1
+\bar1{\color{red}\bar32}+{\color{red}\bar3}\bar1{\color{red}2}+{\color{red}\bar32}\bar1
+\bar1{\color{red}2\bar3}+{\color{red}2}\bar1{\color{red}\bar3}+{\color{red}2\bar3}\bar1
+\bar1{\color{red}\bar2\bar3}+{\color{red}\bar2}\bar1{\color{red}\bar3}+{\color{red}\bar2\bar3}\bar1,
\]
\[
\bar1\CupB {\color{red}21} = 
\bar1{\color{red}32}+\bar2{\color{red}31}+{\bar3}{\color{red}21}
+\bar1{\color{red}3\bar2}+{\bar2}{\color{red}3\bar1}+{\bar3}{\color{red}2\bar1}
+{\bar1}{\color{red}2\bar3}+{\bar2}{\color{red}1\bar3}+{\bar3}{\color{red}1\bar2}
+{\bar1}{\color{red}\bar2\bar3}+{\bar2}{\color{red}\bar1\bar3}+{\bar3}{\color{red}\bar1\bar2},
\]
\[
\unshuffleB 2\bar4\bar31=\varnothing\otimes4123 +1\otimes123
+1\bar2\otimes12 +1\bar3\bar2\otimes1 +2\bar4\bar31\otimes\varnothing,
\]
\[
\CapB 2\bar4\bar31 = \varnothing\otimes3421+1\otimes231+
21\otimes12+2\bar31\otimes1+2\bar4\bar31\otimes\varnothing.
\]

Let $u\in\B_m$, $v\in\SS_n$, and $w\in\B_{m+n}$. Denote by $\unshuffle^B_m w$ the $m$-th term in $\unshuffleB\! w$, and similarly for $\Cap^B_m w$.  
Using Proposition~\ref{prop:BA} and the linear maps defined in Definition~\ref{def:MapsZW}, with $S=\{s_0,\ldots,s_{m+n-1}\}$ and $I=S\setminus\{s_m\}$, one has
\[ u\shuffleB v = \,\bar\mu_I^S(u\times v),\qquad \unshuffle^B_m w = \,\bar\rho_I^S(w), \]
\[ u\CupB v = \mu_I^S(u\times v),\qquad \Cap^B_m w = \rho_{I}^S(w).\]

\begin{proposition}\label{prop:MRBA}
(i) $(\ZZ\B,\shuffleB\!,\unshuffleB\! )$ is a graded right module and comodule over the graded Hopf algebra $(\ZZ\SS,\shuffle,\unshuffle)$.

\noindent(ii) $(\ZZ\B,\CupB,\CapB)$ is a graded right module and comodule over the graded Hopf algebra $(\ZZ\SS,\Cup,\Cap)$.

\noindent(iii) $(\ZZ\B,\shuffleB,\unshuffleB )$ is dual to $(\ZZ\B,\CupB,\CapB)$ via the pairing  $\langle u,v\rangle:=\delta_{u,v}$, $\forall u,v\in\B$.

\noindent(iv) Sending $w$ to $w^{-1}$ gives an isomorphism between $(\ZZ\B,\shuffleB,\unshuffleB )$ and $(\ZZ\B,\CupB,\CapB)$.
\end{proposition}

\begin{proof}
It is clear that $u\CupB\varnothing=u$ for any $u\in\B$, where $\varnothing\in\SS_0$ is the empty permutation. Let $u\in\B_m$, $v\in\SS_n$, and $r\in\SS_k$. Proposition~\ref{prop:IJS} implies $(u\CupB v)\CupB r = u\CupB(v\Cup r)$. In fact, one can show that $(u\CupB v)\CupB r$ and $u\CupB(v\Cup r)$ both equal  
\[
\{w\in\B_{m+n+k}: \stB(w[1,m])=u,\ \st(w[m+1,m+n])=v,\ \st(w[m+n+1,m+n+k])=r\}.
\]
Hence $(\ZZ\B,\CupB)$ is a graded right $(\ZZ\SS,\Cup)$-module. The remaining results follow from Proposition~\ref{prop:MR}: The dual of  $(\ZZ\B,\CupB)$ is the $(\ZZ\SS,\unshuffle)$-comodule $(\ZZ\B,\unshuffleB)$, and applying $w\mapsto w^{-1}$ to $(\ZZ\B,\CupB)$ and $(\ZZ\B,\unshuffleB)$ gives the $(\ZZ\SS,\shuffle)$-module $(\ZZ\B,\shuffleB)$ and the $(\ZZ\SS,\Cap)$-comodule $(\ZZ\B,\CapB)$, respectively. 
 \end{proof}

\begin{remark}
One can check that $\unshuffleB (1\shuffleB 1) \ne (\unshuffleB  1)\shuffleB (\unshuffle 1)$. Hence $(\ZZ\B,\shuffleB,\unshuffleB )$ is \emph{not} a $(\ZZ\SS,\shuffle,\unshuffle)$-Hopf module.
Consequently, the dual $(\ZZ\B,\CupB,\CapB)$ is \emph{not} a $(\ZZ\SS,\Cup,\Cap)$-Hopf module either.
\end{remark}

Let $\Sigma(\B):=\bigoplus_{n\ge0}\Sigma(\B_n)$ where $\Sigma(\B_n)$ is the free $\ZZ$-module with a basis consisting of descent classes
\[ D_\alpha(\B_n):= \left\{w\in\B_n: D(w)=D(\alpha) \right\},\quad \forall \alpha\modelsB n. \]
One has an embedding $\imath:\Sigma(\B)\hookrightarrow \ZZ\B$ by inclusion. If $\alpha\modelsB m$ and $\beta\models n$ then  
\[ D_\alpha(\B_m) \Cup D_\beta(\SS_n) = D_{\alpha\cdot\beta}(\B_{m+n}) + D_{\alpha\rhd\beta}(\B_{m+n})\]
by Proposition~\ref{prop:IndDes}, where the last term is treated as zero when $\alpha\rhd\beta$ is undefined.

Let $\Sigma^*(\B):=\bigoplus_{n\ge0} \Sigma^*(\B_n)$ where $\Sigma^*(\B_n)$ is the dual of $\Sigma(\B_n)$ with a dual basis $\left\{D^*_\alpha(\B_n): \alpha\modelsB n \right\}$.
Dual to $\imath:\Sigma(\B)\hookrightarrow \ZZ\B$ is a surjection $\chi:\ZZ\B\twoheadrightarrow\Sigma^*(\B)$ sending each $w\in\B_n$ to $D^*_w(\B_n):=D^*_\alpha(\B_n)$, where $\alpha\modelsB n$ satisfies $D(w)=D(\alpha)$.

Recall from Section~\ref{sec:Coxeter} that $\Lambda(\B_n)$ is the $\ZZ$-span of $\Lambda_\alpha(\B_n):=\chi'(D_\alpha(\B_n))$ for all $\alpha\modelsB n$, where $\chi':=\chi\circ(\ )^{-1}$. 
Let $\Lambda(\B):=\bigoplus_{n\ge0}\Lambda(\B_n)$. 
For $\alpha\modelsB m$ and $\beta\modelsB n$ define 
\begin{equation}\label{eq:cab}
\langle \Lambda_\alpha(\B_m),\Lambda_\beta(\B_n) \rangle :=
\begin{cases} \#\{w\in\B_n: D(w^{-1}) = D(\alpha),\ D(w)=D(\beta) \}, &\text{if } m=n, \\
0, & \text{if } m\ne n.
\end{cases} 
\end{equation}
By Proposition~\ref{prop:SymForm}, this gives a well-defined symmetric nondegenerate bilinear form on $\Lambda(\B)$ such that $\imath:\Lambda(\B)\hookrightarrow\Sigma^*(\B)$ and $\chi':\Sigma(\B)\twoheadrightarrow\Lambda(\B)$ are dual to each other. 

\begin{theorem}\label{thm:DesB}
The following diagram is commutative with each entry being a graded right module and comodule over the corresponding type A Hopf algebra in \eqref{eq:DiamondCoxA}.
\begin{equation}\label{eq:DiamondCoxB}
\xymatrix @R=16pt @C=7pt {
 & \ZZ\B \ar@{->>}[rd]^{\chi'} \\
 \Sigma(\B) \ar@{^(->}[ru]^{\imath} \ar@{->>}[rd]_{\chi'} & & \Sigma^*(\B) \ar@{<-->}[ll]^{\txt{\small dual}} \\
 & \Lambda(\B) \ar@{^(->}[ru]_\imath }
\end{equation}
Reflecting it across the vertical line through $\ZZ\B$ and $\Lambda(\B)$ gives a dual diagram. 
\end{theorem}

\begin{proof}
Apply Theorem~\ref{thm:FullDiagram} to $(\ZZ\B,\shuffleB,\unshuffleB,\CupB,\CapB)$ and then use Corollary~\ref{cor:Sym}.
 \end{proof}

\subsection{Free quasisymmetric functions of type B and related results}\label{sec:FQSymB}
In this subsection we obtain the following commutative diagram.
\begin{equation}\label{eq:DiamondFQSymB}
\DiamondFQSym{B}
\end{equation}
It is isomorphic to the diagram \eqref{eq:DiamondCoxB} with each entry being a graded right module and comodule over the corresponding type A Hopf algebra in \eqref{eq:DiamondFQSymA}. Reflecting it across the vertical line through $\FQSym^B$ and $\Sym^B$ gives a dual diagram of graded modules and comodules.

\subsubsection{Free quasisymmetric functions of type B}
Let $(W,S)$ be the Coxeter system of type $B_n$, where $W=\B_n$ and $S=\{s_0=s^B_0,s_1,\ldots,s_{n-1}\}$. Let $E=\mathbb R^n$ be a Euclidean space with standard basis $\{e_1,\ldots,e_n\}$. The hyperoctahedral group $\B_n$ can be realized as a reflection group of $E$ whose root system $\Phi$ is the disjoint union of $\Phi^+=\{e_i,e_j\pm e_i: 1\leq i<j\leq n\}$ and $\Phi^-=-\Phi^+$. The set of simple roots is $\Delta=\{e_1,e_2-e_1,\ldots,e_n-e_{n-1}\}$, corresponding to the generating set $S$ of simple reflections. 

Let $\bX=\{\bx_i:i\in\ZZ\}$ be a set of noncommutative variables. 
Define $\FQSym^B_n$ to be the $\ZZ$-span of the generating functions $\bF_P^S$ for all parsets $P$ of $\Phi$. 
Let $\FQSym^B:=\bigoplus_{n\geq0}\FQSym^B_n$.
By Proposition~\ref{prop:FQSymW}, $\FQSym^B_n$ has free $\ZZ$-bases $\{\bFB_w:w\in\B_n\}$ and $\{\bsB_w:w\in\B_n\}$, where $\bFB_w$ is the generating function of the parset $w\Phi^+$ and $\bsB_w:=\bFB_{w^{-1}}$. Applying the definition of $f\in\mathcal A(w\Phi^+)$ to $w\alpha$ for all $\alpha\in\Delta$ gives
\begin{equation}\label{eq:BFw}
\bFB_w = \sum_{ \substack{ 
f(w(0))\leq f(w(1))\leq\cdots\leq f(w(n)) \\
i\in D(w) \Rightarrow f(w(i))<f(w(i+1)) }}
\bx_{f(1)}\cdots \bx_{f(n)},\qquad\forall w\in\B_n.
\end{equation}
Here we set $w(0)=0$, $f(0)=0$, and $f(-i)=-f(i)$ for all $i\in[n]$ by convention.

\begin{proposition}\label{prop:bsB}
Let $w\in\B_n$ and let $f\in\ZZ^n$. Then $f\in\mathcal A(w^{-1}\Phi^+) \Leftrightarrow \stB(f)=w$ and thus 
\[ \bsB_w = \sum_{f\in\ZZ^n:\,\stB(f)=w} \bx_f. \]
\end{proposition}

\begin{proof}
By Lemma~\ref{lem:Z^n}, it suffices to show that $f\in\mathcal A(w^{-1}\Phi^+)$ implies $\stB(f)=w$. Applying Proposition~\ref{prop:InverseP} to $\alpha=e_i$ for all $i\in[n]$ gives $\Neg(f)=\Neg(w)$. Let $1\leq i<j\leq n$ below. Apply Proposition~\ref{prop:InverseP} to $\alpha=e_j-e_i$ gives $w(i)<w(j) \Leftrightarrow f(i)\leq f(j)$, which implies \eqref{def:bst} when $w(i)$ and $w(j)$ have the same sign. Applying Proposition~\ref{prop:InverseP} to $\alpha=e_j+e_i$ gives $w(j)+w(i)>0 \Leftrightarrow f(j)+f(i)\geq0$, which implies (\ref{def:bst}) when $w(i)$ and $w(j)$ have the opposite signs. Hence $\stB(f)=w$.
 \end{proof}

\begin{corollary}
If $w\in \SS_n$ then $\bs_w$ equals the sum of $\bsB_u$ for all $u\in\B_n$ with $\st(u)=w$. Consequently $\FQSym\subseteq\FQSym^B$.
\end{corollary}
\begin{proof}
The result follows from \eqref{eq:bs} and Proposition~\ref{prop:bsB} if we can prove $\st(\stB(a))=\st(a)$ for all $a\in\ZZ^n$. Let $\stB(a) = u \in \B_n$. Then $\Neg(a)=\Neg(u)$. Assume $1\leq i<j\leq n$ below.

If $a_i$ and $a_j$ are both positive then so are $u(i)$ and $u(j)$, and thus $u(i)<u(j) \Leftrightarrow a_i\leq a_j$ by (\ref{def:bst}).

If $a_i$ and $a_j$ are both negative then so are $u(i)$ and $u(j)$, and thus $u(i)<u(j) \Leftrightarrow a_i\leq a_j$ by (\ref{def:bst}).

If $a_i<0\leq a_j$ then $u(i)<0<u(j)$. Similarly, if $a_j<0\leq a_i$ then $u(j)<0<u(i)$.

Therefore $u(i)<u(j)$ if and only if $a_i\leq a_j$. This implies $\st(u) = \st(a)$ and completes the proof.
 \end{proof}

For example, one has $\bs_{12}=\bsB_{12}+\bsB_{\bar12}+\bsB_{\bar21}+\bsB_{\bar2\bar1}$ and $\bs_{21}=\bsB_{21}+\bsB_{2\bar1}+\bsB_{1\bar2}+\bsB_{\bar1\bar2}$.

Now we define an action and a coaction of $\FQSym$ on $\FQSym^B$.
Proposition~\ref{prop:shuffleW} implies 
\begin{equation}\label{eq:Action bFB}
\bFB_u \cdot \bF_v = \sum_{w\in\, u\shuffleB v} \bFB_w, \quad \forall u\in\B_m,\ \forall v\in\SS_n.
\end{equation}
This gives a right action of $\FQSym$ on $\FQSym^B$, which is denoted by ``$\odotB$''. Note that $\bF_v$ is $\bF_v(\bX)$ instead of $\bF_v(\bX_{>0})$ by our convention in this paper. If $u\in\B_m$ then define
\begin{equation}\label{eq:Coaction bFB}
\DeltaB(\bFB_u) := \sum_{0\leq i\leq m} \bFB_{\,\stB(u[1,i])}\otimes\bF_{\,\st(u[i+1,n])}.
\end{equation}
Also define a bilinear form on $\FQSym^B$ by $\langle \bFB_u,\bsB_v\rangle := \delta_{u,v}$ for all $u,v\in\B$.

\begin{proposition}\label{prop:FQSymB}
$(\FQSym^B,\odotB,\DeltaB)$ is a self-dual graded right module and comodule over $\FQSym$ isomorphic to $(\ZZ\B,\shuffleB,\unshuffleB )$ via $\bFB_w\mapsto w$ and to $(\ZZ\B,\CupB,\CapB)$ via $\bsB_w\mapsto w$.
\end{proposition}

\begin{proof}
This follows from (\ref{eq:Action bFB}), (\ref{eq:Coaction bFB}), and Proposition~\ref{prop:MRBA} (iv). 
 \end{proof}

It follows that if $u\in\B_m$ and $v\in\SS_n$ then
\begin{equation}\label{eq:Coaction bsB}
\bsB_u \cdot \bs_v = \sum_{w\in\, u\CupB v}\bsB_w \qand
\DeltaB(\bsB_u) = \sum_{0\leq i\leq m} \bsB_{u|[1,i]}\otimes\bs_{\st(\hat u|[i+1,m])}.
\end{equation}

\subsubsection{Noncommutative symmetric functions of Type B}
In our earlier work~\cite{H0Tab} we defined a type B analogue $\NSym^B$ of $\NSym$ with two free $\ZZ$-bases $\{\bhB_\alpha\}$ and $\{\bsB_\alpha\}$, where $\alpha$ runs through all pseudo-compositions. If $\alpha=(\alpha_1,\ldots,\alpha_\ell)$ is a pseudo-composition then $\bhB_\alpha$ and $\bsB_\alpha$ are defined by 
\[ \bhB_\alpha = \bhB_{\alpha_1}\cdot\bh_{\alpha_2} \cdots\bh_{\alpha_\ell} = \sum_{\beta\cleq\alpha}\bsB_\beta \]
where \begin{equation}\label{eq:bhB}
\bhB_k = \bsB_k := \sum_{0\le i_1\le\cdots\le i_k} \bx_{i_1}\cdots \bx_{i_k} 
= \sum_{0\leq i\leq k} \bx_0^i \cdot \bh_{k-i}, \quad \forall k\ge0.
\end{equation}
It follows that $\NSym^B$ is a free right $\NSym$-module with a basis $\{\bhB_k:k\geq0\}$, and one has 
\[ \bhB_\alpha\cdot\bh_\beta = \bhB_{\alpha\cdot\beta} \qand \bsB_\alpha\cdot\bs_\beta=\bsB_{\alpha\cdot\beta}+\bsB_{\alpha\rhd\beta} \]
for all $\alpha\modelsB m$ and $\beta\models n$, where the last term is treated as zero when $\alpha\rhd\beta$ is undefined.
This right $\NSym$-action on $\NSym^B$, denoted by $\odotB$ for consistency of notation, was used by Chow~\cite{Chow} to define $\NSym^B$ abstractly (without a power series realization). We next provide an embedding of $\NSym^B$ into $\FQSym^B$, which will recover the above $\NSym$-module structure and also induce a $\NSym$-comodule structure on $\NSym^B$.

Let $\tau$ be a \emph{tableau} of pseudo-ribbon shape $\alpha$, i.e., a filling of the pseudo-ribbon $\alpha$ with integers. Reading these integers from the bottom row to the top row and proceeding from left to right within each row, \emph{excluding the extra 0} in the pseudo-ribbon $\alpha$, gives the \emph{reading word} $w(\tau)$ of $\tau$. We call $\tau$ \emph{(type B) semistandard} if each row is weakly increasing from left to right and each column is strictly increasing from top to bottom, \emph{including the extra 0}. We showed in~\cite{H0Tab} that $\bsB_\alpha$ is the sum of $\bx_{w(\tau)}$ for all semistandard tableaux $\tau$ of pseudo-ribbon shape $\alpha$.

\begin{proposition}\label{prop:bsBa}
If $\alpha\modelsB n$ then $\bsB_\alpha$ equals the sum of $\bsB_w$ for all $w\in\B_n$ with $D(w)=D(\alpha)$.
\end{proposition}

\begin{proof}
 Each $f\in\ZZ^n$ corresponds to a unique tableau $\tau$ of shape $\alpha$ such that $w(\tau)=f$. One sees that $D(\stB(f))=D(\alpha)$ if and only if $\tau$ is semistandard. Thus the result follows from Proposition~\ref{prop:bsB}.
 \end{proof}

It follows that there is an injection $\imath:\NSym^B\hookrightarrow \FQSym^B$ by inclusion.

\begin{proposition}\label{prop:NSymB}
The graded right module and comodule $\FQSym^B$ over $\FQSym$ restricts to a graded right module and comodule $(\NSym^B,\odotB,\DeltaB)$ over $\NSym$, which is isomorphic to the graded right module and comodule $(\Sigma(\B),\CupB,\CapB)$ over $(\Sigma(\SS),\Cup,\Cap)$ via the map $\bsB_\alpha\mapsto D_\alpha(\B_m)$, $\forall\alpha\modelsB m$, $\forall m\ge0$. 
\end{proposition}

\begin{proof}
By Proposition~\ref{prop:FQSymB}, there is an isomorphism $(\FQSym^B,\odotB,\DeltaB)\cong (\ZZ\B,\CupB,\CapB)$ via $\bsB_w\mapsto w$, $\forall w\in\B$.
Restricting this isomorphism to $\NSym^B$ gives the result.
 \end{proof}

\begin{remark}
If $k$ is a nonnegative integer then it follows from \eqref{eq:Coaction bsB} and Proposition~\ref{prop:bsBa} that
\[ \DeltaB(\bhB_k) = \DeltaB(\bsB_{12\cdots k}) = \sum_{0\le i\le k} \bsB_{12\cdots i} \otimes \bs_{\st(i+1,\ldots, k)} = \sum_{0\le i\le k} \bhB_i \otimes \bh_{k-i}. \]
We do not have any explicit formula for $\DeltaB(\bhB_\alpha)$ or $\Delta(\bsB_\alpha)$ for an arbitrary $\alpha\modelsB n$.
\end{remark}

\subsubsection{Quasisymmetric functions of type B}

Let $X_{\ge0}=\{x_0,x_1,x_2,\ldots\}$ be a totally ordered set of commutative variables.  Chow~\cite{Chow} introduced a type B analogue $\QSym^B$ of $\QSym$, which admits two free $\ZZ$-bases $\{\FB_\alpha\}$ and $\{\MB_\alpha\}$, where $\alpha$ runs through all pseudo-compositions. If $\alpha=(\alpha_1,\ldots,\alpha_\ell)$ is a pseudo-composition of $n$ then one has the \emph{type B monomial quasisymmetric function}
\begin{equation}\label{eq:MB}
\MB_\alpha:=\sum_{0<i_2<\cdots<i_\ell} x_0^{\alpha_1}x_{i_2}^{\alpha_2}\cdots x_{i_\ell}^{\alpha_\ell} 
= x_0^{\alpha_1} \cdot M_{(\alpha_2,\dots,\alpha_\ell)} 
\end{equation}
and the \emph{type B fundamental quasisymmetric function} (with $i_0:=0$)
\begin{equation}\label{eq:FB}
\FB_\alpha:=\sum_{\alpha\cleq\beta}\MB_\beta = \sum_{ \substack{ 0\le i_1\le \cdots\le i_n\\ j\in D(\alpha) \Rightarrow i_j<i_{j+1} }} x_{i_1}\cdots x_{i_n} .
\end{equation}
There exist unique $\alpha_{\le i} \modelsB i$ and $\alpha_{>i}\models n-i$ such that $\alpha\in\{\alpha_{\le i}\cdot \alpha_{>i},\alpha_{\le i}\rhd\alpha_{>i}\}$ for each $i\in\{0,1,\ldots,n\}$, and one can show that 
$\FB_\alpha = \sum_{0\leq i\leq \alpha_1} x_0^i\cdot F_{\alpha_{>i}}.$
By \eqref{eq:MB}, one also has
\[ \QSym^B = \ZZ[x_0]\cdot \QSym \cong \ZZ[x_0] \otimes_\ZZ \QSym.\]




We define an algebra map $\chiB:\ZZ\langle \bX\rangle \to \ZZ[X_{\ge0}]$ by $\bx_i \mapsto x_{|i|}$ for all $i\in\ZZ$. If $w\in\B_n$ then 
\[ \FB_w := \chiB(\bFB_w) = \sum_{\substack{ 0\leq f(w(1))\leq\cdots\leq f(w(n)) \\ j\in D(w)\Rightarrow f(w(j))<f(w(j+1)) }} x_{f(w(1))}\cdots x_{f(w(n))} \]
by (\ref{eq:BFw}).
One sees that $\FB_w=\FB_\alpha$ if $w\in\B_n$, $\alpha\modelsB n$, and $D(w)=D(\alpha)$. 
This gives a surjection $\chiB: \FQSym^B\twoheadrightarrow \QSym^B$. 

If $u\in\B_m$ and $v\in\SS_n$ then applying the algebra homomorphism $\chiB$ to (\ref{eq:Action bFB}) gives 
\begin{equation}\label{eq:Action FB}
\FB_u\odotB F_v:=\FB_u\cdot\chiB(\bF_v)=\sum_{w\in u\shuffleB v} \FB_w.
\end{equation}
This defines a right $\QSym$-action on $\QSym^B$, which is believed to be new. 

Chow~\cite{Chow} introduced a right coaction of $\QSym$ on $\QSym^B$ by $\DeltaB \FB:=\FB(X_{\ge0}+Y_{>0})$, $\forall\FB\in\QSym^B$, where $X_{\ge0}+Y_{>0}=\{x_0,x_1,x_2,\ldots,y_1,y_2,\ldots\}$ is a totally ordered set of commutative variables. 
One can check that if $\alpha=(\alpha_1,\ldots,\alpha_\ell)\modelsB n$ then
\[ \DeltaB \MB_\alpha =\sum_{1\le j\le\ell} \MB_{(\alpha_1,\ldots,\alpha_j)}\otimes M_{(\alpha_{j+1},\ldots,\alpha_\ell)} 
\qand \DeltaB \FB_\alpha = \sum_{0\leq i\leq n} \FB_{\alpha{\le i}}\otimes F_{\alpha{>i}}. \]
The second equality is equivalent to 
\begin{equation}\label{eq:Coaction FB}
\DeltaB \FB_w = \sum_{0\leq i\leq n} \FB_{\,\stB(w[1,i])}\otimes F_{\,\st(w[i+1,n])},\quad \forall w\in\B_n.
\end{equation}
Therefore this coaction is preserved by the surjections $\chiB: \FQSym^B\twoheadrightarrow\QSym^B$ and $\chi: \FQSym\twoheadrightarrow \QSym$. 
We observe that this coaction can also be obtained by applying the coproduct of $\QSym$ to the second tensor component of $\QSym^B\cong\ZZ[x_0] \otimes_\ZZ \QSym$.

\begin{proposition}\label{prop:QSymB}
$(\QSym^B,\odotB,\DeltaB)$ is a graded right module and comodule over $\QSym$ isomorphic to the graded right module and comodule $(\Sigma^*(\B),\shuffleB,\unshuffleB )$ over $(\Sigma^*(\SS),\shuffle,\unshuffle)$.
The map $\chiB$ induces a surjection from the graded right module and comodule $\FQSym^B$ over $\FQSym$ onto $\QSym^B$.
\end{proposition}

\begin{proof}
The result follows from (\ref{eq:Action FB}) and (\ref{eq:Coaction FB}).
 \end{proof}

Define a pairing between $\NSym^B$ and $\QSym^B$ by $\langle \bsB_\alpha,\FB_\beta\rangle = \langle \bhB_\alpha, \MB_\beta\rangle:= \delta_{\alpha,\beta}$ for all pseudo-compositions $\alpha$ and $\beta$.

\begin{corollary}
The embedding $\imath:\NSym^B\hookrightarrow \FQSym^B$ and the surjection $\chiB:\FQSym^B\twoheadrightarrow\QSym^B$ are dual morphisms of graded modules and comodules.
\end{corollary}

\begin{proof} 
This follows from Theorem~\ref{thm:DesB}, Proposition~\ref{prop:NSymB}, and Proposition~\ref{prop:QSymB}.
 \end{proof}

\subsubsection{Symmetric functions of type B}
We next investigate $\Sym^B:=\chi^B(\NSym^B) \subseteq \QSym^B$, which is the $\ZZ$-span of
$\sB_\alpha:= \chiB(\bsB_\alpha)$ for all pseudo-compositions $\alpha$.
A representation theoretic interpretation for $\sB_\alpha$ will be provided in Proposition~\ref{prop:ChPB}. 
Another spanning set for $\Sym^B$ consists of
\[ \hB_\alpha:=\chiB(\bhB_\alpha) = \sum_{\beta\cleq\alpha} \chiB(\bsB_\beta) = \sum_{\beta\cleq\alpha} \sB_\beta\]
for all pseudo-compositions $\alpha$.
We have an isomorphism $\Lambda(\B)\cong\Sym^B$ via $\Lambda_\alpha(\B_n) \mapsto \bsB_\alpha$, $\forall \alpha\modelsB n$, $\forall n\ge0$.
This and \eqref{eq:cab} give a nondegenerate symmetric bilinear form on $\Sym^B$.

A \emph{pseudo-partition} $\lambda$ of $n$, denoted by $\lambda\vdash^B n$, is a sequence of integers $\lambda=(\lambda_1,\ldots,\lambda_\ell)$ such that $\lambda_1\ge0$, $\lambda_2\ge\cdots\ge\lambda_\ell\ge1$, and its \emph{size} $|\lambda|:=\lambda_1+\cdots+\lambda_\ell$ equals $n$. 
Given a pseudo-composition $\alpha=(\alpha_1,\ldots,\alpha_\ell)$, denote by $\lambda^B(\alpha)$ the pseudo-partition obtained from $\alpha$ by rearranging $\alpha_2,\ldots,\alpha_\ell$. 

\begin{proposition}\label{prop:BasisSymB}
If $\lambda\vdash^B n$ then $\hB_\lambda:=\hB_\alpha$ is well defined for any $\alpha\modelsB n$ with $\lambda^B(\alpha)=\lambda$.
Moreover, there is a free $\ZZ$-basis $\{\hB_\lambda:\lambda\vdash^B n,\ n\ge0\}$ for $\Sym^B$.
\end{proposition}

\begin{proof}
According to Remark~\ref{rem:Sym} (i), two parabolic subgroups $\SS^B_\alpha$ and $\SS^B_\beta$ are conjugate in $\SS^B_n$ if and only if $\lambda^B(\alpha) = \lambda^B(\beta)$.
Thus the result follows from Proposition~\ref{prop:Sym}.
\end{proof}

\begin{proposition}\label{prop:TensorSymB}
We have $\Sym^B\subseteq \ZZ[x_0]\cdot\Sym\cong\ZZ[x_0]\otimes_\ZZ\Sym$ and if $\KK$ be a field of characteristic zero then $\Sym^B_\KK = \KK[x_0]\cdot \Sym_\KK \cong\KK[x_0] \otimes_\KK \Sym_\KK.$
\end{proposition}

\begin{proof}
For any nonnegative integer $k$ it follows from \eqref{eq:bhB} and the definition of $\bh_k$ that
\begin{equation}\label{eq:hB}
\hB_k = \chiB(\bhB_k) =  \sum_{0\leq i\leq k} x_0^i \cdot h_{k-i} \quad\text{and}
\end{equation}
\begin{equation}\label{eq:chiBh}
\chiB(\bh_k) = \sum_{ i_1\le\cdots\le i_k} x_{|i_1|}\cdots x_{|i_k|} = \sum_{\substack{a, b,c\ge0 \\ a+b+c=k}} h_a \cdot x_0^b \cdot h_c.
\end{equation}
Hence $\hB_\alpha = \chiB(\bhB_{\alpha_1})\cdot\chiB(\bh_{\alpha_2})\cdots\chiB(\bh_{\alpha_\ell}) \in\ZZ[x_0]\cdot\Sym$ for all $\alpha=(\alpha_1,\ldots,\alpha_\ell)\modelsB n$.
This implies $\Sym^B\subseteq \ZZ[x_0]\cdot\Sym$.
By Proposition~\ref{prop:BasisSymB}, the homogeneous component $\Sym^B_n$ of degree $n$ of $\Sym^B$ is a free $\ZZ$-module of rank equal to the sum of the numbers of partitions of $k$ for all $k=0,1,2,\ldots,n$.
Thus the equality between $\Sym^B_\KK$ and $\KK[x_0]\cdot \Sym_\KK$ follows from a comparison of the dimensions of their homogeneous components.
\end{proof}

If $\lambda=(\lambda_1,\ldots,\lambda_\ell)$ is a pseudo-partition then define $\mB_\lambda := \sum_{\lambda(\alpha)=\lambda} \MB_\alpha$.
We have
\begin{equation}\label{eq:mB}
\mB_\lambda = x_0^{\lambda_1}\cdot \sum_{\lambda(\alpha_2,\ldots,\alpha_\ell)=(\lambda_2,\ldots,\lambda_\ell)} M_{(\alpha_2,\ldots,\alpha_\ell)} = x_0^{\lambda_1} \cdot m_{(\lambda_2,\ldots,\lambda_\ell)}.
\end{equation}
Hence $\{\mB_\lambda\}$ is a free $\ZZ$-basis for $\ZZ[x_0]\cdot\Sym$, where $\lambda$ runs through all pseudo-partitions.

\begin{proposition}
One has dual $\KK$-bases $\{\hB_\lambda\}$ and $\{\mB_\lambda\}$ for $\Sym^B_\KK$, where $\lambda$ runs through all pseudo-partitions.
\end{proposition}

\begin{proof}
We already know that $\{\hB_\lambda\}$ is a free $\ZZ$-basis for $\Sym^B$, and hence a $\KK$-basis for $\Sym^B_\KK$.
By Proposition~\ref{prop:TensorSymB} and \eqref{eq:mB}, $\{\mB_\lambda\}$ is a $\KK$-basis for $\Sym^B_\KK=\KK[x_0]\cdot \Sym_\KK$. 
Let $\lambda$ and $\mu$ be two pseudo-partitions, and let $\alpha$ be a pseudo-composition with $\lambda(\alpha)=\lambda$. 
Then
\[ \langle \hB_\lambda,\mB_\mu \rangle = \langle \bhB_\alpha,\sum_{\lambda(\beta)=\mu} \MB_\beta \rangle = \delta_{\lambda,\mu}.\]
Thus $\{\hB_\lambda\}$ and $\{\mB_\lambda\}$ are dual $\KK$-bases for $\Sym^B_\KK$.
\end{proof}

\begin{remark}
(i) Proposition~\ref{prop:bsBa} implies that  $\sB_{0\cdot1^k} = s_{1^k}$.

\noindent (ii) 
We know that $\{\hB_\lambda\}$ is a free $\ZZ$-basis for $\Sym^B$ and $\{\mB_\lambda\}$ is a free $\ZZ$-basis for $\ZZ[x_0]\cdot\Sym$.
However, $\{\mB_\lambda\}$ is not a free $\ZZ$-basis for $\Sym^B$, since $\Sym^B \subsetneq\ZZ[x_0]\cdot\Sym$ by the following calculation using \eqref{eq:hB} and \eqref{eq:chiBh}.
\[ \arraycolsep=2pt \def\arraystretch{1.2}
\begin{matrix}
\hB_2 &=& x_0^2 & + & x_0h_1 & + & h_2  \\
 \hB_{11} &=& x_0^2 & +&  3x_0h_1 & + & & & 2h_{11} \\
 \hB_{02} &=& x_0^2 & +& 2x_0h_1 & +& 2h_2 & +& h_{11} \\
 \hB_{011} &=& x_0^2 &+& 4x_0h_1 & + & & & 4h_{11}
\end{matrix} 
\qquad\Rightarrow\qquad
\arraycolsep=2pt  \def\arraystretch{1.2} \begin{matrix}
x_0^2 &=& \frac83 \hB_2 & - & \frac43 \hB_{11} &- &  \frac43 \hB_{02} &+& \hB_{011}\\
x_0h_1 &=& -\frac43 \hB_2 & + &  \frac53 \hB_{11} & + & \frac23 \hB_{02} & - & \hB_{011}  \\
h_2 &=& -\frac13 \hB_2 & - &  \frac13 \hB_{11} & + & \frac23 \hB_{02} \\
 h_{11} &=& \frac23 \hB_2 & - & \frac43\hB_{11} &  - & \frac13 \hB_{02} & + & \hB_{011}
\end{matrix}\]
\end{remark}

For any pseudo-composition $\alpha$ and composition $\beta$, applying $\chiB$ to $\bhB_\alpha\cdot\bh_\beta=\bhB_{\alpha\cdot\beta}$ gives
\[ \hB_{\alpha}\odotB h_{\beta} := \hB_\alpha \cdot \chiB(\bh_\beta) = \hB_{\alpha\cdot\beta}.\]
Also define
\[ \DeltaB(\sB_\alpha) = \DeltaB(\chiB(\bsB_\alpha)) := (\chiB\otimes\chi)(\DeltaB(\bsB_\alpha)).\]

\begin{proposition}\label{prop:SymB}
$(\Sym^B,\odotB,\DeltaB)$ is a graded right module and comodule over $\Sym$ isomorphic to the graded right module and comodule $\Lambda(\B)$ over $\Lambda(\SS)$. The injection $\imath:\Sym^B\hookrightarrow\QSym^B$ and the surjection $\chi^B:\NSym^B\twoheadrightarrow\Sym^B$ are dual morphisms of graded right modules and comodules.
There is a free basis $\{\hB_k:k\ge0\}$ for $\Sym^B$ as a right $\Sym$-module.
\end{proposition}

\begin{proof}
Apply Theorem~\ref{thm:DesB} and Propositions~\ref{prop:NSymB},~\ref{prop:QSymB}, and~\ref{prop:BasisSymB}.
 \end{proof}

\subsection{Representation theory} 
Now we study the connections of $\NSym^B$ and $\QSym^B$ with the representation theory of 0-Hecke algebras of type B. 
Let $(W,S)$ be the finite Coxeter system of type $B_n$, where $W=\B_n$ and $S=\{s_0=s^B_0,s_1,\ldots,s_{n-1}\}$. 
Let $\FF$ be a field. The 0-Hecke algebra $\HB_n(0)$ of $(W,S)$ is an $\FF$-algebra with two generating sets $\{\pib_i:0\le i\le n-1\}$ and $\{\pi_i:0\le i\le n-1\}$. One can realize $\pi_0,\pi_1,\ldots,\pi_{n-1}$ as \emph{signed bubble-sorting operators} on $\ZZ^n$: if $(a_1,\ldots,a_n)\in \ZZ^n$ then
\[
\pi_i(a_1,\ldots,a_n):=
\begin{cases}
(-a_1,a_2,\ldots, a_n), & \text{if } i=0,\ a_1>0,\\
(a_1,\ldots,a_{i+1},a_i,\ldots,a_n), & \text{if } 1\leq i\leq n-1,\ a_i<a_{i+1},\\
(a_1,\ldots,a_n), & {\rm otherwise}.
\end{cases}
\]

The projective indecomposable $\HB_n(0)$-modules and simple $\HB_n(0)$-modules are given by $\PB_\alpha:=\P_I^S$ and $\CB_\alpha:=\C_I^S$, respectively, where $I=\{s_i:i\in D(\alpha)\}$, for all $\alpha\modelsB n$. One can realize $\PB_\alpha$ as the $\FF$-space of \emph{standard tableaux of pseudo-ribbon shape} $\alpha$ with an appropriate $\HB_n(0)$-action~\cite{H0Tab}. 

The parabolic subalgebra $\HB_\alpha(0)$ of $\HB_n(0)$ is generated by $\{\pi_i: i\in\{0,1,\ldots,n-1\}\setminus D(\alpha)\}$. 
If $m$ and $n$ are nonnegative integers one has $\HB_m(0)\otimes H_n(0)\cong \HB_{m,n}(0) \subseteq \HB_{m+n}(0)$, giving an embedding $\HB_m(0)\otimes H_n(0)\hookrightarrow\HB_{m+n}(0)$.

Associated with the tower of algebras $\HB_\bullet(0): \HB_0(0)\hookrightarrow \HB_1(0)\hookrightarrow \HB_2(0)\hookrightarrow\cdots$ are Grothendieck groups
\[ G_0(\HB_\bullet(0)):=\bigoplus_{n\geq0}G_0(\HB_n(0)) \qand K_0(\HB_\bullet(0)):=\bigoplus_{n\geq0}K_0(\HB_n(0)). \]
Let $M$ and $N$ be finitely generated modules over $\HB_m(0)$ and $H_n(0)$, respectively. Define
\[ M \odotB N := (M\otimes N)\uparrow\,_{\HB_{m,n}(0)}^{\HB_{m+n}(0)} \qand 
\DeltaB (M) := \sum_{0\le i\le m} M \downarrow\,_{\HB_{i,m-i}(0)}^{\HB_{m}(0)}. \] 
Using the linear maps $\bar\mu_I^S:G_0(H_{W_I}(0))\to G_0(H_W(0))$ and $\bar\rho_I^S:G_0(H_W(0))\to G_0(H_{W_I}(0))$ defined in \eqref{def:IndResG0}, with $S=\{s_0,\ldots,s_{m+n-1}\}$ and $I=S\setminus\{s_m\}$, one has
\[ M\odotB N = \bar\mu_I^S(M\otimes N) \qand \DeltaB_m (Q) := \downarrow\,_{\HB_{m,n}(0)}^{\HB_{m+n}(0)} = \bar\rho_I^S(Q) \]
where $Q$ is a finitely generated $\HB_{m+n}(0)$-module.
Also recall from Proposition~\ref{prop:IndResG0K0} that $\bar\mu_I^S$ and $\bar\rho_I^S$ restrict to $\mu_I^S:K_0(H_{W_I}(0))\to K_0(H_W(0))$ and $\rho_I^S:K_0(H_W(0))\to K_0(H_{W_I}(0))$.
If $M$, $N$, and $Q$ are all projective then 
\[ M \odotB N = \mu_I^S(M\otimes N) \qand \DeltaB_m(Q) = \rho_I^S(Q).\]
Define the following characteristic maps, where $\alpha$ runs through all pseudo-compositions: 
\[ \begin{matrix}
\mathrm{Ch}: & G_0(\HB_\bullet(0)) & \to & \QSym^B & \qand & \bch: & K_0(\HB_\bullet(0)) & \to & \NSym^B \\  
& \CB_\alpha & \mapsto & \FB_\alpha & & & \PB_\alpha & \mapsto & \bsB_\alpha.
\end{matrix}\]

\begin{theorem}
(i) $(G_0(\HB_\bullet(0)),\odotB,\DeltaB)$ is a graded right module and comodule over the Hopf algebra $G_0(H_\bullet(0))$.

\noindent (ii) $(K_0(\HB_\bullet(0)),\odotB,\DeltaB)$ is a graded right module and comodule over the Hopf algebra $K_0(H_\bullet(0))$.

\noindent (iii) $(G_0(\HB_\bullet(0)),\odotB,\DeltaB)$ is dual to $(K_0(\HB_\bullet(0)),\odotB,\DeltaB)$ via the pairing $\langle \PB_\alpha, \CB_\beta \rangle:=\delta_{\alpha,\beta}$.

\noindent(iv) Both $\Ch$ and $\bch$ are isomorphisms of graded modules and comodules.
\end{theorem}

\begin{proof}
Apply Theorem~\ref{thm:GrH0}, Theorem~\ref{thm:DesB}, Proposition~\ref{prop:NSymB}, and Proposition~\ref{prop:QSymB}.
 \end{proof}

\begin{proposition}\label{prop:ChPB}
If $\alpha$ is a pseudo-composition of $n$ then ${\rm Ch}(\,\PB_\alpha) = \sB_\alpha$. 
\end{proposition}
\begin{proof}
This follows from Proposition~\ref{prop:PtoC}.
 \end{proof}

\subsection{Other results}\label{sec:HopfB}

Our results in type B are based on the embedding $\B_m\times\SS_n\hookrightarrow\B_{m+n}$ which identifies $\B_m\times\SS_n$ with the parabolic subgroup $\B_{m,n}$ of $\B_{m+n}$. 
There is another embedding $\B_m\times\B_n\hookrightarrow\B_{m+n}$ whose image is a non-parabolic subgroup of $\B_{m+n}$ generated by $\{s_0,\ldots,s_{m-1}, s'_m, s_{m+1},\ldots,s_{m+n-1}\}$ where 
\[s'_m:=s_{m}\cdots s_1s_0s_1\cdots s_{m} = 1\cdots \overline{(m+1)}\cdots(m+n).\]
It does not induce an embedding $\HB_m(0)\otimes \HB_n(0)\hookrightarrow \HB_{m+n}(0)$, as
$ \pi'_m:=\pi_m\cdots \pi_1\pi_0\pi_1\cdots\pi_m $ does not satisfy the relation $(\pi'_m)^2=\pi'_m$. For example, in $\HB_{1+1}(0)$ one has $\pi'_1=\pi_1\pi_0\pi_1$ and $(\pi'_1)^2=\pi_0\pi_1\pi_0\pi_1 \ne \pi'_1$.

Although the embedding $\B_m\times\B_n\hookrightarrow\B_{m+n}$ does not fit into our general theory in Section~\ref{sec:General}, it leads to a self-dual graded Hopf algebra structure on $\ZZ\B$.
In fact, an element $(u,v)\in \B_m\times\B_n$ can be identified with $u\times v\in\B_{m+n}$ whose window notation is $[u(1),\ldots,u(m),v(1)^{+m},\ldots,v(n)^{+m}]$ where $a^{+m} := a + \operatorname{sgn}(a)\cdot m$. This gives an embedding $\B_m\times\B_n\hookrightarrow\B_{m+n}$. One shows that every $w\in \B_{m+n}$ can be written uniquely as $w=z(u\times v)$, where $u\in \B_m$, $v\in \B_n$, and $z\in\B_{m+n}$ satisfying 
\[ 0<z(1)<\cdots<z(m) \qand 0<z(m+1)<\cdots<z(m+n). \]
Given such an expression, one can check that  
\[ \begin{matrix}
u & = & \stB(w[1,m]),\quad & v & = & \stB(w[m+1,m+n]),\\
u^{-1} & = & w^{-1}|[1,m],\quad & v^{-1} & = & \stB(w^{-1}|[m+1,m+n]).
\end{matrix} \]
If $u\in\B_m$ and $v\in\B_n$ then we define 
\[ u\shuffle v := \left\{w\in\B_{m+n}:w|[1,m]=u,\ \stB(w|[m+1,m+n])=v \right\}, \]
\[ u \Cup v := \left\{w\in\B_{m+n}:\stB(w[1,m])=u,\ \stB(w[m+1,m+n])=v \right\},\]
\[ \unshuffle(u):=\sum_{0\leq i\leq m}\stB(u[1,i])\otimes\stB(u[i+1,m]), \]
\[ \Cap\,(u):=\sum_{0\leq i\leq m} u|[1,i]\otimes \stB(u|[i+1,m]).\]
This is very similar to the Malvenuto-Reutenauer algebra defined in Section~\ref{sec:MRA} and hence we use the same notation.
For example, one has 
\[ \bar21 \shuffle {\color{red}1\bar 2} = \bar21{\color{red}3\bar 4} + \bar2{\color{red}3}1{\color{red}\bar4} + {\color{red}3}\bar21{\color{red}\bar4} + \bar2{\color{red}3\bar4}1 + {\color{red}3}\bar2{\color{red}\bar4}1 + {\color{red}3\bar4}\bar21, \]
\[ \bar21 \Cup {\color{red}1\bar 2} = \bar21{\color{red}3\bar 4} + \bar31{\color{red}2\bar 4} + \bar32{\color{red}1\bar 4} + \bar41{\color{red}2\bar 3} + \bar42{\color{red}1\bar 3} + \bar43{\color{red}1\bar 2}, \]
\[ \unshuffle(\bar24\bar31)=\varnothing\otimes \bar24\bar31+\bar1\otimes 3\bar21+\bar12\otimes \bar21+\bar13\bar2\otimes 1+\bar24\bar31\otimes\varnothing,\]
\[ \Cap\,(\bar24\bar31) = \varnothing\otimes\bar24\bar31 + 1\otimes\bar13\bar2 + \bar21\otimes2\bar1 + \bar2\bar31\otimes1 + \bar24\bar31\otimes\varnothing.\]

\begin{proposition}\label{prop:MRB}
$(\ZZ\B,\shuffle,\unshuffle)$ and $(\ZZ\B,\Cup,\Cap)$ are graded Hopf algebras isomorphic to each other via $w\mapsto w^{-1}$, $\forall w\in\B$, and dual to each other via $\langle u,v\rangle:=\delta_{u,v}$, $\forall u,v\in\B$.
\end{proposition}
\begin{proof}
The definition of $\shuffle$ implies $w\shuffle\varnothing=\varnothing\shuffle w=w$ for all $w\in \B$.
If $u\in\B_m$, $v\in\B_n$, and $r\in\B_k$ then one can check that $(u\shuffle v)\shuffle w$ and $u\shuffle(v\shuffle w)$ both equal the sum of all $w\in\B_{m+n+k}$ such that
\[ w|[1,m]=u,\quad \stB(w|[m+1,m+n])=v,\qand \stB(w|[m+n+1,m+n+k])=r. \]
Thus $(\ZZ\B,\shuffle)$ is a graded algebra, whose dual is the graded coalgebra $(\ZZ\B,\Cap)$. 
Applying $w\mapsto w^{-1}$ to $(\ZZ\B,\shuffle)$ and $(\ZZ\B,\Cap)$ gives the isomorphic graded algebra $(\ZZ\B,\Cup)$ and its dual graded coalgebra $(\ZZ\B,\unshuffle)$.
One checks that $\unshuffle(u\shuffle v)$ and $(\unshuffle u)\shuffle(\unshuffle v)$ both equal
\[\sum_{0\leq i\leq m}\sum_{0\leq j\leq n} \left( \stB(u[1,i])\shuffle\stB(v[1,j]) \right) \otimes \left( \stB(u[i+1,m]) \shuffle \stB(v[j+1,n]) \right) \]
where the shuffle product $(\unshuffle u)\shuffle(\unshuffle v)$ is defined  tensor-component-wise. 
It follows that $(\ZZ\B,\shuffle,\unshuffle)$ is a Hopf algebra and so is the isomorphic $(\ZZ\B,\Cup,\Cap)$.
 \end{proof}

Next, let $ab$ be the concatenation of two words $a$ and $b$ in $\ZZ^n$. Assume $\stB(ab)=w$ and $u=\stB(w[1,m])$. Since $|w(1)|,\ldots,|w(m)|$ are distinct, one sees that $|u(i)|<|u(j)|$ if and only if $|w(i)|<|w(j)|$ whenever $i,j\in[m]$. Thus  $\stB(a)=u=\stB(w[1,m])$ by restricting the definition of $\stB(ab)=w$ to $i,j\in [m]$. Similarly one has $\stB(b)=\stB(w[m+1,m+n])$. Hence for all $u,v\in\B$, Proposition~\ref{prop:bsB} implies
\begin{equation}\label{eq:Prod bsB}
\bsB_u\cdot\bsB_v=\sum_{\stB(f)=u}\sum_{\stB(g)=v} \bx_{f(1)}\cdots \bx_{f(m)}\bx_{g(1)}\cdots\bx_{g(n)}=\sum_{w\in\, u\,\Cup\, v}\bsB_w. 
\end{equation}
If $w\in \B_n$ then define
\begin{equation}\label{eq:Delta bsB}
\Delta\,\bsB_w :=\sum_{0\leq i\leq n}  \bsB_{\,w|[1,i]}\otimes\bsB_{\,\stB(w|[i+1,n])}.
\end{equation}

\begin{proposition}
There are Hopf algebra isomorphisms $(\FQSym^B,\cdot,\Delta)\cong(\ZZ\B,\shuffle,\unshuffle)$ via $\bFB_w\mapsto w$, $\forall w\in\B$, and $(\FQSym^B,\cdot,\Delta)\cong(\ZZ\B,\Cup,\Cap)$ via $\bsB_w\mapsto w$, $\forall w\in\B$.
\end{proposition}

\begin{proof}
Comparing (\ref{eq:Prod bsB}) and (\ref{eq:Delta bsB}) with the definition of $(\ZZ\B,\Cup,\Cap)$ gives the second isomorphism. 
Applying $w\mapsto w^{-1}$ gives the first one.
 \end{proof}

\begin{remark}
Novelli and Thibon~\cite{NovelliThibon} generalized $\FQSym$ to a family of graded Hopf algebras $\FQSym^{(\ell)}$ using an $\ell$-colored standardization. The Hopf algebra $(\FQSym^B,\cdot,\Delta)$ is isomorphic to $\FQSym^{(2)}$, but the dual bases $\{\bFB_w\}$ and $\{\bsB_w\}$ for $\FQSym^B$ are different from the dual bases for $\FQSym^{(2)}$ provided in~\cite{NovelliThibon}. In fact, $\bFB_w$ is the generating function of $\mathcal A(w\Phi^+)$, and this forces $\bsB_w:=\bFB_{w^{-1}}$ to involve the signed standardization, which is not a $2$-colored standardization defined in~\cite{NovelliThibon}.
\end{remark}

By the duality between $(\ZZ\B,\shuffle,\unshuffle)$ and $(\ZZ\B,\Cup,\Cap)$, if $u\in\B_m$ and $v\in\B_n$ then
\[ \bFB_u\cdot\bFB_v=\sum_{w\in u\shuffle v}\bFB_w \qand
\Delta\,\bFB_u=\sum_{0\le i\le m} \bFB_{\,\stB(u[1,i])}\otimes\bFB_{\,\stB(u[i+1,n])}.\]
One can also directly verify the first equation.
Applying $\chiB:\FQSym^B\twoheadrightarrow\QSym^B$ gives 
\[\FB_u\cdot \FB_v=\sum_{w\in u\shuffle v} \FB_w, \qquad \forall u, v\in\B.\]
This recovers a graded algebra structure for $\QSym^B$ studied by Chow~\cite{Chow}.

Next, we apply the embeddings $\B_m\otimes \B_n\hookrightarrow \B_{m+n}$ and $\B_m\otimes \SS_n\hookrightarrow \B_{m+n}$ to the (complex) representation theory of hyperoctahedral groups. 
We first review the representation theory of the hyperoctahedral group $\B_n$ from Geissinger and Kinch~\cite{GeissingerKinch}. Let $E(n)$ be the subgroup of $\B_n$ generated by the signed permutations $e(n,i):=1\cdots\bar i\cdots n$ for all $i\in[n]$. Then $\B_n$ is a semidirect product of $\SS_n$ and $E(n)$. For every $k\in[n]$ there is a one-dimensional $E(n)$-module $\chi(n,k)$ on which $e(n,i)$ acts by $-1$ if $1\leq i\leq k$ or by $1$ if $k<i\leq n$. Let $\B_{k,n-k}$ be the subgroup of $\B_n$ consisting of those signed permutations $w\in\B_{n}$ satisfying $|w(i)|\leq k$ for all $i\in[k]$ and $|w(i)|>k$ for all $i\in[k+1,n]$. Then $\B_{k,n-k}\cong\B_k\times\B_{n-k}$. One sees that $\B_{k,n-k}$ is the semidirect product of $\SS_{k,n-k}\cong\SS_k\times\SS_{n-k}$ and $E(n)$. Hence one can define 
\[
S_{\mu,\nu} := 
\left( \chi(n,k)\otimes S_\mu\otimes S_\nu \right)\,\uparrow\,_{\B_{k,n-k}}^{\B_n}
\]
where $(\mu,\nu)$ is a \emph{double partition} of $n$, i.e. an ordered pair of partitions $(\mu,\nu)$ such that $|\mu|+|\nu|=n$. A complete list of pairwise non-isomorphic simple $\CC\B_n$-modules are given by $S_{\mu,\nu}$ for all double partitions of $(\mu,\nu)$ of $n$.

The Grothendieck group $G_0(\CC\B_\bullet)$ of the tower $\CC\B_0\hookrightarrow \CC\B_1\hookrightarrow \CC\B_2\hookrightarrow\cdots$ of algebras has a product and a coproduct given by the induction and restriction along the embeddings $\B_k\times \B_{n-k}\hookrightarrow\B_n$. This gives a Hopf algebra structure on $G_0(\CC\B_\bullet)$, which turns out to be isomorphic to $G_0(\CC\SS_\bullet)\otimes G_0(\CC\SS_\bullet)$ via $\phi: S_\mu\otimes S_\nu\mapsto S_{\mu,\nu}$ for all double partitions $(\mu,\nu)$~\cite[\S4.5]{GrinbergReiner}. The self-duality of the Hopf algebra $G_0(\CC\B_\bullet)$ is given by the pairing $\langle S_{\mu,\nu},S_{\zeta,\eta} \rangle:= \delta_{\mu,\zeta}\delta_{\nu,\eta}$ for all double partitions $(\mu,\nu)$ and $(\zeta,\eta)$. We denote by $\cdot$ and $\Delta$ the product and coproduct of both $G_0(\CC\SS_\bullet)$ and $G_0(\CC\B_\bullet)$.

Now we define a right action and coaction of $G_0(\CC\SS_\bullet)$ on $G_0(\CC\B_\bullet)$ by
\[ S_{\mu,\nu}\odotB S_\lambda := (S_{\mu,\nu}\otimes S_{\lambda}) \uparrow\,_{\B_{m}\times\SS_{n}}^{\B_{m+n}} \qand
\DeltaB(S_{\mu,\nu}) := \bigoplus_{0\leq i\leq m}\ S_{\mu,\nu}\downarrow\,_{\B_i\times\SS_{m-i}}^{\B_{m}} \]
for all double partitions $(\mu,\nu)$ of $m$ and all partitions $\lambda$ of $n$.

\begin{proposition}
$(G_0(\CC\B_\bullet),\odotB,\DeltaB)$ is a dual graded right module and comodule over $G_0(\CC\SS_\bullet)$ such that if $(\mu,\nu)$ is a double partition and $\lambda$ is a partition then
 \[ S_{\mu,\nu}\odotB S_\lambda = S_{\mu,\nu}\cdot \phi(\Delta(S_\lambda)).\]
\end{proposition}

\begin{proof}
Let $(\zeta,\eta)$ be a double partition of $n$ with $|\zeta|=k$. By Mackey's formula~\cite[\S4.1.5]{GrinbergReiner}, 
\[ \left( \chi(n,k)\otimes S_\zeta\otimes S_\eta \right)\,\uparrow\,_{\B_{k,n-k}}^{\B_n} \downarrow\,_{\SS_n}^{\B_n} 
\cong \left( \chi(n,k)\otimes S_\zeta\otimes S_\eta \right) \downarrow\,_{\SS_{k,n-k}}^{\B_{k,n-k}} \uparrow\,_{\SS_{k,n-k}}^{\SS_n}
\]
since there is only one double $(\SS_n,\B_{k,n-k})$-coset in $\B_n$. Hence
\[
S_{\zeta,\eta} \downarrow\,_{\SS_n}^{\B_n} \cong (S_\zeta \otimes S_\eta) \uparrow\,_{\SS_{k,n-k}}^{\SS_n} \cong S_\zeta\cdot S_\eta.
\]
Using Frobenius reciprocity, the self-duality of $G_0(\CC\SS_\bullet)$, and the Hopf algebra isomorphism $\phi$, one obtains
\[
\langle S_\lambda \uparrow\,_{\SS_n}^{\B_n}, S_{\zeta,\eta} \rangle =
\langle S_\lambda, S_\zeta\cdot S_\eta \rangle = 
\langle \Delta(S_\lambda), S_\zeta\otimes S_\eta \rangle = 
\langle \phi(\Delta(S_\lambda)), S_{\zeta,\eta} \rangle
\]
for every partition $\lambda$ of $n$. Therefore for any double partition $(\mu,\nu)$ of $m$ one has
\[
S_{\mu,\nu}\odotB S_\lambda :=
(S_{\mu,\nu}\otimes S_{\lambda}) \uparrow\,_{\B_{m}\times\SS_{n}}^{\B_{m+n}} 
= S_{\mu,\nu}\cdot (S_{\lambda} \uparrow\,_{\SS_{n}}^{\B_{n}})
= S_{\mu,\nu}\cdot \phi(\Delta(S_\lambda)).
\]
Using this formula we show that $(G_0(\CC\B_\bullet),\odotB)$ is a graded right $G_0(\CC\SS_\bullet)$-module. It is clear that $S_{\mu,\nu}\odotB S_\varnothing = S_{\mu,\nu}$ where $\varnothing$ is the empty partition. If $\xi$ is another partition then \begin{eqnarray*}
S_{\mu,\nu}\odotB (S_\lambda \cdot S_\xi) &=&
S_{\mu,\nu}\cdot \phi(\Delta(S_\lambda \cdot S_\xi)) \\
&=& S_{\mu,\nu}\cdot \phi(\Delta(S_\lambda )) \cdot \phi(\Delta(S_\xi))\\
&=& (S_{\mu,\nu}\odotB S_\lambda) \odotB S_\xi. 
\end{eqnarray*}
Hence $(G_0(\CC\B_\bullet),\odotB)$ is a graded right module over $G_0(\CC\SS_\bullet)$.
Finally, by the Frobenius reciprocity, $(G_0(\CC\B_\bullet),\DeltaB)$ is dual to $(G_0(\CC\B_\bullet),\odotB)$, and hence a graded right comodule over $G_0(\CC\SS_\bullet)$. 
 \end{proof}

Recall from Proposition~\ref{prop:SymB} that $\Sym^B$ is a graded module and comodule over $\Sym$. Even though there is a Hopf algebra isomorphism $\Sym\cong G_0(\CC\SS_\bullet)$, there is no $\ZZ$-module isomorphism between $G_0(\CC\B_\bullet)$ and $\Sym^B$. One can also check that $(G_0(\CC\B_\bullet),\odotB,\DeltaB)$ is \emph{not} a Hopf module over the Hopf algebra $G_0(\CC\SS_\bullet)$.


\section{Type D}\label{sec:D}

In this section we apply  our results in Section~\ref{sec:General} to type D and obtain some new results.

\subsection{Malvenuto--Reutenauer algebra and descent algebra of type D}\label{sec:MRDA}

The hyperoctahedral group $\B_n$ admits a subgroup $\D_n$ consisting of all signed permutations $w\in\B_n$ with $\neg(w)$ even. 
When $n\ge2$ this subgroup $\D_n$ has a generating set $S$ consisting of $s_0^D:=\bar2\bar13\cdots n$ and $s_i:=[1,\ldots,i-1,i+1,i,i+2,\ldots,n]$ for all $i\in[n-1]$. We write $s_0=s_0^D$ and $s_0^B=\bar12\cdots n$ in this section. The pair $(\D_n,S)$ is the finite irreducible Coxeter system of type $D_n$ whose Coxeter diagram is illustrated below.
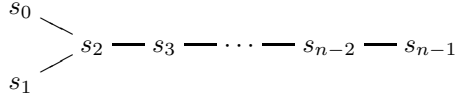
\[ \xymatrix @R=2pt @C=12pt{
s_0 \ar@{-}[rd] \\
& s_2 \ar@{-}[r] & s_3 \ar@{-}[r] & \cdots \ar@{-}[r] & s_{n-2} \ar@{-}[r] & s_{n-1} \\
s_1 \ar@{-}[ru]
} \]

Let $w\in\D_n$ and set $w(0):=-w(2)$. The descent set of $w$ consists of all $s_i$ such that $i\in\{0,1,\ldots,n-1\}$ and $w(i)>w(i+1)$. 
The length of $w$ equals $\inv(w)+\nsp(w)$.
Subsets of $S=\{s_0,s_1,\ldots,s_{n-1}\}$ are indexed by pseudo-compositions of $n$. 
For consistency of notation we write $\alpha\modelsD n$ for a pseudo-composition $\alpha$ of $n$ in this section.
If $\alpha\modelsD n$ then the parabolic subgroup $\D_\alpha$ of $\D_n$ is generated by $\{s_i:0\le i\le n-1,\ i\notin D(\alpha)\}$ and the set of minimal representatives for left $\D_\alpha$-cosets in $\D_n$ is 
\[ (\D)^\alpha:=\{ w\in D_n: D(w)\subseteq D(\alpha)\}.\]

We define two \emph{type D standardizations} $\Dst(a)$ and $\stD(a)$ of a word $a\in \ZZ^n$ as follows. If $\stB(a)\in \D_n$ then let $\Dst(a)=\stD(a):=\stB(a)$; otherwise let $\Dst(a):=s_0^B\stB(a)$ and $\stD:=\stB(a)s_0^B$. Then $\Dst(a)$ and $\stD(a)$ are both elements of $\D_n$. For example, one has
\[ \stB(211\bar32\bar1) = \Dst(211\bar32\bar1) =  \stD(211\bar32\bar1) = 423\bar65\bar1 \qand \]
\[ \stB(21\bar1\bar32\bar1) = 43\bar2\bar65\bar1, \quad \Dst(21\bar1\bar32\bar1) = 43\bar2\bar65{\color{red}1},\quad \stD(21\bar1\bar32\bar1) = {\color{red}\bar4}3\bar2\bar65\bar1. \]

Assume $m\ge2$ and $n\ge0$. 
One sees that the embedding $\B_m\times\SS_n\hookrightarrow \B_{m+n}$ restrict to an embedding $\D_m\times\SS_n\hookrightarrow\SS_{m+n}$, which identifies $(u,v)\in \D_m\times\SS_n$ with an element $u\times v$ of $\D_{m+n}$ whose window notation is $u\times v := [u(1),\ldots,u(m),m+v(1),\ldots,m+v(n)]$.
The image of this embedding is the parabolic subgroup $\D_{m,n}$ of $\D_{m+n}$ generated by $\{s_i:0\le i\le n-1, i\ne m \}$.
The set of minimal representatives for left $\D_{m,n}$-cosets in $\D_{m+n}$ is 
\[ (\D)^{m,n}:=\{ z\in\D_{m+n}: -z(2)<z(1)<\cdots<z(m),\ z(m+1)<\cdots<z(m+n)\}.\]

\begin{proposition}\label{prop:DA}
Assume $m\ge2$ and $n\ge0$. Then every element $w\in\D_{m+n}$ can be written uniquely as $w=(u\times v)z$ where $u\in \D_m$, $v\in \SS_n$, and $z^{-1}\in(\D)^{m,n}$. Moreover, one has 
\[\begin{matrix}
u & = & \stD(w|[1,m]),\quad & v & = & \st (\hat{w}|[m+1,m+n]), \\
u^{-1} & = & \Dst(w^{-1}[1,m]),\quad & v^{-1} & = & \st(w^{-1}[m+1,m+n]).
\end{matrix}\]
\end{proposition}

\begin{proof}
Applying Proposition~\ref{prop:parabolic} to the parabolic subgroup $\D_{m,n}$ shows that every element $w\in \D_{m+n}$ can be written uniquely as $w=(u\times v)z$ where $u\in \D_m$, $v\in \SS_n$, and $z^{-1}\in(\D)^{m,n}$. By Proposition~\ref{prop:BA} one also has $w= (u' \times v')z'$ where 
\[ u'=w|[1,m] \in \B_m,\quad v'=\st(\hat w|[m+1,m+n])\in\SS_n,\qand (z')^{-1}\in (\B)^{m,n}. \]

First assume $u'\in \D_m$. 
Since $\neg(z')\equiv \neg(u')\equiv0$ (mod $2$) and $-(z')^{-1}(2)<0<(z')^{-1}(1)$, one has $z'\in (\D)^{m,n}$.
It follows that 
\[ u=u'=\stD(w|[1,m]), \quad v=v', \qand z=z'.\]
Since $\neg((u')^{-1})=\neg(u')$, Proposition~\ref{prop:BA} implies
\[ u^{-1} 
=\stB(w^{-1}[1,m])=\Dst(w^{-1}[1,m]) 
\qand v^{-1} 
=\st(w^{-1}[m+1,m+n]).\]

Next assume $u'\notin \D_m$. Then $w=(u's^B_0\times v)s^B_0z'$. Since $\neg(u')\equiv \neg(z')\equiv1$ (mod 2), one has $u's^B_0\in\D_m$ and $s^B_0z'\in \D_{m+n}$. Since $(s^B_0z')^{-1}(1) = -(z')^{-1}(1)$ and $(s^B_0z')^{-1}(i) = (z')^{-1}(i)$ if $2\le i\le m+n$, one sees that $(z')^{-1}\in(\B)^{m,n}$ implies $(s^B_0z')^{-1}\in(\D)^{m,n}$. Hence
\[ u=u's^B_0=\stD(w|[1,m]),\quad v=v', \qand z=s^B_0z'.\] 
Since $\neg((u')^{-1})=\neg(u')\equiv 1$ (mod $2$), Proposition~\ref{prop:BA} implies
\[ u^{-1}=s^B_0(u')^{-1}= \Dst(w^{-1}[1,m]) \qand v^{-1}=\st(w^{-1}[m+1,m+n]).\]
This completes the proof.
 \end{proof}

For example, $w= 2\bar51\bar34 = (\bar21\bar3\times21)\cdot\bar1\bar4235$ and $w^{-1}=31\bar45\bar2=\bar134\bar25\cdot(2\bar1\bar3\times21)$.\vskip3pt

Let $\D:=\bigsqcup_{\,n\ge2}\D_n$. We give $\ZZ\D=\bigoplus_{n\geq2}\ZZ\D_n$ a dual graded right module and comodule structure over the Malvenuto--Reutenauer algebra $\ZZ\SS$. 
Assume $m\ge2$ and $n\ge0$. Let $u\in \D_m$ and $v\in\SS_n$. We define
\[ u\shuffleD v := \left\{ w\in\D_{m+n}: \stD(w|[1,m])=u,\ \st (\hat w|[m+1,m+n] )=v \right\}, \]
\[ u\CupD v := \left\{ w\in\D_{m+n}:\Dst ( w[1,m] )=u,\ \st( w[m+1,m+n] )=v \right\}, \]
\[ \unshuffleD u :=\sum_{2\le i\le m} \Dst (u[1,i])\otimes\st (w[i+1,m] ), \]
\[ \CapD u := \sum_{2\le i\le m} \stD(u|[1,i]) \otimes \st( \hat{u}|[i+1,m] ). \]
For example, one has
\[ {\color{blue}\bar2}3\bar1\shuffleD {\color{red}1} =
{\color{blue}\bar2}3\bar1{\color{red}4}+{\color{blue}\bar2}3{\color{red}4}\bar1
+{\color{blue}\bar2}{\color{red}4}3\bar1+{\color{red}4}{\color{blue}\bar2}3\bar1
+{\color{blue}2}3\bar1{\color{red}\bar4}+{\color{blue}2}3{\color{red}\bar4}\bar1
+{\color{blue}2}{\color{red}\bar4}3\bar1+{\color{red}\bar4}{\color{blue}2}3\bar1,\]
\[ \bar23{\color{blue}\bar1}\CupD {\color{red}1} = 
\bar23{\color{blue}\bar1}{\color{red}4}+\bar24{\color{blue}\bar1}{\color{red}3}
+\bar34{\color{blue}\bar1}{\color{red}2}+\bar34{\color{blue}\bar2}{\color{red}1}
+\bar23{\color{blue}1}{\color{red}\bar4}+\bar24{\color{blue}1}{\color{red}\bar3}
+\bar34{\color{blue}1}{\color{red}\bar2}+\bar34{\color{blue}2}{\color{red}\bar1},\]
\[ \unshuffleD 2\bar4\bar31 = 
\bar1\bar2\otimes12 +1\bar3\bar2\otimes1 +2\bar4\bar31\otimes\varnothing,\]
\[ \CapD 2\bar4\bar31 = 
21\otimes12+\bar2\bar31\otimes1+2\bar4\bar31\otimes\varnothing.
\]
Using Proposition~\ref{prop:DA} and the linear maps defined in Definition~\ref{def:MapsZW}, with $S=\{s_0,\ldots,s_{m+n-1}\}$ and $I=S\setminus\{s_m\}$, one has
\[ u\shuffleD v = \,\bar\mu_I^S(u\times v),\qquad \unshuffle^D_m w = \,\bar\rho_I^S(w), \]
\[ u\CupD v = \mu_I^S(u\times v),\qquad \Cap^D_m w = \rho_{I}^S(w).\]
where $w\in\D_{m+n}$, $\unshuffle^D_m w$ is the $m$-th term in $\unshuffleD\! w$, and similarly for $\Cap^D_m w$.

\begin{proposition}\label{prop:MRDA}
(i) $(\ZZ\D\!,\shuffleD\!,\unshuffleD\!)$ is a graded right module and comodule over the graded Hopf algebra $(\ZZ\SS,\shuffle,\unshuffle)$.

\noindent(ii) $(\ZZ\D,\CupD,\CapD)$ is a graded right module and comodule over the graded Hopf algebra $(\ZZ\SS,\Cup,\Cap)$.

\noindent(iii) $(\ZZ\D,\shuffleD,\unshuffleD)$ is dual to $(\ZZ\D,\CupD,\CapD)$ via the pairing $\langle u,v\rangle:=\delta_{u,v}$, $\forall u,v\in\D$.

\noindent(iv) $(\ZZ\D,\shuffleD,\unshuffleD)$ is isomorphic to $(\ZZ\D,\CupD,\CapD)$ via $w\mapsto w^{-1}$, $\forall w\in\SS^D\cup\SS$.
\end{proposition}

\begin{proof}
This is similar to the proof of Proposition~\ref{prop:MRBA}.
 \end{proof}

Let $\Sigma(\D):=\bigoplus_{n\ge2}\Sigma(\D_n)$ where $\Sigma(\D_n)$ is a free $\ZZ$-module with a basis consisting of descent classes
\[ D_\alpha(\D_n):= \left\{w\in\D_n: D(w)=D(\alpha) \right\},\quad \forall \alpha\modelsD n\ge2. \]
One has an embedding $\imath:\Sigma(\D)\hookrightarrow \ZZ\D$ by inclusion.  
If $\alpha\modelsD m\ge2$ and $\beta\models n\ge0$ then 
\[ D_\alpha(\D_m) \Cup D_\beta(\SS_n) = D_{\alpha\cdot\beta}(\D_{m+n}) + D_{\alpha\rhd\beta}(\D_{m+n})\]
by Proposition~\ref{prop:IndDes}, where the last term is treated as zero when $\alpha\rhd\beta$ is undefined.

Let $\Sigma^*(\D):=\bigoplus_{n\ge2} \Sigma^*(\D_n)$ where $\Sigma^*(\D_n)$ is the dual of $\Sigma(\D_n)$ with a dual basis $\{D^*_\alpha(\D_n): \alpha\modelsD n\}$.
Dual to $\imath:\Sigma(\D)\hookrightarrow \ZZ\D$ is a surjection $\chi:\ZZ\D\twoheadrightarrow\Sigma^*(\D)$ defined by sending each $w\in\D_n$ to $D^*_w(\D_n):=D^*_\alpha(\D_n)$, where $\alpha\modelsD n$ satisfies $D(w)=D(\alpha)$.

Recall from Section~\ref{sec:Coxeter} that $\Lambda(\D_n)$ be the $\ZZ$-span of $\Lambda_\alpha(\D_n):=\chi'(D_\alpha(\D_n))$ for all $\alpha\modelsD n$, where $\chi':=\chi\circ(\ )^{-1}$.
For $\alpha\modelsD m$ and $\beta\modelsD n$ we define 
\[ \langle \Lambda_\alpha(\D_m),\Lambda_\beta(\D_n) \rangle := 
\begin{cases} \#\{w\in\D_n: D(w^{-1}) = D(\alpha),\ D(w)=D(\beta) \}, &\text{if } m=n, \\
0, & \text{if } m\ne n.
\end{cases} \]
By Proposition~\ref{prop:SymForm}, this gives a well-defined nondegenerate symmetric bilinear form on the free $\ZZ$-module $\Lambda(\D):=\bigoplus_{n\ge2}\Lambda(\D_n)$ such that $\imath:\Lambda(\D)\hookrightarrow\Sigma^*(\D)$ and $\chi':\Sigma(\D)\twoheadrightarrow\Lambda(\D)$ are dual to each other. 

\begin{theorem}\label{thm:DesD}
The following diagram is commutative with each entry being a graded right module and comodule over the corresponding type A Hopf algebra in \eqref{eq:DiamondCoxA}. 
\begin{equation}\label{eq:DiamondCoxD}
\xymatrix @R=16pt @C=7pt {
 & \ZZ\D \ar@{->>}[rd]^{\chi'} \\
 \Sigma(\D) \ar@{^(->}[ru]^{\imath} \ar@{->>}[rd]_{\chi'} & & \Sigma^*(\D) \ar@{<-->}[ll]^{\txt{\small dual}} \\
 & \Lambda(\D) \ar@{^(->}[ru]_\imath }
\end{equation}
Reflecting it across the vertical line through $\ZZ\D$ and $\Lambda(\D)$ gives a dual diagram. 
\end{theorem}

\begin{proof}
Apply Theorem~\ref{thm:FullDiagram} to $(\ZZ\D,\shuffleD,\unshuffleD,\CupD,\CapD)$ and then use Corollary~\ref{cor:Sym}.
 \end{proof}


\subsection{Free quasisymmetric functions of type D and related results}\label{sec:FQSymD}
In this subsection we obtain the following commutative diagram.
\begin{equation}\label{eq:DiamondFQSymD}
\DiamondFQSym{D}
\end{equation}
It is isomorphic to the diagram \eqref{eq:DiamondCoxD}, with each entry being a graded right module and comodule over the corresponding type A Hopf algebra in \eqref{eq:DiamondFQSymA}. Reflecting it across the vertical line through $\FQSym^D$ and $\Sym^D$ gives a dual diagram of graded modules and comodules.

\subsubsection{Free quasisymmetric functions of type D}
Let $(W,S)$ be the Coxeter system of type $D_n$, where $W=\D_n$ and $S=\{s_0=s^D_0,s_1,\ldots,s_{n-1}\}$, with $n\ge2$. Let $E=\mathbb R^n$ be a Euclidean space with a standard basis $\{e_1,\ldots,e_n\}$. The group $\D_n$ can be realized as a reflection group of $E$ whose root system $\Phi$ is the disjoint union of $\Phi^+ =\{e_j\pm e_i: 1\leq i<j\leq n\}$ and $\Phi^-=-\Phi^+$. The set $\Delta=\{e_1+e_2,e_2-e_1,\ldots,e_n-e_{n-1}\}$ of simple roots corresponds to the generating set $S$ of simple reflections. 

Let $\bX=\{\bx_i:i\in\ZZ\}$ be a set of noncommutative variables. 
We define $\FQSym^D_n$ to be the $\ZZ$-span of $\bF_P^S$ for all parsets $P$ of $\Phi$ and define $\FQSym^D:=\bigoplus_{n\geq2}\FQSym^D_n$. 
By Proposition~\ref{prop:FQSymW}, $\FQSym^D$ has free $\ZZ$-bases $\{\bFD_w:w\in\D\}$ and $\{\bsD_w:w\in\D\}$, where $\bFD_w$ is the generating function of $w\Phi^+$ and $\bsD_w:=\bFD_{w^{-1}}$. Applying the definition of $f\in\mathcal A(w\Phi^+)$ to $w\alpha$ for all $\alpha\in\Delta$ one sees that  
\begin{equation}\label{eq:bFD}
\bFD_w = \sum_{ \substack{ 
f(w(0))\leq f(w(1))\leq\cdots\leq f(w(n)) \\
i\in D(w) \Rightarrow f(w(i))<f(w(i+1)) }}
\bx_{f(1)}\cdots \bx_{f(n)},\qquad\forall w\in\D_n.
\end{equation}
Here we set $w(0)=-w(2)$ and $f(-i)=(i)$ for all $i\in[n]$ by convention.

\begin{proposition}\label{prop:bsD}
If $w\in\D_n$ ($n\ge2$) and $f\in\ZZ^n$ then $f\in\mathcal A(w^{-1}\Phi^+) \Leftrightarrow \Dst(f)=w$ and thus 
\[ \bsD_w = \sum_{f\in\ZZ^n:\Dst(f)=w} \bx_f. \]
\end{proposition}

\begin{proof}
Let $f\in\ZZ^n$.
By Lemma~\ref{lem:Z^n}, there exists a unique $w\in\D_n$ such that $f\in\mathcal A(w^{-1}\Phi^+)$, i.e., the following holds for all $\alpha=e_j\pm e_i$, $1\le i<j\le n$:
\begin{equation}\label{eq:bsD}
\begin{cases}
(f,w^{-1}\alpha)\geq0, & {\rm if}\ w^{-1}\alpha>0,\\
(f,w^{-1}\alpha)>0, & {\rm if}\ w^{-1}\alpha<0.
\end{cases} 
\end{equation}
Let $u=\stB(f)$. Since a positive root of $\D_n$ is also a positive root of $\B_n$, Proposition~\ref{prop:bsB} implies  \eqref{eq:bsD} if $w=u$. 
Thus if $u\in\D_n$ then \eqref{eq:bsD} holds for $w=u=\Dst(f)$. 
On the other hand, if $u\notin\D_n$ then \eqref{eq:bsD} still holds for $w=s^B_0u=\Dst(f)$, since if $\alpha=e_j\pm e_1$, where $1<j\le n$, then $w^{-1}(\alpha) = u^{-1}(e_j\mp e_1)$, and if $\alpha=e_j\pm e_i$, where $1<i<j\le n$, then $w^{-1}(\alpha)= u^{-1}(\alpha)$. Hence the result holds. 
 \end{proof}

\begin{corollary}\label{cor:bFD}
If $w\in\D_n$ ($n \ge2$) then $\bsD_w = \bsB_w+\bsB_{s_0^Bw}$ and $\bFD_w = \bFB_w+\bFB_{ws_0^B}$. 
Consequently, one has $\FQSym^D \subseteq \FQSym^B$.
\end{corollary}  

\begin{proof}
Let $w\in\D_n$ ($n \ge2$). It follows from Proposition~\ref{prop:bsD} that
\[ \bsD_w = \sum_{\Dst(f)=w} \bx_f = \sum_{\stB(f)=w} \bx_f + \sum_{\stB(f)=s_0^Bw} \bx_f  = \bsB_w+\bsB_{s_0^Bw}\]
and applying $w\mapsto w^{-1}$ gives $\bFD_w = \bFB_w+\bFB_{ws_0^B}$. Hence $\FQSym^D \subseteq \FQSym^B$.
 \end{proof}

Let $m\ge2$ and $n\ge0$. 
By Proposition~\ref{prop:shuffleW}, if $u\in\D_m$ and $v\in \SS_n$ then 
\begin{equation}\label{eq:Action bFD}
\bFD_u \cdot \bF_v = \sum_{w\in\, u\shuffleD v} \bFD_w.
\end{equation}
This gives a right action of $\FQSym$ on $\FQSym^D$. For consistency of notation we denote this action by ``$\odotD$''. Note that $\bF_v$ is $\bF_v(\bX)$ instead of $\bF_v(\bX_{>0})$ by our convention in this paper. If $u\in\D_m$ then define
\begin{equation}\label{eq:Coaction bFD}
\Delta^D(\bFD_u) := \sum_{2\leq i\leq m} \bFD_{\Dst(u[1,i])}\otimes\bF_{\st(u[i+1,n])}.
\end{equation}
This gives a right coaction of $\FQSym$ on $\FQSym^D$. 

\begin{proposition}\label{prop:FQSymD}
$(\FQSym^D,\odotD,\Delta^D)$ is a self-dual graded right module and comodule over $\FQSym$ isomorphic to $(\ZZ\D,\shuffleD,\unshuffleD )$ via $\bFD_w\mapsto w$ and to $(\ZZ\D,\CupD,\CapD)$ via $\bsD_w\mapsto w$. 
\end{proposition}

\begin{proof}
This follows from \eqref{eq:Action bFD}, \eqref{eq:Coaction bFD} and Proposition~\ref{prop:MRDA} (iv).
 \end{proof}

It follows that if $u\in\D_m$ ($m\ge2$) and $v\in\SS_n$ then
\begin{equation}\label{eq:Coaction bsD}
\bsD_u\cdot \bs_v = \sum_{w\in u\CupD v}\bsD_w \qand
\Delta^D(\bsD_u) = \sum_{2\leq i\leq m} \bsD_{\stD(u|[1,i])}\otimes\bs_{\st(\hat u|[i+1,n])}.
\end{equation}

\begin{remark}
Given $u\in\D_m$ ($m\ge2$), it follows from \eqref{eq:Coaction bFB}, \eqref{eq:Coaction bsB}, and Corollary~\ref{cor:bFD} that
\[ \DeltaB(\bFD_u) = \sum_{2\le i\le m} \bFD_{\stD(u[1,i])} \otimes\bF_{\st(u[i+1,m])} \qand \DeltaB(\bsD_u) = \sum_{2\le i\le m} \bsD_{\Dst(u|[1,i])} \otimes\bs_{\st(\hat u|[i+1,m])}. \]
This does not give the desired coaction of $\FQSym$ on $\FQSym^D$.
\end{remark}

\subsubsection{Noncommutative symmetric functions of type D} 
Let $\alpha\modelsD n\ge2$. A tableau $\tau$ of  pseudo-ribbon shape $\alpha$ is \emph{type D semistandard} if each row is weakly increasing from left to right and each column is strictly increasing from top to bottom, \emph{with the extra 0-entry interpreted as $-w(\tau)(2)$}. 
Let $\bsD_\alpha$ be the sum of $\bx_{w(\tau)}$ for all type D semistandard tableaux $\tau$ of pseudo-ribbon shape $\alpha$ and define $\bhD_\alpha:= \sum_{\beta\cleq\alpha}\bsD_\beta$.

In our earlier work~\cite{H0Tab} we defined a type D analogue $\NSym^D$ of $\NSym$, which has two free $\ZZ$-bases consisting of $\bsD_\alpha$ and $\bhD_\alpha$, respectively, for all $\alpha\modelsD n$ and all $n\ge2$. We showed 
\[ \bhD_\alpha\cdot\bh_\beta = \bhD_{\alpha\cdot\beta} \qand
 \bsD_\alpha\cdot\bs_\beta=\bsD_{\alpha\cdot\beta}+\bsD_{\alpha\rhd\beta} \]
for $\alpha\modelsD n\ge2$ and $\beta\models n\ge0$, where the last term is treated as zero when $\alpha\rhd\beta$ is undefined. 
This gives a right $\NSym$-action on $\NSym^D$, which is denoted by $\odotD$ for consistency of notation.
Now we study the relation between $\NSym^D$ and $\FQSym^D$, and use it to obtain a $\NSym$-coaction on $\NSym^D$.
 
\begin{proposition}\label{prop:bsDa}
Let $\alpha\modelsD n\ge2$. Then $\bsD_\alpha$ equals the sum of $\bsD_w$ for all $w\in D_\alpha(\D_n)$.
\end{proposition}

\begin{proof}
Each $f\in\ZZ^n$ corresponds to a unique tableau $\tau$ of shape $\alpha$ such that $w(\tau)=f$. Let $u=\stB(f)$ and $w=\Dst(f)$. By Proposition~\ref{prop:bsD}, it suffices to show that 
\begin{equation}\label{eq:bsDa}
\text{$D(w)=D(\alpha)$ if and only if $\tau$ is type D semistandard.}
\end{equation}
One sees that \eqref{eq:bsDa} holds when $w=u\in\D_n$. If $u\notin D_n$ then $w=s^B_0u$ and one still has \eqref{eq:bsDa} as negating $\pm1$ does not change any descent. 
 \end{proof}

It follows that there is an injection $\imath:\NSym^D\hookrightarrow \FQSym^D$ by inclusion.

\begin{proposition}\label{prop:NSymD}
The graded right module and comodule $\FQSym^D$ over $\FQSym$ restricts to a graded right module and comodule $(\NSym^D,\odotD,\DeltaD)$ over $\NSym$, which is isomorphic to the graded right module and comodule $(\Sigma(\D),\CupD,\CapD)$ over $(\Sigma(\SS),\Cup,\Cap)$ via the map $\bsD_\alpha\mapsto D_\alpha(\D_m)$, $\forall \alpha\modelsD m$, $\forall m\ge2$.
\end{proposition}

\begin{proof}
By Proposition~\ref{prop:FQSymD}, there is an isomorphism $(\FQSym^D,\odotD,\DeltaD)\cong (\ZZ\D,\CupD,\CapD)$ via $\bsD_w\mapsto w$, $\forall w\in\D$.
Restricting this isomorphism to $\NSym^D$ gives the result.
 \end{proof}

\begin{remark}
If $k$ is a nonnegative integer then it follows from \eqref{eq:Coaction bsD} and Proposition~\ref{prop:bsDa} that
\[ \DeltaD(\bhD_k) = \DeltaD(\bsD_{12\cdots k}) = \sum_{0\le i\le k} \bsD_{12\cdots i} \otimes \bs_{\st(i+1,\ldots, k)} = \sum_{0\le i\le k} \bhD_i \otimes \bh_{k-i}. \]
We do not have any explicit formula for $\DeltaD(\bhD_\alpha)$ or $\Delta(\bsD_\alpha)$ for an arbitrary $\alpha\modelsD n\ge2$. 
\end{remark}

\subsubsection{Quasisymmetric functions of type D}
Let $X=\{x_i:i\in\mathbb Z\}$ be a totally ordered set of commutative variables. If $\alpha=(\alpha_1,\ldots,\alpha_\ell)$ is a pseudo-composition of $n\ge2$ and $i_0:=-i_2$ then define
\begin{eqnarray} 
\MD_\alpha &:=& \sum_{\substack{-i_2\leq i_1\leq\cdots\leq i_n \\ j\in D(\alpha)\Leftrightarrow i_j<i_{j+1} }} x_{i_1}x_{i_2}\cdots x_{i_n} \qand \label{eq:MD} \\
\FD_\alpha &:=& \sum_{\substack{-i_2\leq i_1\leq\cdots\leq i_n \\ j\in D(\alpha)\Rightarrow i_j<i_{j+1} }} x_{i_1}x_{i_2}\cdots x_{i_n} = \sum_{\alpha\cleq\beta} \MD_\beta. \label{eq:FD}
\end{eqnarray}
One can check that $\MD_\alpha = \MB_\alpha$ if $1\notin D(\alpha)$. In fact, one has
\[ \MD_\alpha = 
\begin{cases} 
 x_0^{\alpha_1}M_{(\alpha_2,\ldots,\alpha_\ell)} = \MB_\alpha, & \text{if }
\alpha_1\ge2,\\
 \sum_{0<j_2<\cdots<j_\ell} x_{\overline{j_2}}x_{j_2}^{\alpha_2}\cdots x_{j_\ell}^{\alpha_\ell}, & \text{if } \alpha_1=1, \\
  M_{(\alpha_2,\ldots,\alpha_\ell)} = \MB_\alpha, & \text{if } \alpha_1=0 \text{ and } \alpha_2\ge2, \\
\sum_{-j_3<j_2<j_3<\cdots<j_\ell} x_{j_2}x_{j_3}^{\alpha_3}\cdots x_{j_\ell}^{\alpha_\ell}, & \text{if } \alpha_1=0 \text{ and } \alpha_2=1.
\end{cases} \] 

In our earlier work~\cite{H0Tab} we defined a type D analogue $\QSym^D$ of $\QSym$, which admits two free $\ZZ$-bases consisting of $\MD_\alpha$ and $\FD_\alpha$, respectively, for all $\alpha\modelsD n$ and all $n\ge2$. 

Recall that $\chiB:\FQSym^B\twoheadrightarrow \NSym^B$ is defined by $\bx_i\mapsto x_{|i|}$ for all $i\in\ZZ$. If $f\in\ZZ^n$ one can write $\chiB(\bx_f)=x_{i_1}\cdots x_{i_n}$ such that $0\le i_1\le \cdots\le i_n$, and we define $\chiD(\bx_f):=x_{\pm i_1} x_{i_2}\cdots x_{i_n}$ where the sign of $i_1$ is the same as $(-1)^{\neg(f)}$. This gives a linear maps $\chiD:\mathbb Z\langle \bX\rangle \to \mathbb Z[X]$, which is not an algebra homomorphism, as $\chiD(\bx_{\bar2}\bx_1) = x_{\bar1}x_2 \ne x_{\bar2}x_1 = \chiD(\bx_{\bar2})\chiD(\bx_{1})$.

\begin{proposition}
Let $w\in\D_n$ and $\alpha\modelsD n\ge2$ with $D(w)=D(\alpha)$. Then
$\FD_w:=\chiD(\bFD_w)=\FD_\alpha$. 
\end{proposition}

\begin{proof}
Suppose that $f\in\mathcal A(w\Phi^+)$. Then $-f(w(2))\le f(w(1))\le f(w(2))\le \cdots\le f(w(n))$. This implies $0\le |f(w(1)|\le f(w(2))\le \cdots\le f(w(n))$. Since $\neg(w)$ is even, it follows that $f(w(1))$ has the same sign as $(-1)^{\neg(f)}$. Therefore $\chiD(\bx_f)=x_{f(w(1))}\cdots x_{f(w(n))}$. 
Then comparing \eqref{eq:bFD} with \eqref{eq:FD} gives the result.
 \end{proof}

Let $X+Y_{>0}=\{\ldots,x_{-2},x_{-1},x_0,x_1,x_2,\ldots,y_1,y_2,\ldots\}$ be a totally ordered set of commutative variables. We defined in~\cite{H0Tab} a right coaction $\DeltaD$ of $\QSym$ on $\QSym^D$ by 
\[ \QSym^D \to\QSym^D(X+Y_{>0})\twoheadrightarrow\QSym^D\otimes\QSym.\] 
The second map above is induced by the canonical projection
\[ \ZZ[X+Y_{>0}]\cong\ZZ[X]\otimes\ZZ[Y_{>0}]\twoheadrightarrow \ZZ[X]_{\ge2}\otimes\ZZ[Y_{>0}]\]
where $\ZZ[X]_{\ge2}$ is the $\ZZ$-span of those polynomials in X with degree at least 2. If $\alpha=(\alpha_1,\ldots,\alpha_\ell)\modelsD n\ge2$ and $k$ is the smallest integer such that $\alpha_1+\cdots+\alpha_k\ge2$ then 
\begin{equation}\label{eq:Coaction MD FD}
\DeltaD \MD_\alpha = \sum_{k\le j\le\ell} \MD_{(\alpha_1,\ldots,\alpha_j)}\otimes M_{(\alpha_{j+1},\ldots,\alpha_\ell)} \qand
\DeltaD \FD_\alpha = \sum_{2\leq i\leq n} \FD_{\alpha_{\le i}}\otimes F_{\alpha_{>i}}.
\end{equation}
The second equality is equivalent to
\begin{equation}\label{eq:Coaction FDw}
\DeltaD \FD_w = \sum_{2\leq i\leq n} \FD_{\,\Dst(w[1,i])}\otimes F_{\,\st(w[i+1,n])}.
\end{equation}

Now we define a $\QSym$-action on $\QSym^D$. Let $u\in\D_m$ ($m\ge2$) and $v\in\SS_n$. Applying $\chiD$ to \eqref{eq:Action bFD} gives 
\begin{equation}\label{eq:Action FDw}
\FD_u \odotD F_v := \chiD(\bFD_u\cdot \bF_v)=\sum_{w\in u\shuffleD v} \FD_w.
\end{equation}

\begin{proposition}\label{prop:QSymD}
$(\QSym^D,\odotD,\DeltaD)$ is a graded right module and comodule over the Hopf algebra $\QSym$ isomorphic to the graded right module and comodule $(\Sigma^*(\D),\shuffleD,\unshuffleD)$ over $(\Sigma^*(\SS),\shuffle,\unshuffle)$. The map $\chiD$ induces a surjection from the graded right module and comodule $\FQSym^D$ over $\FQSym$ onto $\QSym^D$.
\end{proposition}

\begin{proof}
Compare (\ref{eq:Coaction FDw}) and (\ref{eq:Action FDw}) with the definition of $(\Sigma^*(\D),\shuffleD,\unshuffleD)$ and with \eqref{eq:Action bFD} and \eqref{eq:Coaction bFD}.
 \end{proof}

\begin{corollary}
 The embedding $\imath:\NSym^D\hookrightarrow \FQSym^D$ and the surjection $\chiD:\FQSym^D\twoheadrightarrow\QSym^D$ are dual morphisms of graded modules and comodules.
\end{corollary}

\begin{proof} 
This follows from Theorem~\ref{thm:DesD}, Proposition~\ref{prop:NSymD}, and Proposition~\ref{prop:QSymD}.
 \end{proof}

\subsubsection{Symmetric functions of type D}

We define $\Sym^D:=\chi^D(\NSym^D) \subseteq \QSym^D$. 
It has two spanning sets $\{\sD_\alpha:\alpha\modelsD n, n\ge2\}$ and $\{\hD_\alpha:\alpha\modelsD n, n\ge2\}$, where
\[ \sD_\alpha:= \chiD(\bsD_\alpha) \qand
\hD_\alpha:=\chiD(\bhD_\alpha) = \sum_{\beta\cleq\alpha} \chiD(\bsD_\beta) = \sum_{\beta\cleq\alpha} \sD_\beta. \] 
We will give a representation theoretic interpretation for $\sD_\alpha$ later. 

Suppose that $\alpha\modelsD m\ge2$ and $\beta\models n\ge0$. We define
\[ \hD_{\alpha}\odotD h_{\beta} := \chiD(\bhD_\alpha \cdot \bh_\beta) = \hD_{\alpha\cdot\beta} \qand
\DeltaD(\sD_\alpha) = \DeltaD(\chiD(\bsD_\alpha)) := (\chiD\otimes\chi)(\DeltaD(\bsD_\alpha)).\]

\begin{proposition}
$(\Sym^D,\odotD,\DeltaD)$ is a graded right module and comodule over the Hopf algebra $\Sym$ isomorphic to the graded right module and comodule $\Lambda(\D)$ over $\Lambda(\SS)$. The injection $\imath:\Sym^D\hookrightarrow\QSym^D$ and the surjection $\chi^B:\NSym^D\twoheadrightarrow\Sym^D$ are dual morphisms of graded right modules and comodules.
\end{proposition}

\begin{proof}
The result follows from Theorem~\ref{thm:DesD}, Proposition~\ref{prop:NSymD}, and Proposition~\ref{prop:QSymD}.
 \end{proof}

\subsection{Representation theory of 0-Hecke algebras of type D} 

Assume $n\ge2$ and let $(W,S)$ be the finite Coxeter system of type $D_n$, where $W=\D_n$ and $S=\{s_0=s^D_0,s_1,\ldots,s_{n-1}\}$. The 0-Hecke algebra $\HD_n(0)$ of $(W,S)$ admits two generating sets $\{\pi_0=\pi_0^D,\pi_1,\ldots,\pi_{n-1}\}$ and $\{\pib_0=\pib_0^D,\pib_1,\ldots,\pib_{n-1}\}$. One can realize $\pi_1,\ldots,\pi_{n-1}$ as signed bubble-sorting operators on $\ZZ^n$ in the same way as in type B and $\pi_0$ acts on $\ZZ^n$ by
\[ \pi_0(a_1,\ldots,a_n):=
\begin{cases}
(-a_2,-a_1,a_3,\ldots, a_n), & \text{if } a_1+a_2>0,\\
(a_1,\ldots,a_n), & \text{if } a_1+a_2\le 0.
\end{cases} \]

The projective indecomposable $\HD_n(0)$-modules and simple $\HD_n(0)$-modules are given by $\PD_\alpha:=\P_I^S$ and $\CD_\alpha:=\C_I^S$, respectively, where $I=\{s_i:i\in D(\alpha)\}$, for all $\alpha\modelsD n\ge2$. One can realize $\PD_\alpha$ as the space of \emph{type D standard tableaux of pseudo-ribbon shape} $\alpha$ with an appropriate $\HD_n(0)$-action~\cite{H0Tab}. 

The parabolic subalgebra $\HD_\alpha(0)$ of $\HD_n(0)$ is generated by $\{\pi_i: i\in\{0,1,\ldots,n-1\}\setminus D(\alpha)\}$. 
If $m\ge2$ and $n\ge0$ then one has $\HD_m(0)\otimes H_n(0)\cong \HD_{m,n}(0)$ and hence an embedding $\HD_m(0)\otimes\HD_n(0)\hookrightarrow\HD_{m+n}(0)$.

Associated with the tower $\HD_\bullet(0): \HD_2(0)\hookrightarrow \HD_3(0)\hookrightarrow\cdots$ are Grothendieck groups
\[ G_0(\HD_\bullet(0)):=\bigoplus_{n\geq2}G_0(\HD_n(0)) \qand K_0(\HD_\bullet(0)):=\bigoplus_{n\geq2}K_0(\HD_n(0)). \]
Assume $m\ge2$ and $n\ge0$.
Let $M$ and $N$ be finitely generated modules over $\HD_m(0)$ and $H_n(0)$, respectively. Define
\[ M \odotD N := (M\otimes N)\uparrow\,_{\HD_{m,n}(0)}^{\HD_{m+n}(0)} \qand
\DeltaD (M) := \sum_{2\le i\le m} M \downarrow\,_{\HD_{i,m-i}(0)}^{H_{m}(0)}. \] 
Using the linear maps $\bar\mu_I^S:G_0(H_{W_I}(0))\to G_0(H_W(0))$ and $\bar\rho_I^S:G_0(H_W(0))\to G_0(H_{W_I}(0))$ defined in \eqref{def:IndResG0}, with $S=\{s_0,\ldots,s_{m+n-1}\}$ and $I=S\setminus\{s_m\}$, one has
\[ M\odotD N = \bar\mu_I^S(M\otimes N) \qand \DeltaD_m (Q) := Q\downarrow\,_{\HD_{m,n}(0)}^{H_{m+n}(0)} = \bar\rho_I^S(Q) \]
where $Q$ is a finitely generated $\HD_{m+n}(0)$-module.
Also recall from Proposition~\ref{prop:IndResG0K0} that $\bar\mu_I^S$ and $\bar\rho_I^S$ restrict to $\mu_I^S:K_0(H_{W_I}(0))\to K_0(H_W(0))$ and $\rho_I^S:K_0(H_W(0))\to K_0(H_{W_I}(0))$.
If $M$, $N$, and $Q$ are all projective then 
\[ M \odotD N = \mu_I^S(M\otimes N) \qand \DeltaD_m(Q) = \rho_I^S(Q).\]
Define two characteristic maps below, where $\alpha\modelsD n\ge 2$: 
\[ \begin{matrix}
\mathrm{Ch}: & G_0(\HD_\bullet(0)) & \to & \QSym^D & \qand & \bch: & K_0(\HD_\bullet(0)) & \to & \NSym^D \\  
& \CD_\alpha & \mapsto & \FD_\alpha & & & \PD_\alpha & \mapsto & \bsD_\alpha.
\end{matrix}\]

\begin{theorem}
(i) $(G_0(\HD_\bullet(0)),\odotD,\DeltaD)$ is a graded right module and comodule over the graded Hopf algebra $G_0(H_\bullet(0))$.

\noindent (ii) $(K_0(\HD_\bullet(0)),\odotD,\DeltaD)$ is a graded right module and comodule over the graded Hopf algebra $K_0(H_\bullet(0))$.

\noindent (iii) $(G_0(\HB_\bullet(0)),\odotD,\DeltaD)$ is dual to $(K_0(\HD_\bullet(0)),\odotD,\DeltaD)$ via $\langle \PD_\alpha, \CD_\beta \rangle:=\delta_{\alpha,\beta}$.

\noindent(iv) Both $\Ch$ and $\bch$ are isomorphisms of graded modules and comodules.
\end{theorem}

\begin{proof}
Apply Theorem~\ref{thm:GrH0}, Theorem~\ref{thm:DesB}, Proposition~\ref{prop:NSymD}, and Proposition~\ref{prop:QSymD}.
 \end{proof}

\begin{proposition}\label{prop:ChPD}
If $\alpha$ is a pseudo-composition of $n\ge2$ then ${\rm Ch}(\,\PD_\alpha) = \sD_\alpha$. 
\end{proposition}
\begin{proof}
This follows from Proposition~\ref{prop:PtoC}.
 \end{proof} 

\begin{remark}
We obtained in Section~\ref{sec:HopfB} a self-dual graded Hopf algebra structure on $\ZZ\B$ using the embedding $\B_m\times\B_n\hookrightarrow\B_{m+n}$.
This embedding restricts to $\D_m\times\D_n\hookrightarrow\D_{m+n}$, which may be used to define dual graded algebra and coalgebra structures on $\ZZ\D$ similarly as in type B.
However, the result is not a Hopf algebra.
\end{remark}


\begin{thebibliography}{30}

\bibitem{AguiarMahajan}
M. Aguiar and S. Mahajan, Coxeter groups and Hopf algebras, Fields Institute Monographs vol. 23, AMS, Providence, RI, 2006.


\bibitem{ASS}
I. Assem, D. Simson, and A. Skowro\'nski, Elements of the representation theory of associative algebras, vol. 1: Techniques of representation theory, London Mathematical Society Student Texts, vol. 65, Cambridge University Press, Cambridge, 2006.

\bibitem{BarceloIhrig}
H. Barcelo, E. Ihrig, Lattices of parabolic subgroups in connection with hyperplane arrangements, J. Alg. Combin. 9 (1999) 5--24.

\bibitem{BergeronLi}
N. Bergeron and H. Li, Algebraic structures on Grothendieck groups of a tower of algebras, J. Algebra 321 (2009) 2068--2084.

\bibitem{BjornerBrenti}
A. Bj\"orner and F. Brenti, Combinatorics of coxeter groups, Graduate Texts in Mathematics, vol. 231, Springer, New York, 2005.

\bibitem{BjornerWachs}
A. Bj\"orner and M. Wachs, Generalized quotients in Coxeter groups,  Trans. Amer. Math. Soc. 308 (1988) 1--37. 

\bibitem{Chow}
C.-O. Chow, Noncommutative symmetric functions of type B, M.I.T. Ph.D. thesis, 2001.

\bibitem{NCSF_VI}
G. Duchamp, F. Hivert, and J.-Y. Thibon, Noncommutative symmetric functions VI: free quasi-symmetric functions and related algebras, Internat. J. Alg. Comput. 12 (2002), 671--717.

\bibitem{GeckPfeiffer}
M. Geck and G. Pfeiffer, Characters of finite Coxeter groups and Iwahori-Hecke algebras, Lond. Math. Soc. Monographs, New Series, Vol. 21, Oxford University Press, New York (2000).

\bibitem{GeissingerKinch}
L. Geissinger and D. Kinch, Representations of the hyperoctahedral groups, J. Algebra 53 (1978) 1--20.

\bibitem{GrinbergReiner}
D. Grinberg and V. Reiner, Hopf algebras in Combinatorics, arXiv:1409.8356v3.


\bibitem{H0Tab}
J. Huang, A tableau approach to the representation theory of 0-Hecke algebras, arXiv:1501.05250v3.

\bibitem{Humphreys}
J. E. Humphreys, Reflection groups and Coxeter groups, Cambridge Advanced Studies in Mathematics, no. 29, Cambridge University Press, Cambridge, 1990.

\bibitem{KrobThibon}
D. Krob and J.-Y. Thibon, Noncommutative symmetric functions IV: Quantum linear groups and Hecke algebras at $q=0$, J. Algebraic Combin. 6 (1997) 339--376.


\bibitem{Lusztig}
G. Lusztig, Hecke algebras with unequal parameters, CRM monograph series, vol.18, American Mathematical Society, 2003. 

\bibitem{MR}
C. Malvenuto and C. Reutenauer, Duality between quasi-symmetric functions and the solomon descent algebra, J. Algebra 177 (1995) 967--982.

\bibitem{Miller}
A. Miller, Reflection arrangements and ribbon representations, European J. Combin. 39 (2014) 24--56.

\bibitem{Norton}
P.N. Norton, 0-Hecke algebras, J. Austral. Math. Soc. A 27 (1979) 337--357.

\bibitem{NovelliThibon}
J.-C. Novelli and J.-Y. Thibon, Free quasi-symmetric functions and descent algebras for wreath products, and noncommutative multi-symmetric functions, Discrete Math. 310 (2010) 3584--3606.

\bibitem{Petersen}
 K. Petersen, A note on three types of quasisymmetric functions, Electron. J. Combin. 12 (2005) \#R61.
 
\bibitem{Poirier}
S. Poirier, Cycle type and descent set in wreath products, Discrete Math. 180 (1998) 315--343.

\bibitem{Reiner1}
V. Reiner, Quotients of Coxeter complexes and P-partitions, Mem. AMS. no. 460, 95(1992) 1--134.


\bibitem{Solomon}
L. Solomon, A decomposition of the group algebra of a finite Coxeter group, J. Algebra 9 (1968) 220--239.


\bibitem{EC2}
R. Stanley, Enumerative combinatorics, volume 2, Cambridge University Press 1999.

\bibitem{Stembridge}
J.R. Stembridge, A short derivation of the M\"obius function for the Bruhat order, J Algebr. Comb. 25 (2007) 141--148.

\end{thebibliography}
\end{document}